\documentclass[11pt,reqno]{amsart}
\usepackage{amsmath, amssymb, mathrsfs,amsthm,amscd,graphicx,esint}
\usepackage{dsfont}
\usepackage{caption}

\usepackage{mathtools}

\DeclarePairedDelimiter\floor{\lfloor}{\rfloor}

\usepackage{bbm}
\usepackage{url}

\usepackage[shortlabels]{enumitem}   
\usepackage[usenames,dvipsnames]{color}
\usepackage[colorlinks=true, pdfstartview=FitV, linkcolor=blue, citecolor=blue, urlcolor=blue]{hyperref}

\newtheorem{theorem}{Theorem}[section]
 \newtheorem{corollary}[theorem]{Corollary}
 \newtheorem{lemma}[theorem]{Lemma}
 
 \newtheorem{proposition}[theorem]{Proposition}

 \newtheorem{remark}[theorem]{Remark}
  
 \numberwithin{equation}{section}
 \newtheorem*{theorem*}{Theorem}
 
 \DeclarePairedDelimiterX{\ip}[2]{\langle}{\rangle}{#1, #2}
 \newcommand{\les}{\lesssim}
 
 \newcommand{\norm}[1]{\left\|#1\right\|}
\newcommand{\abs}[1]{\left|#1\right|}
\newcommand{\defeq}{\mathrel{\mathop:}=}
\newcommand{\eps}{\varepsilon}
\newcommand{\R}{\mathbb R}
\newcommand{\T}{\mathbb T}
\newcommand{\jm}{\mathcal J}

\definecolor{violet}{rgb}{0.54, 0.17, 0.89}

\newcommand{\Z}{\mathbb{Z}}

\usepackage{etoolbox}
\makeatletter
\let\ams@starttoc\@starttoc
\makeatother
\usepackage[parfill]{parskip}
\makeatletter
\let\@starttoc\ams@starttoc
\patchcmd{\@starttoc}{\makeatletter}{\makeatletter\parskip\z@}{}{}
\makeatother

\title[Onset of wave turbulence for NLS]{Onset of the wave turbulence description of the longtime behavior of the nonlinear Schr\"odinger equation}
\author{T. Buckmaster, P. Germain, Z. Hani, J. Shatah}

\begin{document}


\begin{abstract} Consider the cubic nonlinear Schr\"odinger equation set on a $d$-dimensional torus, with data whose Fourier coefficients have  phases which are uniformly distributed and independent. We show that, on average, the evolution of the moduli of the Fourier coefficients is governed by the so-called \emph{wave kinetic equation}, predicted in wave turbulence theory, on a nontrivial timescale.
\end{abstract}
\maketitle

{\centering \emph{We dedicate this manuscript to the memory of Jean Bourgain (1954 -- 2018).}\par}

\setcounter{tocdepth}{1}
\tableofcontents

\section{Introduction}

\subsection{The Kinetic Equation}
The central theme in the theory of non-equilibrium statistical physics of interacting particles is the derivation of a kinetic equation that describes the distribution of particles in phase space. The main example here is Boltzmann's kinetic theory: rather than looking at the individual trajectories of $N$-point particles following $N-$body Newtonian dynamics, Boltzmann derived a kinetic equation that described the effective dynamics of the distribution function in a certain large-particle limit (so-called the Boltzmann-Grad limit).

A parallel kinetic theory for waves, being as fundamental as particles, was proposed by physicists in the past century. Much like the Boltzmann theory, the aim is to understand the effective behavior and energy-dynamics of systems where many waves interact nonlinearly according to time-reversible dispersive or wave equations. The theory predicts that the macroscopic behavior of such nonlinear wave systems is described by a \emph{wave kinetic equation} that gives the average distribution of energy among the available wave numbers (frequencies). Of course, the shape of this kinetic equation depends directly on the particular dispersive system/PDE that describes the reversible microscopic dynamics. 

The aim of this work is to start the rigorous investigation of such passage from a reversible nonlinear dispersive PDE to an irreversible kinetic equation that describes its effective dynamics. For this, 
we consider the cubic nonlinear Schr\"odinger equations on a generic torus of size $L$ (with periodic boundary conditions) and with a parameter $\lambda>0$ quantifying the importance of nonlinear effects (or equivalently via scaling, the size of the initial datum):
\[
 \tag{NLS} \label{NLS}
\begin{cases}
i \partial_t u - \Delta_\beta u = - \lambda^{2} |u|^{2} u,  \quad x\in \mathbb{T}_L^d = [0,L]^d,   \\[.6em]
u(0,x) = u_0(x).
\end{cases}
\]
The spatial dimension is $d \geq 3$. Here, and throughout the paper, we denote 
\[
\Delta_\beta \defeq  \frac{1}{2\pi}  \sum\limits_{i=1}^d \beta_i \partial_i^2,
\]
where $\beta\defeq  (\beta_1,\dots,\beta_d)\in [1,2]^d$, and we denote $\mathbb{Z}_L^d \defeq \frac{1}{L} \mathbb{Z}^d$, the dual space to $\mathbb{T}_L^d$.  

Typically in this theory, the initial data are randomly distributed in an appropriate fashion. For us, we consider random initial data of the form
\begin{equation}\label{random data}
u_0(x) = \frac{1}{L^{d}} \sum\limits_{k \in \mathbb{Z}_L^d} \sqrt{\phi(k)}  e^{2\pi i [k \cdot x + \vartheta_k(\omega)]},
\end{equation}
for some nice (smooth and localized) deterministic function $\phi:\R^d \to [0, \infty)$. 
The phases $\vartheta_k(\omega)$  are  independent random variables,  uniformly distributed  on $[0,1]$. 
Notice that the normalization of the Fourier transform is chosen so that
\[
 \| u_0 \|_{L^2} \sim 1.
\]

Filtering by the linear group and expanding in Fourier series, we write
\begin{equation}
\label{uFourierseries}
u(t,x) = \frac{1}{L^{d}} \sum\limits_{k \in \mathbb{Z}_L^d} a_k(t)  e^{2\pi i [k \cdot x + t Q(k)]}, \quad \mbox{where} \quad Q(k) \defeq\sum\limits_{i=1}^d \beta_i (k_i)^2.
\end{equation}
The main conjecture of wave turbulence theory is that as $L \to \infty$ (large box limit) and $\frac{\lambda^2}{L^d} \to 0$ (weakly nonlinear limit), the quantity
$$
\rho^L_k(t) = \mathbb{E} |a_k(t)|^2
$$
converges to a solution of a kinetic equation. More precisely,  it is conjectured that, as $L\to \infty$, $t\to \infty$ and $\frac{\lambda^2}{L^d} \to 0$, then $\rho^L_k(t) \sim \rho(t,k)$, where $\rho: \mathbb{R}\times\mathbb{R}^d\to \mathbb{R}_+$ satisfies the  wave kinetic equation
\begin{equation*}
\label{KW} \tag{WKE}
\begin{cases}
\partial_t \rho = \frac{1}{\tau} \mathcal{T}(\rho)=  \frac{1}{\tau} 
 \int\limits_{(\mathbb{R}^{d})^{3}} \delta(\varSigma) \delta(\varOmega) \prod_{i=0}^3 \rho(k_i) \left[ \sum_{i=0}^3 \frac{(-1)^i}{\rho(k_i)}  \right] \prod_{i=0}^3 dk_i,\\
 \rho(0,k) = \phi(k).
\end{cases}
\end{equation*}
where $\tau \sim \left(\frac{L^{d}}{\lambda^2}\right)^{2}$, we introduced the convention $k_0=k$ and the notation
\[
\begin{cases}
\varSigma = \varSigma(k,k_1,\dots,k_3) = \sum_{i=0}^3 (-1)^i k_i  \\[.3em]
\varOmega= \varOmega(k,k_1,\dots,k_3) = \sum_{i=0}^3 (-1)^i Q(k_i),
\end{cases}
\]
and finally $\delta(\Sigma) \delta(\Omega)$ is to be understood in the sense of distributions: $\delta(\Sigma)$ is just the convolution integral over $k_1-k_2+k_3=k$, whereas $\delta(\Omega=0):=\lim_{\epsilon\to 0}\int \varphi(\frac{\Omega}{\epsilon}) dk_1 dk_2 dk_3$ for some $\varphi \in C_c^\infty(\R)$ with $\int \varphi=1$. Note that this is absolutely continuous to the surface measure through the formula $\delta(\Omega) = \frac{1}{|\nabla \Omega|} d\mu_\Omega$, with $d \mu_\Omega$ being the surface measures on $\{ \Omega = 0\}$.


\subsection{Background}

In the physics literature, the  wave kinetic equation~\eqref{KW} was first
derived by Peierls~\cite{Peierls} in his investigations of solid state
physics; it was discovered again by
Hasselmann~\cite{Hasselmann1,Hasselmann2} in his work on the energy
spectrum of water waves. The subject was revived and
systematically investigated by Zakharov and his collaborators~\cite{ZLF}, particularly after the discovery of special power-type stationary solutions for the kinetic equation that serve as analogs of the Kolmogorov spectra of hydrodynamic turbulence. These so-called \emph{Kolmogorov-Zakharov spectra} predict steady states of the corresponding microscopic system (possibly with forcing and dissipation at well-separated extreme scales), where the energy \emph{cascades} at a constant flux through the (intermediate) frequency scales. Nowadays, wave turbulence is a vibrant area of research in nonlinear wave theory with important practical applications in several areas including oceanography and plasma physics, to mention a few. We refer to~\cite{Nazarenko,NR} for recent reviews.

The analysis of~\eqref{KW} is full of very interesting questions,
see~\cite{EV,GIT,ST} for recent developments, but we will focus here on
the problem of its rigorous derivation. Several partial or heuristic derivations have
been put forward for~\eqref{KW}, or the closely related quantum Boltzmann
equations~\cite{BCEP1,BCEP2,BCEP3,ESY0,DSTz,Faou,LS1,MMT,Spohn2}. However, to
the best of our knowledge, there is no rigorous mathematical statement on
the derivation of~\eqref{KW} from random data. The closest attempt in this
direction is due to Lukkarinen and Spohn~\cite{LS2}, who studied the large
box limit for the discrete nonlinear Schr\"odinger equation at statistical
equilibrium (corresponding to a stationary solution to \eqref{KW}).

In preparation for such a study, one can first try to understand the large box and weakly nonlinear limit
of~\eqref{NLS} without assuming any randomness in the data. In the case
where~\eqref{NLS} is set on a rational torus, it is possible to extract a
governing equation by retaining only exact
resonances~\cite{FGH,GHT1,GHT2,BGHS}. The limiting equation is
then Hamiltonian and dictates the behavior of the microscopic system (NLS on $\T^d_L$) on the timescales $L^2/\lambda^2$ (up to a log loss for $d=2$) and for sufficiently small $\lambda$. It is worth mentioning that such a result is not possible if the equation
is set on generic tori, since most of the exact resonances are destroyed there.

Finally, we point out that there are very few instances where the derivation of kinetic equations has been done rigorously. The fundamental result of
Lanford~\cite{Lanford}, later clarified in~\cite{GST}, deals with the
$N$-body Newtonian dynamics, from which emerges, in the Grad limit, the
Boltzmann equation. This can be understood as a classical analog of the rigorous derivation on \eqref{KW}. Another instance of such success was the case of random
linear Schr\"odinger operators (Anderson's model)~\cite{Erdos,ESY1,ESY2,Spohn1}. This can be understood as a linear analog of the problem of rigorously deriving \eqref{KW}.

\subsection{The difficulties of the problem}

There are several difficulties in proving the validity of \eqref{KW} which we now enumerate:

\begin{enumerate}[(a), leftmargin =*]

\item  \label{rpa} The textbook derivation of the wave kinetic equation is done under the assumption that the independence of the data propagates for all time.  This assumption cannot be verified for any nonlinear model.  A way around this difficulty is to Taylor expand the profile $a_k$ in terms of the initial data. Such an expansion can be represented by  Feynman trees, and  permits us to utilize  the statistical independence of the data  in  computing the expected value of $|a_k|^2$.  Moreover
 one needs to control the errors in such an expansion to derive the kinetic equation \eqref{KW}.  These calculations are presented in Sections \ref{feynman} and \ref{LWP section}.

\item The wave kinetic equation induces an $O(1)$ change on its initial configuration at a timescale of $\tau$.  Thus we need to establish that  for solutions of  \eqref{NLS},  the expansion mentioned above converge up to time $\tau$. This requires a local existence result  on a timescale which is several orders of magnitude longer than what is known.  This shortcoming is a main reason why our argument cannot reach the kinetic timescale $\tau$, and we have to contend with a derivation over timescales where the kinetic equation only affects a relatively small change on the initial distribution, and as such coincides (up to negligible errors) with its first time-iterate.     

Therefore, a pressing issue is to increase the length of the time interval $[0,T]$, 
over which the Taylor expansion gives a good representation of solutions to the  nonlinear problem. For deterministic data, the best known results that give effective bounds in terms of $L$ come from our previous work \cite{BGHS} which gives a description of the solution up to times $\sim L^2/\lambda^2$ (up to a $\log L$ loss for $d=2$) and for $\lambda \ll 1$. Such timescale would be very short for our purposes.

To increase $T$, we have to rely on the randomness of the initial data. Roughly speaking, for a random field that is normalized to 1 in $L^2(\T^d_L)$, its $L^\infty$ norm can be heuristically bounded on average by $L^{-d/2}$. Therefore, regarding the nonlinearity $\lambda^{2} |u|^{2}u$ as a nonlinear potential $Vu$ with $V=\lambda^{2}|u|^{2}$ and $\|V\|_{L^\infty}\lesssim \lambda^{2}L^{-d}$, one would hope that this should get a convergent expansion on an interval $[0,T]$ provided that $T \lambda^{2}L^{-d} \ll 1$, which amounts to $T\leq \sqrt{\tau}$. This is the target in this manuscript.

The heuristic presented above can be implemented by relying on Khinchine-type improvements to the Strichartz norms of a linear solution $e^{it\Delta_\beta} u_{0}$ with random initial data $u_{0}$. Similar improvements have been used to lower the regularity threshold for well-posedness of nonlinear dispersive PDE. Here, the aim is to prolong  the  existence time and improve the Taylor approximation.
The randomness gives  us better control on  the size of the linear solution  over the interval $[0,T]$, while  an improved \emph{deterministic} Strichartz estimate for $\|e^{it\Delta_\beta}\psi\|_{L^p([0,T]\times \T^d)}$ with $\psi \in L^2(\T^d)$, allows us to maintain the random improvement for the nonlinear problem. The genericity of the $(\beta_i)$ is crucial (as was first observed in \cite{DGG}), and allows us to go beyond the  limiting $T^{1/p}$ growth that occurs on the rational torus. Unfortunately, the available estimates here (including those in \cite{DGG}) are not optimal for some ranges of the parameters $\lambda$ and $L$, which is why, in $d=3$, our result in Theorem \ref{introtheo} below falls  short of the timescale $\sqrt \tau\sim L^3/\lambda$. 



\item  \label{quasi-res}   To derive the kinetic equation in the large box limit, using the expansion for $\rho^L_k(t) = \mathbb{E} |a_k(t)|^2$,   one has to prove equidistribution theorems for the quasi-resonances over a very fine scale, i.e.,  $T^{-1}$. Since $T$ could be $\gg L^2$, such scales are much finer than the any equidistribution scale on the rational torus. Again, here the genericity of the $(\beta_i)$ is crucial. For this we use and extend a recent result of Bourgain on pair correlation for irrational quadratic forms \cite{Bourgain}.

\end{enumerate}


\subsection{The main result} Precise statements of our results in arbitrary dimensions $d\geq 3$ will be given in Section \ref{Precise results}. Those statements depend on several parameters coming from equidistribution of lattice points and Strichartz estimates.  For the purposes of this introductory section, we present a less general theorem without the explicit appearance of these parameters. 

\begin{theorem}\label{introtheo} Consider the cubic \eqref{NLS} on the three-dimensional torus $\T^3_L$. Assume that the initial data are chosen randomly as in \eqref{random data}. There exists $\delta>0$ such that the following holds for $L$ sufficiently large and $L^{-A}\leq \lambda \leq L^B$ (for positive $A$ and $B$):
\begin{equation}\label{theo1eq}
\mathbb{E} |a_k(t)|^2 = \phi(k)+\frac{t}{\tau} \mathcal{T}(\phi)(k) +  O_{\ell^\infty } \left( L^{-\delta} \frac{t}{\tau} \right), \quad L^{\delta} \leq t \leq T,
\end{equation}
where $\tau=\frac{1}{2}\left(\frac{L^{3}}{\lambda^2}\right)^{2}$ and $T\sim \frac{L^{3-\gamma}}{\lambda^2}$, for some $0<\gamma<1$ stated explicitly in Theorem \ref{cubic theorem}.
\end{theorem}

We note that the right-hand side of \eqref{theo1eq} is nothing but the first time-iterate of the wave kinetic equation \eqref{KW} with initial data $\phi$ (cf.\ \eqref{random data}) which coincides (up to the error term in \eqref{theo1eq}) with the exact solution of the \eqref{KW} over long times scales, but shorter than the kinetic timescale $\tau$.

The proof this theorem can be split into three   components:

\begin{enumerate}[(1), leftmargin =*]
\item { Section \ref{feynman}: Feynman tree representation.}  In this section  we derive the Taylor expansion of the nonlinear solution in terms of the initial data. Roughly speaking, we write the Fourier modes of the nonlinear solution $a_k(t)$ (see \eqref{uFourierseries}) as follows:
\[
a_ k(t) = \sum\limits_{n = 0}^N \jm_n(t, k)(\boldsymbol{a}^{(0)}) +R_{N+1}(t,  k)(\boldsymbol{a}^{(t)}),
\]
where $\jm_n$ are sums of monomials of degree $2n+1$ in the initial data $\boldsymbol{a}^{(0)}$,
and $R_N$ is the remainder which depends on the nonlinear solution $\boldsymbol{a}^{(t)}$. Each term of $\jm_n$ can be represented by a Feynman tree which makes the calculations of 
 $\mathbb E( \jm_n \jm_{n'})$ more transparent.  Such terms appear  in the expansion of $\mathbb E |a_k|^2$.  The estimates in this section rely on essentially sharp bounds on quasi-resonant sums of the form
\begin{equation}\label{samplesum}
\sum_{\substack{\vec k \in \Z_L^{rd}\\ }} \mathds{1} (|\vec k|\lesssim 1) \mathds{1}(|\mathcal Q(k)|\sim 2^{-A})\lesssim 2^{-A}L^{rd}
\end{equation}
where $\mathds1(S)$ denotes the characteristic function of a set $S$ and $\mathcal Q$ is an irrational quadratic form. Since $A$ will be taken large  $2^A\sim T \gg L^2$, such estimates belong to the realm of number theory and will be a consequence the third  component of this work.  

The bounds we obtain for such interaction are good up to times of order  $\sqrt \tau$ which is sufficient given the restrictions on the time interval of convergence imposed by the second component below.  

\item Section \ref{LWP section}:  Construction  of solutions. In this section  we construct  solutions on a time interval $[0, T]$ via a contraction mapping argument. To maximize $T$ while maintaining a contraction, we rely on the Khinchine improvement to the space-time Strichartz bounds,
as well as  the long-time Strichartz estimates on generic irrational tori proved in \cite{DGG}. It is here that our estimates are  very far from optimal, since there is no proof  to the  conjectured optimal  Strichartz estimates.

\item Section \ref{sectnumbtheo}: Equidistribution of irrational quadratic forms.  The purpose of this section is two-fold.  The first is proving bounds on quasi-resonant sums like those in \eqref{samplesum} for the largest possible $T$, and the second is to extract the exact asymptotic, with effective error bounds,  of the leading part of the sum.  It is this leading part 
that converges to the kinetic equation collision kernel as $L\to \infty$. 

Here we remark, that if $\mathcal Q$ is a rational form, then the largest $A$ for which one could hope for an estimate like \eqref{samplesum} is $2^A\sim L^2$ which reflects the fact that a rational quadratic form cannot be equidistributed at scales smaller than $L^{-2}$ (at the level of NLS, it would yield a time interval restriction of $T\lesssim L^2$ for the rational torus). However, for generic irrational quadratic forms, $\mathcal Q$ is actually equidistributed at much finer scales than $L^{-2}$. Here, we adapt a recent work of Bourgain \cite{Bourgain} which will allow us to reach equidistribution scales essentially up to $L^{-d}$.
\end{enumerate}

\subsection{Notations} In addition to the notation introduced earlier for $\mathbb{T}_L^d = [0,L]^d$ and $\mathbb{Z}^d_L = \frac{1}{L} \mathbb{Z}^d$, we use standard notations.   A function $f$ on $\mathbb{T}_L^d$ and its Fourier transform $\widehat{f}$ on $\mathbb{Z}^d_L$ are related by
$$
f (x) = \frac{1}{L^d} \sum\limits_{\mathbb{Z}^d_L} \widehat{f}_k  e^{ 2 \pi i k\cdot x} \qquad \mbox{and} \qquad \widehat{f}_k = \int\limits_{\mathbb{T}^d_L}  e^{ -2 \pi i k\cdot x}f(x) \, dx.
$$
Parseval's  theorem becomes
\[
\| f \|^2_{L^2(\mathbb{T}^d_L)} = \| \widehat{f} \|^2_{\ell_L^2 (\mathbb{Z}^d_L)} = \frac{1}{L^d} \sum\limits_{k \in \mathbb{Z}^d_L} |\widehat{f}_k|^2 .
\]
We adopt the following definition for weighted $\ell^p$ spaces: if $p\geq 1$, $s\in \mathbb{R}$, and $b\in \ell^p$,
$$
\| b \|_{\ell^{p,s}_L(\mathbb{Z}_L^d)} = \left[ \frac{1}{L^d}\sum\limits_{k \in \mathbb{Z}^d_L} ( \langle k \rangle^{s} |b_k|)^p \right]^{1/p}.
$$

Sobolev spaces $H^s(\mathbb{T}^d)$ are then defined naturally by
$$
\| f \|_{H^s(\mathbb{T}^d)} = \| \langle k \rangle^s \widehat{f} \|_{\ell^{2,s}(\mathbb{Z}_L^d)}.
$$
For functions defined on $\mathbb{R}^d$, we adopt the normalization
\[f (x) = \int\limits_{\mathbb R^d}  e^{ 2 \pi i \xi \cdot x}  \widehat{f}(\xi) \,d\xi \qquad \mbox{and} \qquad \widehat{f}(\xi) = \int\limits_{\mathbb{R}^d}  e^{ -2 \pi i k\cdot x}f(x) \, dx.\]

We denote by  $C$  any constant whose value does not depend on $\lambda$ or $L$.
The notation $A \lesssim B$ means that there exists a constant $C$ such that $A \leq C B$. We also write
$
A \lesssim L^{r^+} B
$,
if for any $\epsilon > 0$ there exists $C_\epsilon$ such that $A \leq C_\epsilon L^{r+\epsilon} B$. Similarly  $A \gtrsim L^{r^-} B$,
if for any $\epsilon > 0$ there exists $C_\epsilon$ such that $A \geq C_\epsilon L^{r-\epsilon} B$. Finally we use the notation $u =O_X(B)$ to mean $\|u\|_X\lesssim B$.

\emph{We would like to thank Peter Sarnak for pointing us to unpublished work by Bourgain \cite{Bourgain}. This reference helped us improve an earlier version of our work. We also would like to thank Peter and Simon Myerson for many helpful and illuminating discussions.}

\section{The general result} \label{Precise results}

We start by writing  the equations for the interaction representation $(a_k(t))_{k \in \Z^d_L}$,  given in \eqref{uFourierseries}:
\begin{equation}\label{akeqn}
\begin{cases}
 i \dot{a_k} = - \left(\frac{\lambda}{L^{d}}\right)^{2} \sum\limits_{\substack{(k_1,\ldots, k_{3}) \in (\mathbb{Z}^d_L)^3 \\ k - k_1 + k_2 -k_3 = 0}} a_{k_1}\overline{a_{k_2}}  a_{k_3} e^{-2\pi it \varOmega(k,k_1,k_2,k_3)} \\[3em]
a_k(0) = a_k^{0} = \sqrt{\phi(k)} e^{i \vartheta_k(\omega)},
\end{cases}
\end{equation}
where we recall 
$
\varOmega(k,k_1,k_2,k_3) = Q(k) - Q(k_1) + Q(k_2)- Q(k_3), 
$
 and $\vartheta_k(\omega)$ are i.i.d.\ random variables that are uniformly distributed in $[0,2\pi]$.
Our results depend on two parameters: the equidistribution parameter $\nu$, and a Strichartz parameter $\theta_p$, which we now explain.

\subsection{The Equidistribution parameter $\nu$}

The interaction frequency  $\varOmega (k,k_1,k_2, k_3)$ above is an irrational quadratic form. Such quadratic forms can be equidistributed at scales that are much smaller than the finest scale $\sim L^{-2}$ of rational forms. 

We will denote by $\nu$ the largest real number such that  for all $k\in\mathbb{Z}^d_L$,  $|k|\le 1$, and  $\epsilon>0$, there exists $\delta >0$ such that, for $|a|, |b|<1$ with $b-a \geq L^{-\nu^-}$,
\[
\sum\limits_{\substack{a\leq \varOmega(k,k_1,k_2,k_3) \leq b\\ \abs{k_1},\abs{k_2} , \abs{k_3}\leq 1 \\ k-k_1+k_2-k_3 = 0}} 1 = (1+O(L^{-\delta})) L^{2d}\hskip -2mm \int\limits_{|k_1|, |k_2|, |k_3| \leq 1} 
\hskip -2mm \mathds{1}_{a\leq \varOmega(k,k_1,k_2,k_3) \leq b} \delta(k-k_1+k_2-k_3)\,dk_1 \, dk_2 \, dk_3.
\]

\begin{proposition} With the above definition for $\nu$, we have
\begin{itemize}
\item[(i)] If $\beta_i = 1$ for all $i \in \{1,\dots,d\}$, $\nu= 2$.
\item[(ii)] If the $\beta_i$ are generic, $\nu = d$.
\end{itemize}
\end{proposition}

\begin{proof} The first assertion  is classical, e.g., see~\cite{BGHS}. The second assertion is proved in Section~\ref{sectnumbtheo}.
\end{proof}

\subsection{The Strichartz parameter $\boldsymbol{\theta_d}$}
Our proof relies  on long-time Strichartz estimates,  which are used to maintain  linear bounds for  the nonlinear problem. The genericity  of the $\beta$'s  gives crucial improvements from the rational case. The improved estimates for generic $\beta$'s were proved in  \cite{DGG},
\[
 \|e^{it\Delta_{\beta}} P_N \psi\|_{L^p_{t,x}([0,T]\times \mathbb{T}^d)} \lesssim
N^{\frac{d}{2} - \frac{d+2}{p}}  \left(1 +  \frac{T}{N^{\gamma(d,p)}}\right)^{1/p} \|\psi\|_{L^2(\mathbb{T}^d)} 
\]
for some $0\le \gamma(d,p) \le d-2$. The $N^{\gamma}$ term can be thought of as the time it takes for a focused wave with localized  wave number $\leq N$, to focus again.  For the rational torus $\gamma=0$. 

Here we only need to use the  ${L^4_{t,x}([0,T]\times \mathbb{T}_L^d)}$ norm, and therefore we introduce  a parameter $\theta_d$ to record how the constant in the ${L^4_{t,x}([0,T]\times \mathbb{T}_L^d)}$  estimates depends on $L$.  By scaling,  the result in  \cite{DGG}  translates into,
\begin{equation}\label{impStrich}
 \|e^{it\Delta_{\beta}} P_{k\le 1} \psi\|_{L^4_{t,x}([0,T]\times \mathbb{T}_L^d)} \lesssim
L^{0^+}  \left(1 +  \frac{T}{L^{\theta_d}}\right)^{1/4} \|\psi\|_{L^2(\mathbb{T}_L^d)} 
\end{equation}
where $\theta_d:=\begin{cases} 
\frac{4}{13}+2, \qquad d=3\\
\frac{(d-2)^2}{2(d-1)}+ 2, \qquad d\geq 4.
		\end{cases}$

\subsection{The approximation theorem}  With these parameters defined, we  state  the approximation theorem for the cubic NLS in  dimension $d\geq 3$ and generic $\beta$'s.

\begin{theorem}\label{cubic theorem}
Assume the $\beta$'s are generic and  $d\geq 3$. Let $\phi_0: \mathbb R^d \to \mathbb R^+$, a rapidly decaying smooth function. Suppose that $a_k(0)=\sqrt{\phi(k)}e^{i\vartheta_k(\omega)}$ where $\vartheta_k(\omega)$ are i.i.d.\ random variables uniformly distributed in $[0,2\pi]$. For every $\epsilon_0$, a  sufficiently small constant, and $L>L_*(\epsilon_0)$ sufficiently large,  the following holds:

 There exists a set $E_{\epsilon_0, L}$ of measure $\mathbb P(E_{\epsilon_0, L})\geq 1-e^{-L^{\epsilon_0}}$ such that: if $\omega\in E_{\epsilon_0,L}$ ,  then for any $L>L_*(\epsilon_0)$, the solution $a_k(t)$ of~\eqref{NLS}  exists in $C_tH^s([0, T] \times \mathbb T^d_L)$ for 
$$
T\sim 
\begin{cases}
\lambda^{-2}L^{\frac{d+ \theta_d}{2}-4\epsilon_0} \qquad\quad \text{if} \qquad 
 L^{\frac{-d+ \theta_d }{4} }\lesssim \lambda \lesssim L^{\frac{d- \theta_d }{4}-2\epsilon_0}\, 
,\\
 \lambda^{-4}L^{d-8\epsilon_0} \qquad \qquad \text{if} \qquad \lambda \geq L^{\frac{d-\theta_d}{4}-2\epsilon_0}.
\end{cases}
$$

Moreover,
$$
\mathbb{E} \left[  |a_k(t)|^2 \mathds{1}_{E_{\epsilon_0,L}} \right] = \phi(k)+\frac{t}{\tau} \mathcal{T}_3(\phi)(k) +  O_{\ell^\infty } \left( L^{-\epsilon_0} \frac{t}{\tau} \right), \quad L^{\epsilon_0} \leq t \leq T, \text{ and }\tau=\frac{L^{2d}}{2\lambda^4}.
$$

For $d=3, 4$, the solutions exist globally in time \cite{Bourgain1993, IoPau}, and one has the same estimate without multiplying with $\mathds{1}_{E_{\epsilon_0}}$ inside the expectation.

\end{theorem}

Here we note that
the error could be controlled in a much stronger norm than $\ell^\infty$, and that
other randomizations of the data are possible (complex Gaussians for instance) without any significant changes in the proof.
\section{Formal derivation of the kinetic equation}\label{formalsection}

In this section, we present the formal derivation of the kinetic equation, whose basic steps we shall follow in the proof. The starting point is  equation \eqref{akeqn}   integrated in time,

%
%
%
\begin{equation}\label{e:induct_exp}
a_k(t) = a_k^0  + \frac{i\lambda^2}{L^{2d}}\int\limits_0^t \sum\limits_{\substack{(k_1,k_2,k_3) \in (\mathbb{Z}^d_L)^3 \\ k - k_1 + k_2 - k_3 = 0}} a_{k_1} \overline{a_{k_2}} a_{k_3} e^{- 2\pi is \varOmega(k,k_1,k_2,k_3)}\, ds
\end{equation}

The derivation of the kinetic equation proceeds as follows:

\underline{Step 1: expanding in the data.} Noting the symmetry in \eqref{e:induct_exp} in the variables $k_1$ and $k_3$, we have upon integrating by parts twice, and substituting   \eqref{akeqn}  for $\dot a_k$, 
\begin{subequations}
\begin{align}
\label{penguin1} a_k(t) = & a_k^0 \\
\label{penguin2}
& + \frac{\lambda^2}{L^{2d}} \sum\limits_{k - k_1 + k_2 - k_3=0} a_{k_1}^0 \overline{a_{k_2}^0} a_{k_3}^0 \frac{1-e^{-2\pi i t\varOmega(k,k_1,k_2,k_3)}}{2\pi  \varOmega(k,k_1,k_2,k_3)} \\
\label{penguin3}
&  \nonumber+ 2 \frac{\lambda^4}{L^{4d}} \sum\limits_{\substack{k-k_1+k_2-k_3 = 0 \\ k_1-k_4+k_5-k_6 = 0}} a_{k_4}^0 \overline{a_{k_5}^0} a_{k_6}^0 \overline{a_{k_2}^0} a_{k_3}^0 \frac{1}{2\pi   \varOmega(k,k_1,k_2,k_3)} \\
&\left[ \frac{e^{-2\pi it\varOmega(k,k_4,k_5,k_6,k_2,k_3)}-1}{2\pi \varOmega(k,k_4,k_5,k_6,k_2,k_3)} - \frac{e^{-2\pi it  \varOmega(k_1,k_4,k_5,k_6)} - 1}{2\pi \varOmega(k_1,k_4,k_5,k_6)} \right] \\
\label{penguin4}
& \nonumber  + \frac{\lambda^4}{L^{4d}}  \sum\limits_{\substack{k-k_1+k_2-k_3 = 0 \\ k_2-k_4+k_5-k_6 = 0}} a_{k_1}^0 \overline{a_{k_4}^0} a_{k_5}^0 \overline{a_{k_6}^0}   a_{k_3}^0  \frac{1}{2\pi   \varOmega(k,k_1,k_2,k_3)} \\
& \left[ \frac{e^{-2\pi it\varOmega(k,k_1,k_4,k_5,k_6,k_3)}-1}{2\pi  \varOmega(k,k_1,k_4,k_5,k_6,k_3)} - \frac{e^{-2\pi it  \varOmega(k_2,k_4,k_5,k_6)} - 1}{2\pi  \varOmega(k_2,k_4,k_5,k_6)} \right] \\
\label{penguin5}
& + \{ \mbox{higher order terms} \}.
\end{align}
\end{subequations}
where we denoted $\varOmega(k,k_1,k_2,k_3,k_4,k_5) = Q(k) + \sum_{i=1}^5(-1)^iQ(k_i)$; we also used the convention that, if $a=0$, $\frac{e^{2\pi i ta} -1}{2\pi a} = it$, while, if $a=b=0$, $\frac{1}{2\pi a}\left( \frac{e^{2\pi i t(a+b)}-1}{2\pi(a+b)} - \frac{e^{2\pi i t a}-1}{2\pi a} \right) = -\frac{1}{2} t^2 $.

\underline {Step 2: parity pairing.} We now compute $\mathbb{E} |a_k|^2$, where the expectation $\mathbb{E}$ is understood with respect to the random phases (random parameter $\omega$). The key observation is,
$$
\mathbb{E} (a^0_{k_1} \dots a_{k_s}^0 \overline{a_{\ell_1}^0 \dots a_{\ell_s}^0} ) =
\left\{
\begin{array}{ll}
\phi_{k_1} \dots \phi_{k_s} & \mbox{if there exists a $\gamma$ such that $k_{\gamma(i)} = \ell_i$} \\
0 & \mbox{otherwise}.
\end{array}
\right.
$$

(for $k \in \mathbb{Z}^d_L$, we write $\phi_k = \phi(k)$). Computing $\mathbb{E}\left( |a_k|^2\right)$ with the help of the above formula, we see that, there are no terms of order $\lambda^2$. There are two kinds of terms of
order $\lambda^4$ obtained as follows: either by pairing the term of order $\lambda^2$, namely~\eqref{penguin2}, with its conjugate, or by pairing one of the terms of order $\lambda^4$,~\eqref{penguin3} or~\eqref{penguin4}, with the term of order $1$, namely $a_k^0$. Overall, this leads to
\begin{align*}
\mathbb{E} |a_k|^2(t) & = \phi_k + \frac{2 \lambda^4}{L^{4d}}\hskip -2mm \sum\limits_{k - k_1 + k_2 - k_3=0} \hskip -4mm\phi_k \phi_{k_1} \phi_{k_2} \phi_{k_3} \Big[ \frac{1}{\phi_k} - \frac{1}{\phi_{k_1}} + \frac{1}{\phi_{k_2}} - \frac{1}{\phi_{k_3}} \Big] \Big| \frac{\sin(t \pi \varOmega(k,k_1,k_2,k_3))}{\pi \varOmega(k,k_1,k_2,k_3)} \Big|^2 \\
& \qquad + \{ \mbox{higher order terms} \} + \{ \mbox{degenerate cases} \},
\end{align*}
where  degenerate cases  occur for instance if $k$, $k_1$, $k_2$, $k_3$ are not distinct\footnote{Degenerate cases, like higher order terms, have smaller order of magnitude, on the timescales we consider as will be illustrated in Section \ref{feynman}.}.
The details of the computation are as follows:
\begin{enumerate}[(a), leftmargin =*]
\item Consider first $\mathbb{E} |\eqref{penguin2}|^2 = \mathbb{E} \eqref{penguin2} \overline{\eqref{penguin2}}$, and denote $k_1,k_2,k_3$ the indices in \eqref{penguin2} and  $k_1',k_2',k_3'$ the indices in $\overline{\eqref{penguin2}}$. There are two possibilities:
\begin{itemize}[leftmargin =*]
\item $ \{ k_1,k_3 \} = \{ k_1', k_3' \}$, in which case  $k_2=k_2'$, and $\varOmega(k,k_1,k_2,k_3) = \varOmega(k,k_1',k_2',k_3')$. 
\item  $(k_2 = k_1 \; \mbox{or}\; k_3) \; \mbox{and}\;(k_2'  = k_1' \; \mbox{or} \; k_3' )$, in which case $\varOmega(k,k_1,k_2,k_3) = \varOmega(k,k_1',k_2',k_3') = 0$. 
\end{itemize}

Overall, we find, neglecting degenerate cases (which occur  if $k$, $k_1$, $k_2$, $k_3$ are not distinct),
$$
\mathbb{E} |\eqref{penguin2}|^2 = \frac{2 \lambda^4}{L^{4d}} \sum\limits_{k - k_1 + k_2 - k_3=0}  \phi_{k_1} \phi_{k_2} \phi_{k_3} \left| \frac{\sin(\pi t\varOmega(k,k_1,k_2,k_3)) }{\pi \varOmega(k,k_1,k_2,k_3)} \right|^2  + \frac{4 \lambda^4}{L^{4d}} t^2 \sum\limits_{k_1,k_3} \phi_k \phi_{k_1} \phi_{k_2}.
$$

\item Consider next the pairing of $a_k^0$ with \eqref{penguin3}, which contributes $2 \mathbb{E} \mathfrak{Re} \left[ \eqref{penguin3} \overline{a_k^0}\right]$. The possible pairings are
\begin{itemize}
\item $\{k,k_2 \} = \{k_4,k_6 \}$, implying $k_3=k_5$, and leading to 
$\varOmega(k_1,k_4,k_5,k_6) = - \varOmega(k,k_1,k_2,k_3)$, and $\varOmega(k,k_4,k_5,k_6,k_2,k_1)=0$. 

\item $(k_3 = k_2 \; \mbox{or}\; k) \; \mbox{and}\; (k_5  = k_4 \; \mbox{or} \; k_6)$ in which case $\varOmega(k,k_1,k_2,k_3) = \varOmega(k_1,k_4,k_5,k_6) = 0$.
\end{itemize}
This gives, neglecting degenerate cases,
\[
\begin{split}
&2 \mathbb{E}  \mathfrak{Re} \left[ \overline{a_k^0} \eqref{penguin3}\right]  =\frac{8 \lambda^4}{L^{4d}} \times \\
&  \sum\limits_{k - k_1 + k_2 - k_3=0}   \phi_k  \phi_{k_2} \phi_{k_3}  \mathfrak{Re} \left[ \frac{e^{-2\pi it\varOmega(k,k_1,k_2,k_3)} - 1}{4\pi^2 \varOmega(k,k_1,k_2,k_3)^2} \right] - \frac{8 \lambda^4}{L^{4d}} t^2 \sum\limits_{k_1,k_3} \phi_k \phi_{k_2} \phi_{k_3} \\
 & = - \frac{2 \lambda^4}{L^{4d}}  \sum\limits_{k - k_1 + k_2 - k_3=0}   \phi_k \phi_{k_1} \phi_{k_2} \phi_k  \left[ \frac{1}{\phi_{k_1}} + \frac{1}{\phi_{k_3}} \right]  \left| \frac{\sin(\pi t\varOmega(k,k_1,k_2,k_3))}{\pi \varOmega(k,k_1,k_2,k_3)} \right|^2 - \frac{8 \lambda^4}{L^{4d}} t^2 \sum\limits_{k_1,k_3} \phi_k \phi_{k_2} \phi_{k_3},
\end{split}
\]
where we used in the last line the symmetry between the variables $k_1$ and $k_3$, as well as the identity $\mathfrak{Re} (e^{iy} - 1 ) = - 2 |\sin(y/2)|^2$, for $y \in \mathbb{R}$.
\item Finally, the pairing of $a_k^0$ with \eqref{penguin4} can be discussed similarly, to yield
$$
2 \mathbb{E}  \mathfrak{Re} \left[ \overline{a_k^0} \eqref{penguin4}\right] = \frac{2 \lambda^4}{L^{4d}}  \sum\limits_{k - k_1 + k_2 - k_3=0}   \phi_k \phi_{k_1} \phi_{k_3}  \left| \frac{\sin(\pi t\varOmega(k,k_1,k_2,k_3))}{\pi \varOmega(k,k_1,k_2,k_3)} \right|^2 + \frac{4 \lambda^4}{L^{4d}} t^2 \sum\limits_{k_1,k_3} \phi_k \phi_{k_2} \phi_{k_3},
$$
\end{enumerate}
Summing the above expressions for $\mathbb{E} |\eqref{penguin2}|^2$, $2 \mathbb{E}  \mathfrak{Re} \left[ \overline{a_k^0} \eqref{penguin3} \right]$ and $2 \mathbb{E}  \mathfrak{Re} \left[ \overline{a_k^0} \eqref{penguin4}\right]$ gives the desired result.

\underline{Step 3: the big box limit $L \to \infty$.} Assuming that $\varOmega(k,k_1,k_2,k_3)$ is equidistributed on a scale 
\begin{equation}
\label{bluebird}
L^{-\nu} \ll \frac{1}{t},
\end{equation}
we see that, as $L \to \infty$,
\begin{align*}
& \sum\limits_{k - k_1 + k_2 - k_3=0} \phi_k \phi_{k_1} \phi_{k_2} \phi_{k_3}  \left[ \frac{1}{\phi_k} - \frac{1}{\phi_{k_1}} + \frac{1}{\phi_{k_2}} - \frac{1}{\phi_{k_3}} \right]
\left| \frac{\sin(\pi t\varOmega(k,k_1,k_2,k_3))}{\pi \varOmega(k,k_1,k_2,k_3)} \right|^2 \sim \\
&   L^{2d} \int\limits \delta(\varSigma) \phi_k \phi_{k_1} \phi_{k_2} \phi_{k_3} \left[ \frac{1}{\phi_k} - \frac{1}{\phi_{k_1}} + \frac{1}{\phi_{k_2}} - \frac{1}{\phi_{k_3}} \right] \left| \frac{\sin(\pi t\varOmega(k,k_1,k_2,k_3))}{\pi \varOmega(k,k_1,k_2,k_3)} \right|^2 \,dk_1 \, dk_2 \,dk_3.
\end{align*}

 \underline{Step 4: the large time limit $t \to \infty$} Observe  that\footnote{This follows from Plancherel's theorem, and the fact that the Fourier transform of $\frac{1}{\pi} \frac{\sin x}{x}$ is the characteristic function of $[-\frac{1}{2\pi},\frac{1}{2\pi}]$.}  $\int\limits \frac{(\sin x)^2}{x^2} \,dx = \pi^2$, so that, in the sense of distributions,
$$
 \left| \frac{\sin(\pi t\varOmega)}{\pi \varOmega} \right|^2 \sim t \delta(\varOmega) \qquad \mbox{as $t \to \infty$}.
$$
Therefore, as $t \to \infty$,
\begin{align*}
& \sum\limits_{k - k_1 + k_2 - k_3=0} \phi_k \phi_{k_1} \phi_{k_2} \phi_{k_3}  \left[ \frac{1}{\phi_k} - \frac{1}{\phi_{k_1}} + \frac{1}{\phi_{k_2}} - \frac{1}{\phi_{k_3}} \right]
\left| \frac{\sin(\pi t\varOmega(k,k_1,k_2,k_3))}{\pi \varOmega(k,k_1,k_2,k_3)} \right|^2 \\
& \qquad \qquad \sim t L^{2d} \int\limits  \delta(\varSigma) \delta(\varOmega) \phi(k) \phi(k_1) \phi(k_2) \phi(k_3) \left[ \frac{1}{\phi(k)} - \frac{1}{\phi(k_1)} + \frac{1}{\phi(k_2)} - \frac{1}{\phi(k_3)} \right]  \,dk_1 \, dk_2 \,dk_3 \\
& \qquad \qquad = t L^{2d} \mathcal{T}(\phi,\phi,\phi).
\end{align*}


\noindent \underline{Conclusion: relevant timescales for the problem.} Overall, we find, assuming that the above limits are justified
\begin{equation}
\label{grackle}
\mathbb{E} |a_k|^2(t) = \phi_k + 2 \frac{\lambda^4}{ L^{2d}} t \mathcal{T}(\phi,\phi,\phi) + \{ \mbox{lower order terms} \}.
\end{equation}
This suggests that the actual timescale of the problem is
$$
\tau= \frac{ L^{2d}}{2 \lambda^4},
$$
and that, setting $s = \frac{t}{\tau}$, the governing equation should read
\begin{equation}
\label{kineticwave}
\partial_s \phi = \mathcal{T}(\phi,\phi,\phi)
\end{equation}
In which regime is this approximation expected? Let $T$ be the timescale over which we consider the equation.
\begin{itemize}
\item In order for~\eqref{grackle} to hold, the condition~\eqref{bluebird} has to hold, and the limits $L\to \infty$ and $T \to \infty$ have to be taken: one needs
$$
T \ll L^{\nu}, \qquad L \gg1, \qquad \mbox{and} \qquad T\gg1.
$$
\item In order for the nonlinear evolution of~\eqref{kineticwave} to affect an $O(\kappa)$ change on the initial data, the two conditions above should be satisfied; in addition $T$ should be of the order of $\kappa \tau$ (equivalently $s\sim \kappa$). Thus we find the conditions
$$
1\ll T \approx \kappa \tau \ll L^{\nu} \qquad \mbox{and} \qquad  \kappa^{\frac 14}L^{d/2} \gg \lambda \gg  \kappa^{\frac 14}L^{d/2 - \nu/4}.
$$
\end{itemize}

\section{Feynman trees: bounding the  terms in the expansion}\label{feynman}

Since we are considering the problem with rapidly decaying  $\phi$, then the rapid decay of $\phi$ yields all the bounds one needs for wave numbers $|k|\ge L^{0^+}$, thus we might as well consider $\phi$ to be compactly supported.

\subsection{Expansion of the solution in the data}

We follow mostly the notations in Lukkarinen-Spohn~\cite{LS2}, Section 3 (see also \cite{Christ07}).

The iterates of $\phi$, considered in the previous section,   can be represented through trees (at least up to lower order error terms). To explain these trees, let us start with the equation satisfied by the amplitude of the wave number $a_ k$
\[
\begin{split}
&a_k(t) = a_k^0  + \frac{i\lambda^2}{L^{2d}} \int\limits_0^t\sum\limits_{\substack{(k_1,k_2,k_3) \in (\mathbb{Z}^d_L)^3 \\ k - k_1 + k_2 - k_3 = 0}} a_{k_1} \overline{a}_{k_2} a_{k_3} e^{- 2\pi is \varOmega(k,k_1,k_2,k_3)}\, ds,\\
&a_k(t) = a_k^0  + \frac{i\lambda^2}{L^{2d}} \int\limits_0^t\mathscr{P}_3(a)(s) e^{- 2\pi is \varOmega}\, ds.
\end{split}
\]
where the subscript in $\mathscr{P}_3$ is to indicate that it is a monomial of degree $3$, and where we suppressed the $k$ dependence for convenience.   The expansion can be obtained by integrating by parts on the oscillating factor $ e^{- 2\pi is \varOmega}$.  Thus the first integration by parts gives the cubic expansion,
\[
a_k(t) = a_k^0  +   \frac{i\lambda^2}{L^{2d}}\mathscr{P}_3(a)(0)F_0^t +  \frac{i\lambda^2}{L^{2d}} \int\limits_0^t \dot{\mathscr{P}}_3(a)(s) F_s^t\, ds,\quad F_s^t := \int\limits_s^t e^{- 2\pi i\tau  \varOmega} d\tau\, . 
\]
Using the equation for $a$, we see that $\dot{\mathscr{P}}_3(a)$ consists of three monomials of degree $5$, and if we denote on of them by 
$\mathscr{P}_5$, then the integral term consists of three integrals of the type,
\[
\left(\frac{i\lambda^2}{L^{2d}}\right)^2 \int\limits_0^t\mathscr{P}_5(a)(s) e^{- 2\pi is \varOmega} F_s^t\, ds.
\]
Another integration by parts gives the quintic expansion, which consist of three terms of the form
\[
\left(\frac{i\lambda^2}{L^{2d}}\right)^2 \mathscr{P}_5(a)(0) G_0^t + \left(\frac{i\lambda^2}{L^{2d}}\right)^2 \int\limits_0^t\dot{\mathscr{P}}_5(a)(s) G_s^t\, ds, \quad G_s^t = \int\limits_s^t e^{- 2\pi i\tau \varOmega} F_\tau^t\ d\tau\, .
\]
Consequently, to compute the expansion to order $N$ we need to integrate by parts $N$ times on the oscillating exponentials, giving the expansion,
\begin{equation}\label{iteration}
\begin{split}
a_ k(t) = \sum\limits_{n = 0}^N \jm_n(t, k)(\boldsymbol{a}^{(0)}) +R_{N+1}(t,  k)(\boldsymbol{a}^{(t)}),\\
\end{split}
\end{equation}
where $\jm_n = \sum\limits_{\boldsymbol \ell} \jm_{n,\boldsymbol \ell}$, and each $\jm_{n,\boldsymbol \ell}$ is a monomial of degree $2n+1$ generated by the $n^{th}$ integration by parts. The index  $\boldsymbol \ell$ is a vector whose entries keep track  of the history of how  the monomial  $\jm_{n,\boldsymbol \ell}$ was generated.  $R_{N+1}$ is the remaining time integral.

Each $\jm_{n,\boldsymbol \ell}$ can be represented by a tree similar to  Figure \ref{fig:feynman} below.
\begin{figure}[h!]
  \includegraphics[width=0.7\linewidth]{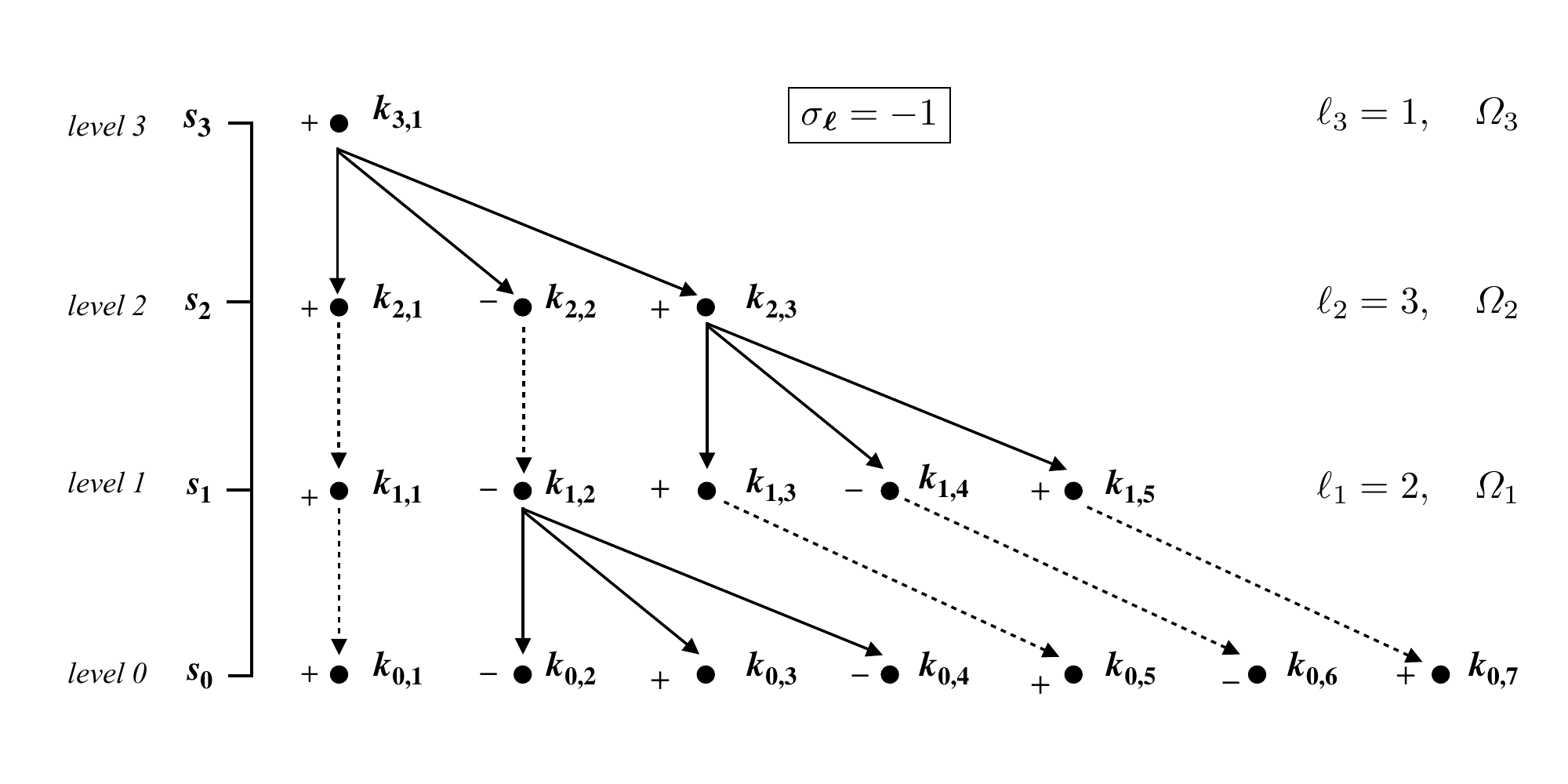}
  \caption{tree of depth $3$.}
  \label{fig:feynman}
\end{figure}
which we now explain. 

The trees will be constructed in reverse order of their usage.  Therefore the labeling of the wave numbers will be done backwards:  $n-j$, $0\le j\le n$.

The tree corresponding to $\jm_{n,\boldsymbol \ell}$, is given as follows. 

 \begin{itemize}[ leftmargin =*]
 
 \item There are $n+1$ levels in the tree, the bottom level is the $0^{th}$ level. Descending from the top to the bottom, each level is generated from the previous level by an integration  by parts step. Thus level  $j$ represents the terms present after $n-j$ integration by parts.
 
 \item   $k_{j,m}$ denote the wave numbers present in level $j$,  and therefore $1\le m \le 2(n-j)+1$. 
 
 \item  $k_{j,m}$ has a parity $\sigma_m $ due to complex conjugation.  For $m$  odd or even, $\sigma_m= +1$ or $\sigma_m= -1$ respectively.
\[
a_{k_{j,m},\sigma_m} =
\begin{cases}
a_{k_{j,m}} & \mbox{if $\sigma_m = +1$} \\[1em]
\overline{a_{k_{j,m}} }& \mbox{if $\sigma_m = -1$}
\end{cases}\,.
\]

\item For each level $j$, we associate a number $\ell_j$, which signals  out  the wave number $ k_{j,\ell_j}$ which has 3 branches.  This is the wave number of the $a$ (or $\bar a$) that was differentiated by the $j^{th}$ integration by parts. The index vector  $\boldsymbol{\ell}$,  keeps track of the integration by parts history in the tree for $\jm_{n,\ell}$. The entries  $\ell_{j}$, $1\le j \le n$,  are given by
 \[
\boldsymbol{\ell}= (\ell_1,\dots, \ell_n) \in \{1,\dots,2n-1\} \times  \{1,\dots,2n-3\} \times \dots \times \{1,2,3\} \times \{1 \}.
\]

\item The tree has a signature $\sigma_{\boldsymbol\ell}= \prod _{j=1}^n (-1)^{\ell_j +1}$.
\item \emph{Transition rules}.  To go from level $j$ to level $j-1$, the wave numbers are related as follows
\begin{equation}
\label{eq:transition}
\begin{cases}
k_{j,m} = k_{j-1,m} & \mbox{for $m <\ell_{j}$} \\
k_{j,m} = k_{j-1,m+2} & \mbox{for $\ell_{j}  <  m$} \\
k_{j,\ell_{j}} = k_{j-1,\ell_{j}} -k_{j-1,\ell_j+1} + k_{j-1,\ell_j +2}
\end{cases}
\end{equation}
Note that for any $j$, $\sum_{m=1}^{2(n-j)+1}(-1)^{m+1} k_{j,m}=  k_{n,1}=k$.  
The wave numbers at level $0$, i.e., those present in $\jm_{n,\boldsymbol \ell}$, are labeled  
\[
 \boldsymbol{k}= ( k_{0,1}, \dots ,k_{0,2n+1}) \in  (\mathbb{Z}^d_L)^{2n+1}\,.
\]   

\item  At each level $j$, the derivative  of the element with wave number  $ k_{j,\ell_j}$ (due to the integration by parts), generates a oscillatory term with frequency
\[
\varOmega_j(\boldsymbol{k})  =  (-1)^{\ell_j +1}\left( Q(k_{j,\ell_j}) - Q(k_{j-1,\ell_j}) + Q(k_{j-1,\ell_j +1}) - Q(k_{j-1,\ell_j +2})\right)\,.
\]

\item We introduce  variables $\boldsymbol{s} = (s_0,\dots,s_n) \in \mathbb{R}_+^{n+1}$;   $t_j(\boldsymbol{s}) = \sum\limits_{k=0}^{j-1} s_k$, $1\leq j \leq n$. 
This choice of variables can be explained as follows. Repeated integration by parts generates terms of the form
\[
\int\limits_0^t g_0(s_0) \int\limits_{s_0}^t g_1(s_1) \dots  \int\limits_{s_{n-2}}^t g_{n-1}(s_{n-1}) =\int\limits_0^t g_0(s_0) \int\limits_{0}^{t-s_0}g_1(s_0 +s_1) \dots \hskip -8mm \int\limits_0^{t- s_0-\dots -s_{n-2}}  \hskip -8mm g_{n-1}(s_0 + \dots + s_{n-1}) 
\]
which can be written as 
\[
\int_{\mathbb{R}_+^{n+1}}g_0(s_0) g_1(s_0 +s_1) \dots g_{n-1}(s_0 + \dots + s_{n-1})  \delta (t-\sum_{l=0}^n s_l)
\]
\end{itemize}

With this notation at hand,  
\[
\begin{split}
&\jm_0=  a^0_k, \quad\,\,  \jm_{1} = \jm_{1,1} =  \eqref{penguin2},\quad\,\,  \jm_2=  \jm_{2,(1,1)} +  \jm_{2,(2,1)} +  \jm_{2,(3,1)},\\
& \jm_{2,(2,1)}= \eqref{penguin4}, \quad    \jm_{2,(1,1)} = \jm_{2,(3,1)} = \frac 12 \eqref{penguin3},
 \end{split}
 \]
and Figure \ref{fig:feynman} represents $\jm_{3,(2,3,1)}$.  The general formula for $\jm_{n, \boldsymbol{\ell}}$  is given by 
\begin{equation}\label{Jnl}
\jm_{n,\boldsymbol{\ell}}(t,\boldsymbol{k}) = \left(\frac{i\lambda^2}{L^{2d}} \right)^{n}  \sigma_{\boldsymbol \ell}\hskip -3mm \sum\limits_{\boldsymbol{k}\in (\mathbb{Z}^d_L)^{2n+1}}  
\hskip -3mm\delta_{k_{n,1}}^k \prod_{j=1}^{2n+1} a^0_{k_{0,j},\sigma_j} \hskip -3mm
  \int\limits_{\mathbb{R}_+^{n+1}} \prod_{m=1}^n e^{-2\pi i t_m(s) \varOmega_m (\boldsymbol{k})}  \delta\left(t-\sum\limits_0^n s_i\right) d\boldsymbol{s}
\end{equation}
Here and throughout the manuscript we write 
\[
\delta_j^k =\begin{cases}1, \quad k=j,\\ 0, \quad k\ne j,\end{cases}
\]
while $\delta(\cdot )$ is the Dirac delta.

Finally, we write
$
R_n(t,  k)({\boldsymbol a})=\sum\limits_{\boldsymbol \ell}\int\limits_0^t  R_{n, \boldsymbol{\ell}}(t, s_0;  k)(\boldsymbol a^{(s_0)}) ds_0, 
$
where
\begin{multline}\label{Rnl}
R_{n, \boldsymbol{\ell}}(t, s_0; k)({\boldsymbol b})=\left(\frac{i\lambda^2}{L^{2d}} \right)^{n} \sigma_{\boldsymbol \ell} \hskip -3mm 
\sum\limits_{\boldsymbol{k}\in (\mathbb{Z}^d_L)^{2n+1}} \hskip -3mm
\delta_{k_{n,1}}^k \prod_{j=1}^{2n+1} b_{k_{0,j},\sigma_j} 
\int\limits_{\mathbb{R}_+^{n}} \prod_{j=1}^n e^{-2\pi i t_j(s) \varOmega_j (\boldsymbol{k})}\times\\ 
 \delta\left(t-s_0-\sum\limits_1^n s_i\right)d\boldsymbol{s}.
\end{multline}

\subsection{Bound on the correlation}

Our aim is to prove the following proposition.

\begin{proposition} If $t < L^{d-\epsilon_0}$, then 
\label{merganser}
\begin{equation}\label{eq:merganser}
\left|\sum_{n+n'=S}\sum_{\ell, \ell'} \mathbb{E} (\jm_{n, \boldsymbol{\ell}}(t,k) \overline{\jm_{n', \boldsymbol{\ell}'}(t,k)}) \right| \lesssim_S( \log t)^2 \left( \frac{t}{\sqrt \tau} \right)^{S} \frac{1}{t}.
\end{equation}
\end{proposition}
\begin{remark}The trivial estimate would be that
$$
\left| \sum_{n+n'=S}\sum_{\ell, \ell'} \mathbb{E} (\jm_{n, \boldsymbol{\ell}}(t,k) \overline{\jm_{n', \boldsymbol{\ell}'}(t,k)}) \right| \lesssim \left( \frac{t}{\sqrt \tau} \right)^S.
$$
Indeed, $\jm_{n, \boldsymbol{\ell}} \jm_{n', \boldsymbol{\ell}'}$ comes with a prefactor $\left( \frac{\lambda^2}{L^{2d}} \right)^{n+n'}$; the size of the domains where the time integration takes place is $O(t^{n+n'})$; and the summation over $\boldsymbol{k}$ and $\boldsymbol{k'}$ is over $2d(n+n'+1)$ dimensions, half of which are canceled by the pairing (see below), out of which $d$ further dimensions are canceled by the requirement that $k_{n,1} = k$. Overall, this gives a bound $\left( \frac{\lambda^2}{L^{2d}} \right)^{n+n'} \times t^{n+n'} \times L^{d(n+n')} = \left( \frac{t}{\sqrt \tau} \right)^{n+n'}$.

Therefore, the above proposition essentially allows a gain of $\frac{1}{t}$ over the trivial bound.  This gain of $\frac{1}{t}$ comes from cancelations in the ``non degenerate interactions" as will be exhibited by equation \eqref{mallard1}.
\end{remark}

Before we start the proof of Proposition~\ref{merganser}, we shall classify the transitions \eqref{eq:transition} as degenerate  if 
\[
k_{j, \ell_j}\in \{k_{j-1, \ell_{j}}, k_{j-1, \ell_j+2}\}  ,
\]
i.e., if the parallelogram with verticies $(k_{j, \ell_j}, k_{j-1, \ell_{j} -1},   k_{j-1, \ell_j},  k_{j-1, \ell_j+2})$ degenerates into a line.
 In this case $\varOmega_{\ell_j} (\boldsymbol{k})=0$. 
When \emph{all} transitions in a tree that represents
 $\jm_{n, \boldsymbol{\ell}}$ are degenerate we denote the term by $D_{n, \boldsymbol{\ell}}(t,k)$, and if one transition is non degenerate we denote it by $\widetilde\jm_{n,\boldsymbol{\ell}}(t,k) $, that is 
 \begin{equation}\label{treesplitting}
\jm_{n, \boldsymbol{\ell}}(t, k)=\widetilde \jm_{n, \boldsymbol{\ell}}(t, k) +D_{n, \boldsymbol{\ell}}(t, k)
\end{equation}
\begin{align}
D_{n, \boldsymbol{\ell}}(t,k) &= \left(\frac{i\lambda^2}{L^{2d}} \right)^{n}\sigma_{\boldsymbol \ell} \sum_{\boldsymbol{k} \in (\mathbb{Z}^d_L)^{2n+1}} \delta_{k_{n,1}}^k (1-\Delta({\boldsymbol k})) \prod_{j=1}^{2n+1} a^0_{k_{0,j},\sigma_j}    
\int_{\mathbb{R}_+^{n+1}}   \delta\Big(t-\sum_0^n s_i\Big) d\boldsymbol{s}\notag
\\
&= 2^n\frac{t^n}{n!} \left(\frac{i\lambda^2}{L^{2d}} \right)^{n}\sigma_{\boldsymbol \ell} a_k^0\sum_{\boldsymbol{k} \in (\mathbb{Z}^d_L)^{n}}  \prod_{j=1}^{n} |a_{k_j}^0|^2, \label{e:degenerate_decomp}
\end{align}
\begin{multline}\label{eq:nondeg}
\widetilde\jm_{n,\boldsymbol{\ell}}(t,k) = \hskip -2pt  \left(\frac{i\lambda^2}{L^{2d}} \right)^{n}  \sigma_{\boldsymbol\ell} \hskip -3mm \sum\limits_{\boldsymbol{k}\in (\mathbb{Z}^d_L)^{2n+1}}   \hskip -4mm \delta_{k_{n,1}}^k \Delta({\boldsymbol k}) \prod_{j=1}^{2n+1} a^0_{k_{0,j},\sigma_j}  \hskip -4mm  \int\limits_{\mathbb{R}_+^{n+1}} \hskip -1mm  \prod_{m=1}^n e^{-2\pi i t_m(s) \varOmega_m (\boldsymbol{k})}  \delta \Big (t-\sum\limits_0^n s_i\Big) d\boldsymbol{s} ,
\end{multline}
where $\Delta({\boldsymbol k})=1-\prod_{j=1}^n\delta_{\{k_{j-1, \ell_{j}+1}, k_{j-1, \ell_j+1+\sigma_{j, \ell_j}}\}}^{k_{j, \ell_j}}$. Note that  $\Delta({\boldsymbol k})=1$ whenever $\widetilde\jm_{n,\boldsymbol{\ell}}(t,k) \ne 0$. 

\subsection{Cancellation of degenerate interactions}

As can be seen from a simple computation in the formula for $D_{n, \boldsymbol{\ell}}$, the contribution of each $\mathbb E(D_{n, \boldsymbol{\ell}}(t,k)\overline D_{n', \boldsymbol{\ell}'}(t,k))$ to the sum in \eqref{eq:merganser} is of size $\sim \left( \frac{t}{\sqrt{\tau}} \right)^S$, which is too large. Luckily, all those terms cancel out as shows the lemma below.

Note that this cancellation between graph expectations is essentially due to the invariance of the expectation $\mathbb E|a_k|^2$ under Wick renormalization, which is a classical trick in the analysis of the nonlinear Schr\"odinger equation that eliminates all degenerate interactions. However, working at the level of graph expectations might be applicable in more general contexts.

\begin{lemma}\label{cancellation} For any $S\geq 2$
$$\sum_{n+n'=S}\sum_{\boldsymbol \ell,  \boldsymbol \ell'} \mathbb E(D_{n, \boldsymbol{\ell}}(t,k)\overline D_{n', \boldsymbol{\ell}'}(t,k))=0.$$
\end{lemma}

\begin{proof}
First we note that since each level in the tree has parity equal to $1$, then
\[
\sum_{\boldsymbol \ell}\sigma_{\boldsymbol \ell} =\prod_{j=1}^n (\mbox{parity of line $j$)}= (1)^n =1\,.
\]
Hence by equation  \eqref{e:degenerate_decomp} 
\[
\sum_{\boldsymbol \ell} D_{n, \boldsymbol{\ell}}(t, k)=2^n\frac{t^n}{n!} \left(\frac{i\lambda^2}{L^{2d}} \right)^{n} \left(\sum_{\boldsymbol{k} \in (\mathbb{Z}^d_L)^{n}}  \prod_{j=1}^{n} |a_{k_j}^0|^2\right) a_k^0\,.
\]
Thus we obtain
\[
\sum_{n+n'=S}\sum_{\boldsymbol  \ell, \boldsymbol \ell'} \mathbb E(D_{n, \boldsymbol{\ell}}(t,k)\overline D_{n', \boldsymbol{\ell}'}(t,k))=2^{S} t^S\left(\frac{\lambda^2}{L^{2d}} \right)^{S}\left(\sum_{\boldsymbol{k} \in (\mathbb{Z}^d_L)^{S}}  \prod_{j=1}^{S} |a_{k_j}^0|^2\right) |a_k^0|^2 \left(\sum_{n+n'=S}\frac{i^{n-n'}}{n! n'!}\right).
\]

The result will follow once we show that 
\begin{equation*}
\sum_{n+n'=S}\frac{i^{n-n'}}{n! n'!}=0.
\end{equation*}

This follows by parametrizing the above sum as $\{(n, n')=(S-j, j): j=0, \ldots S\}$, which gives
\begin{equation*}
\sum_{n+n'=S}\frac{i^{n-n'}}{n! n'!}= i^S \sum_{j=0}^{S} \frac{(-1)^{j}}{(S-j)!j!}=\frac{i^S}{S!} \sum_{j=0}^{S} \frac{(-1)^{j}S!}{(S-j)!j!} =\frac{i^S}{S!}(1+x)^{S}\big|_{x=-1}=0.
\end{equation*}
\end{proof}

\subsection{Estimate on non-degenerate interactions}

Proposition \ref{merganser} now follows from the following lemma:
\begin{lemma} \label{nondeg trees} 
Suppose $G_{n', \boldsymbol{\ell}'}(t,k)\in \{\widetilde \jm_{n', \boldsymbol{\ell}'}(t,k)), D_{n', \boldsymbol{\ell}'}(t,k))\}$, then for $0<t< L^{d-\epsilon_0}$, 
$$
\left|\mathbb{E} (\widetilde \jm_{n, \boldsymbol{\ell}}(t,k) \overline{G_{n', \boldsymbol{\ell}'}(t,k)}) \right| \lesssim_n (\log t)^2 \left( \frac{t}{\sqrt \tau} \right)^{n+n'} \frac{1}{t}.
$$
\end{lemma}

\begin{proof}
We will only consider the case of $G_{n', \boldsymbol{\ell}'}(t,k)=\widetilde \jm_{n', \boldsymbol{\ell}'}(t,k)$, since the case $G_{n', \boldsymbol{\ell}'}(t,k)=D_{n', \boldsymbol{\ell}'}(t,k))$  is  easier to bound. Using the identity
\[
\delta \left(t - \sum_0^n s_j\right) = {\frac{1}{2\pi}}\int e^{-i\alpha (t - \sum_{j=0}^n s_j)}\,d\alpha\,,
\]
we can write for any $(e_0,\dots,e_n) \in \mathbb{R}^{n+1}$ and $\eta>0$
\begin{align}
\int_{\mathbb{R}_+^{n+1}} \prod_{j=0}^n e^{-i s_j e_j}  \delta\left(t-\sum_0^n s_i\right) d\boldsymbol{s} 
&=\frac{e^{\eta t}}{ 2\pi} \int_{\mathbb{R}} e^{-i \alpha t}  \prod_{j=0}^n \frac{i}{ \alpha -e_j + i\eta}\,d\alpha\,.
\end{align}
Thus by choosing  $\eta = \frac{1}{t}$, we have
\begin{multline*}
\widetilde \jm_{n, \boldsymbol{\ell}}(t,k) =  {  \frac{ie}{ 2\pi}} \left(- \frac{\lambda^2}{L^{2d}} \right)^{n} \sigma_{\boldsymbol\ell}\sum_{\boldsymbol{k} \in (\mathbb{Z}^d_L)^{2n+1}} \delta_{k_{n,1}}^k\Delta({\boldsymbol k}) \prod_{j=1}^{2n+1} a^0_{k_{0,j},\sigma_j}\times \\
\int \frac{e^{-i\alpha t} d\alpha}{(\alpha-\varOmega_1 - \varOmega_2 -  \dots - \varOmega_{n} + \frac{i}{t}) \dots (\alpha - \varOmega_{n} + \frac{i}{t})(\alpha + \frac{i}{t})}.
\end{multline*}
Here we employed the notation  $\varOmega_j = 2\pi \varOmega_{j}(\boldsymbol{k})$.

To bound $\mathbb{E} (\widetilde \jm_{n, \boldsymbol{\ell}}(t,k) \overline{\widetilde \jm_{n', \boldsymbol{\ell}'}(t,k)})$, we will simplify the notation by setting $k_{0,2n+1+j}= k'_{0,j}$, which also preserves the parity convention. 
Consequently, we have $2n +2n' +2$ wave numbers present in the expression for 
$\mathbb{E} (\widetilde \jm_{n, \boldsymbol{\ell}}(t,k) \overline{\widetilde \jm_{n', \boldsymbol{\ell}'}(t,k)})$.
\begin{figure}[h!]
  \includegraphics[width=0.9\linewidth]{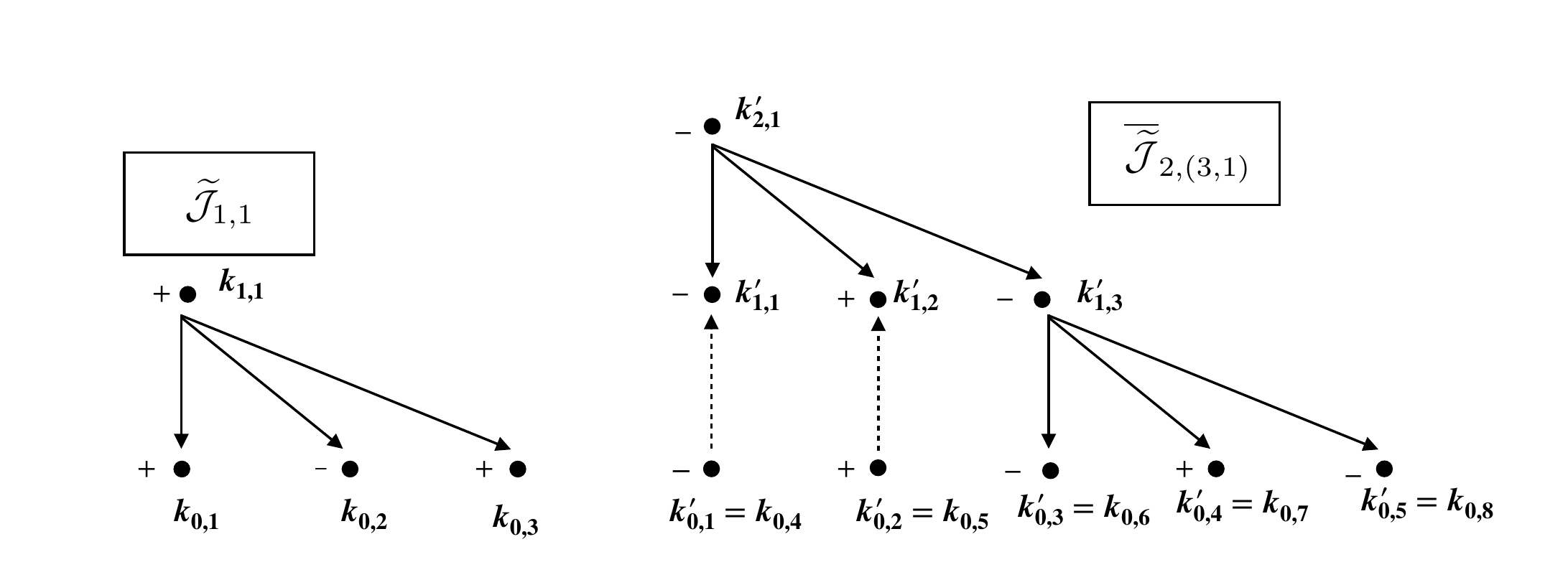}
  \caption{Relabeling  trees.}
  \label{fig:feynman-produce}
\end{figure}

Next, since the phases are i.i.d.\ with mean $0$, then only specific paring of the wave numbers contribute nonzero terms, namely the paring should be between terms with the same wave number and opposite parity. For this reason we  introduce $\mathcal{P} = \mathcal{P}(n,n',\boldsymbol{\sigma},\boldsymbol{\sigma}')$ a class of pairings indices and parities, as illustrated in Figure \ref{pairing}
\[
\mathcal{P}\ni \psi: \{1,\dots, 2n+2n'+2\} \to  \{1,\dots, 2n+2n'+2\} \Leftrightarrow \psi(j)=l \Rightarrow 
 \psi(l)=j, \quad \mbox{and }\, \sigma_{\psi(j)}= - \sigma_j
\]

Furthermore,  we define the pairing  of wave numbers induced by $\psi$, 
$\varGamma_\psi(\boldsymbol{k},\boldsymbol{k}') = \prod_{j=1}^{2n+2n'+2} \delta^{k_{0,j}}_{ k_{0,\psi(j)}}.
$ By the independence of the phases $\vartheta_{k_{0,j}}(\omega)$, we have,
\[
\left| \mathbb{E}_\omega \left[ \prod_{j=1}^{2n+1} e^{i \sigma_{0,j} \vartheta_{k_{0,j}}(\omega)} \prod_{j'=1}^{2n'+1} e^{- i \sigma_{0,j'}  \vartheta_{k_{0,j}}(\omega)} \right]  \right| \lesssim \sum_{\psi \in \mathcal{P}} \varGamma_\psi(\boldsymbol{k}\,.\boldsymbol{k}'),
\]
Hence we obtain
\begin{multline*}
 \left|\mathbb{E} (\widetilde \jm_{n, \boldsymbol{\ell}}(t,k) \overline{\widetilde \jm_{n', \boldsymbol{\ell}'}(t,k)}) \right|
 \lesssim \left( \frac{\lambda^2}{L^{2d}} \right)^{n+n'} \sum_{\psi \in \mathcal{P}} \sum_{\substack{\boldsymbol{k} \in (\mathbb{Z}^d_L)^{2n+1} \\ \boldsymbol{k}' \in (\mathbb{Z}^d_L)^{2n'+1}}} \mathscr{A}_{\psi}(\boldsymbol{k},\boldsymbol{k}')\times\\
 \bigg| \int \frac{e^{-i\alpha t} d\alpha}{(\alpha-\varOmega_1 - \dots - \varOmega_{n} + \frac{i}{t}) \dots (\alpha - \varOmega_{n} + \frac{i}{t})(\alpha + \frac{i}{t})} \times \\
 \int \frac{e^{i\alpha' t} d\alpha'}{(\alpha'-\varOmega_1' - \dots - \varOmega_{n'}' + \frac{i}{t}) \dots (\alpha'- \varOmega_{n'}' + \frac{i}{t})(\alpha' + \frac{i}{t})} \bigg|.
\end{multline*}
where  $\mathscr{A}_{\psi}(\boldsymbol{k},\boldsymbol{k}')=\delta^{k_{n,1}}_k\delta_{ k'_{n',1}}^k\Delta({\boldsymbol k})\Delta({\boldsymbol k'}) \varGamma_\psi(\boldsymbol{k},\boldsymbol{k}') 
\prod_{j=1}^{2n+1} \sqrt{\phi(k_{0,j})} \prod_{j'=1}^{2n'+1} \sqrt{\phi(k_{0,j'}')} $.

%
%

By H\"older's inequality, for any $m \geq 1$ and $b_1,\dots,b_{n'+1} \in \mathbb{R}$,
\begin{equation}
\label{boundintalpha}
\int_{\mathbb{R}} \frac{d\alpha'}{|\alpha'-b_1 + \frac{i}{t}| \dots |\alpha'-b_{m+1}  + \frac{i}{t}|} \lesssim t^{m}\, ,
\end{equation}
and applying this bound to the $\alpha'$ integral yields
\begin{multline*}
 \left| \mathbb{E} (\widetilde \jm_{n, \boldsymbol{\ell}}(t,k) \overline{\widetilde \jm_{n', \boldsymbol{\ell}'}(t,k)}) \right| \lesssim  t^{n'} \left( \frac{\lambda^2}{L^{2d}} \right)^{n+n'}\sum_{\psi \in \mathcal{P}} \sum_{\substack{\boldsymbol{k} \in (\mathbb{Z}^d_L)^{2n+1} \\ \boldsymbol{k}' \in (\mathbb{Z}^d_L)^{2n'+1}}}\mathscr{A}_{\psi}(\boldsymbol{k},\boldsymbol{k}') \times\\
\left| \int \frac{d\alpha}{(\alpha-\sum^{n}_{l=1}\varOmega_l+ \frac{i}{t}) \dots (\alpha -\varOmega_{n} + \frac{i}{t})(\alpha + \frac{i}{t})} \right| \,.
\end{multline*}

Let $p=p(\boldsymbol{k})$ be the smallest integer such that $k_{p+1, \ell_p}\notin \{k_{p, \ell_{p}+1}, k_{p, \ell_p+1+\sigma_{p+1, \ell_p}}\}$, i.e., in the tree for $\widetilde \jm_{n, \boldsymbol{\ell}}$ the transition from level $p+1$ to level $p$  is not degenerate. Note that  $0 \leq p \leq n-1$, and
\begin{multline*}
 \left| \mathbb{E} (\widetilde \jm_{n, \boldsymbol{\ell}}(t,k) \overline{\widetilde \jm_{n', \boldsymbol{\ell}'}(t,k)}) \right| \lesssim  t^{n'} \left( \frac{\lambda^2}{L^{2d}} \right)^{n+n'}\sum_{\psi \in \mathcal{P}} \sum_{\substack{\boldsymbol{k} \in (\mathbb{Z}^d_L)^{2n+1} \\ \boldsymbol{k}' \in (\mathbb{Z}^d_L)^{2n'+1}}}\mathscr{A}_{\psi}(\boldsymbol{k},\boldsymbol{k}') \times\\
 \left| \int \frac{d\alpha}
 {(\alpha-\sum^{n}_{l=p+ 1}\varOmega_l + \frac{i}{t})^{p+1} \dots (\alpha -\varOmega_{n} + \frac{i}{t})(\alpha + \frac{i}{t})} \right| 
 =: \sum_\psi \sum_p  \mathfrak{I}_{p,\psi}\,.
\end{multline*}

We now set
\begin{alignat*}{4}
&I_1 = \ell_p, \hskip 15mm && I_2 = \ell_p + 1, \hskip 15mm && I_3 = \ell_p +2, \hskip 15mm&& \boldsymbol{k_p}= (k_{p,I_1}, k_{p,I_2}, k_{p,I_3}),\\
&J_1 = \psi(I_1),  && J_2 = \psi(I_2),  && J_3 = \psi(I_3)\,. && \phantom{a}
\end{alignat*}
 Note that, by definition of $p$,
$$
\{ I_1,I_2,I_3 \} \cap \{ J_1,J_2,J_3 \} = \emptyset.
$$
The figure below illustrate all the introduced notations and parings for the product of two non degenerate terms.
\begin{figure}[h!]
  \includegraphics[width=0.9\linewidth]{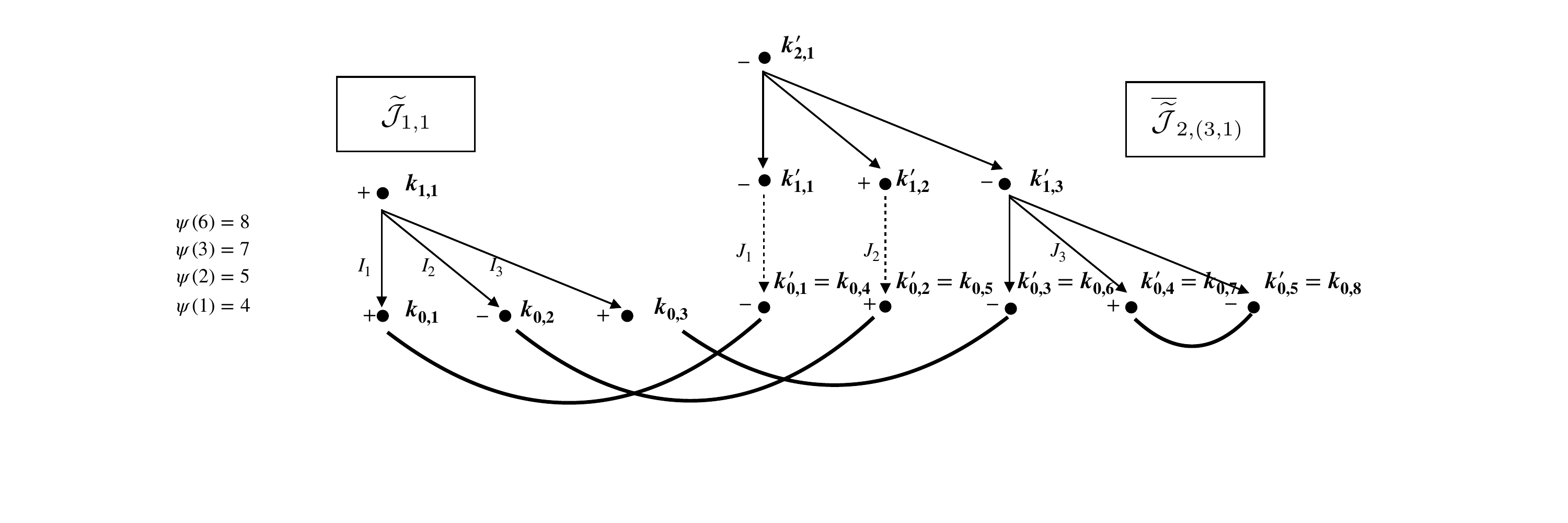}
  \caption{Pairing trees.}
  \label{pairing}
\end{figure}

We  distinguish three cases depending on the values of the numbers $J_i$.

\underline{ Case 1:  $J_1,J_2,J_3 \geq 2n+2$.} 
For a fixed $p$ and $\psi$ we sum over all wave numbers in $  \mathfrak{I}_{p,\psi}$ that yield degenerate transitions, i.e., wave numbers generated in rows $0\le l \le p-1$.  This contributes $L^{dp}$ to the bound,
\[
\mathfrak{I}_{p,\psi} \lesssim  t^{n'} L^{dp} \left( \frac{\lambda^2}{L^{2d}} \right)^{n+n'} \sum\nolimits^* \sum_{\boldsymbol{k_p}}
\mathscr{A}_{\psi}(\boldsymbol{k},\boldsymbol{k}')
\int \frac{d\alpha}{|( \alpha-\sum^{n}_{l=p+ 1}\varOmega_l + \frac{i}{t})^{p+1} \dots ( \alpha-\varOmega_{n}  + \frac{i}{t})(\alpha + \frac{i}{t})|} ,
\]
where $\sum^*$ stands for the sum over $k_{p,j}$, where $1 \leq j \leq 2(n-p)  + 1$ and $j \notin \{ I_1,I_2,I_3 \}$, and $k'_{0, j'}$ for $1\leq j'\leq 2n'+1$ with $j' \notin \{J_1, J_2,J_3\}$.

The contribution of the above integral is acceptable as long as the denominator is $O(\langle \alpha \rangle^{-2})$. Therefore, it suffices to prove the desired bound when the domain of integration reduces to $\alpha \in [-R,R]$, for some $R>0$, since  the resonance moduli $\varOmega_i$ are bounded.  Furthermore by bounding the integrand by $\frac{t^{n-1}}{|\alpha-\varOmega_{p+1} - \dots - \varOmega_{n}  + \frac{i}{t}| |\alpha + \frac{i}{t}|}$, matters reduce to estimating
\begin{equation}\label{eq:bound4}
t^{n'+n-1} L^{dp} \left( \frac{\lambda^2}{L^{2d}} \right)^{n+n'} \sum\nolimits^* 
 \mathscr{A}_{\psi}(\boldsymbol{k},\boldsymbol{k}')
\sum_{\boldsymbol{k_p}}
 \int_{-R}^R \frac{d\alpha}{|\alpha-\sum^{n}_{l=p+ 1}\varOmega_l    + \frac{i}{t}| |\alpha + \frac{i}{t}|}\, .
\end{equation}
By the identity $ k_{p+1,I_1} - k_{p,I_2} = \sigma_{p,I_1} (k_{p,I_3} - k_{p,I_1})$, this can also be written
\[
t^{n'+n-1} L^{dp} \left( \frac{\lambda^2}{L^{2d}} \right)^{n+n'} \sum\nolimits^* 
 \mathscr{A}_{\psi}(\boldsymbol{k},\boldsymbol{k}')
 \int_{-R}^R  \left[ \sum_{\boldsymbol{k_p}}
 \delta^{ k_{p+1,I_1}} _{ k_{p,I_2} + \sigma_{p,I_1} (k_{p,I_3} - k_{p,I_1})} 
  \frac{1} {| \alpha-\sum^{n}_{l=p+ 1}\varOmega_l    + \frac{i}{t}|} \right] \frac{d\alpha}{ |\alpha + \frac{i}{t}|},
\]
and since $\sum_{j=1}^{2(n-p)+1}\sigma_{p, j}k_{p, j}=k_{n,1}=k$,  we note that
\begin{multline}\label{eq:q4}
\varOmega_{p+1}+\ldots+\varOmega_n  = Q(k) - \sum_{j = 1}^{2(n-p) + 1} \sigma_{p,j} Q(k_{p,j})=\\
 -\sigma_{p, I_1}Q(k_{p, I_1})-\sigma_{p, I_3} Q(k_{p, I_3})-\sigma_{p, I_2}Q(k_{p+1,I_1} - \sigma_{p,I_1} (k_{p,I_3} - k_{p,I_1}))+C
\end{multline}
where $C$ depends only on $k$ and the variables $k_{p, j}$ with $j\notin \{I_1, I_2, I_3\}$. 

By setting $P= k_{p,I_3}$ and $R=k_{p,I_1}$,  for $t < L^{\nu}$  we bound
\[
\sum_{\substack{P,R \in \mathbb{Z}^d_L\\ |P|, |R|\le 1}}  \frac{1}{\left| -Q(P) + Q(R) - Q(N+P-R) + C + \frac{i}{t}\right|},
\]
using the  equidistribution result in Section \ref{sectnumbtheo}. If $| -Q(P) + Q(R) - Q(N+P-R) + C| \le t^{-1}$, we have by Corollary \ref{l:lossy_count}
\[
\sum_{\substack{P,R \in \mathbb{Z}^d_L\\ |P|, |R|\le 1}}  \frac{1}{\left| -Q(P) + Q(R) - Q(N+P-R) + C + \frac{i}{t}\right|}\lesssim t \Big(L^{2d }\frac 1t + L^{d }\Big)\,.
\]
Whereas for $| -Q(P) + Q(R) - Q(N+P-R) + C| \ge t^{-1}$, we bound 
\begin{multline}
\label{mallard1}
\sum_{\substack{P,R \in \mathbb{Z}^d_L\\ |P|, |R|\le 1}}  \frac{1}{\left| -Q(P) + Q(R) - Q(N+P-R) + C + \frac{i}{t}\right|}  \lesssim \\
\sum_{\frac{1}{t} < 2^n \lesssim 1} 2^{-n} \sum_{|-Q(P) + Q(R) - Q(N+Q-R) + C| \sim 2^n} 1 \lesssim L^{2d} \sum_{\frac{1}{t} < 2^n \lesssim 1} 1 \lesssim L^{2d} \log t.
\end{multline}
Therefore, we can bound \eqref{eq:bound4} by
\begin{align*}
t^{n'+n-1} L^{dp}\Big( L^{2d }\log t + t L^{d }\Big) \left( \frac{\lambda^2}{L^{2d}} \right)^{n+n'} \sum\nolimits^* \varGamma_\psi(\boldsymbol{k},\boldsymbol{k}')  \int_{-R}^R  \frac{d\alpha}{ |\alpha + \frac{i}{t}|},
\end{align*}

The sum $\sum^*$ is over $2(n+n'-p-2)$ variables; however, because of  the pairing $\varGamma_\psi$, half of them drop out, so that the remaining sum is $\lesssim L^{d(n+n'-p-2)}$.  
Ans since $\int_{-R}^R \frac{d\alpha}{ |\alpha + \frac{i}{t}|} \lesssim \log t$, the above expression can be bounded by, 
\[
 t^{n'+n-1}\Big((\log t)^2 L^{d(n+n')} + t(\log t)  L^{d(n+n' -1)}\Big)\left( \frac{\lambda^2}{L^{2d}} \right)^{n+n'},
\]
which gives the stated bound.

\underline{Case 2: only two of $J_1,J_2,J_3$ are $\geq 2n+2$.} Suppose for instance that $J_2\leq 2n+1$. Then, there exists $I_4\leq 2(n-p)+1$ such that $\psi(I_4) = J_4 \geq 2n+2$ (such an index exists because there is an odd number of elements in the set of elements in $\{1, \ldots, 2(n-p)+1\}\setminus \{I_1, I_2, I_3, J_2\}$, so they cannot be paired together completely). One can then follow the above argument replacing $I_2$ by $I_4$.

 \underline{Case 3: two of $J_1,J_2,J_3$ are  $\leq 2n+1$} Assume for instance that $J_1, J_3 \leq 2n+1$. Proceeding as in Case 1, it suffices to bound
 \[
  t^{n'} L^{dp} \left( \frac{\lambda^2}{L^{2d}} \right)^{n+n'} \sum\nolimits^* \sum_{k_{p,I_1},k_{p,I_3}}
\mathscr{A}_{\psi}(\boldsymbol{k},\boldsymbol{k}')
\int \frac{d\alpha}{|( \alpha-\sum^{n}_{l=p+ 1}\varOmega_l + \frac{i}{t})^{p+1} \dots ( \alpha-\varOmega_{n}  + \frac{i}{t})(\alpha + \frac{i}{t})|} ,
\]
where $\Sigma^*$ is the sum over $k_{p,j}$, with $j \in \{1,\dots,2(n-p)+1\} \setminus \{I_1,I_3,J_1,J_3\}$, and over $k_{0,j'}$, with $j' \in \{1,\dots,2n'+1\}$.

A crucial observation is that, since $\sum_{j=1}^{2(n-p)+1}\sigma_{p, j}k_{p, j}=k_{n,1}=k$, the wave numbers  $k_{p,I_1}$ and $k_{p,I_3}$ do not contribute to this sum since the paring $k_{p,I_1} = k_{p,J_1}$ and $k_{p,I_3} = k_{p,J_3}$, causes them to cancel one another. Furthermore, $0\le p \leq n-2$ since $J_1,J_3 \leq 2n+1$, and therefore we bound the integrand by $\frac{t^{n-1}}{| \alpha-\varOmega_{p+2} - \dots - \varOmega_{n} + \frac{i}{t}||\alpha + \frac{i}{t}|}$. Overall, we can bound the above by
\[
L^{dp} t^{n+n'-1} \left( \frac{\lambda^2}{L^{2d}} \right)^{n+n'} \sum\nolimits^*\mathscr{A}_{\psi}(\boldsymbol{k},\boldsymbol{k}')
 \int_{-R}^R \left( \sum_{k_{p,I_1},k_{p,I_3}} \frac{1}{| \alpha-\varOmega_{p+2} - \dots - \varOmega_{n} + \frac{i}{t}|} \right) \frac{d\alpha}{|\alpha + \frac{i}{t}|}.
\]

From equation \eqref{eq:q4}, we conclude
\[
\sum^{n}_{l=p+ 2}\varOmega_l =  -\sigma_{p, I_1}Q(k_{p, I_1})-\sigma_{p, I_3} Q(k_{p, I_3})-\sigma_{p, I_2}Q(k_{p+1,I_1} - \sigma_{p,I_1} (k_{p,I_3} - k_{p,I_1}))+C
\]
where $C$ only depends on the variables in $\sum^*$. Applying~\eqref{mallard1} enables us to bound the inner sum by  $L^{2d}  \log t$, and the $\alpha$ integral by $\log t$. Finally, the number of variables in $\sum^*$ is $2(n+n'-p-1)$. By  pairing them  there are only $n+n'-p-1$, and  fixing   $k_{n,1}=k$ brings their number down to $n+n'-p-2$. Thus  $\sum^*$ will contribute $\lesssim L^{d(n+n'-p-2)}$. Overall, we obtain the bound
$$
\lesssim (\log t)^2 t^{n'+n-1} L^{d(n+n')}  \left( \frac{\lambda^2}{L^{2d}} \right)^{n+n'},
$$
which is the desired estimate.
\end{proof}

\section{Deterministic local well-posedness}\label{LWP section}

Local or long time existence existence of smooth solutions is usually  carried out by using Strichartz estimates to bound solutions.  The known Strichartz estimates for our problem \eqref{impStrich} are not sufficient to allow us to prove existence of solutions for a long time interval where the wave kinetic equation \eqref{KW} emerges.  However, if the data is assumed to be random,  then one has improved estimates due to Khinchin's inequality ~\cite{BT}.  Based on this, we first present a local well-posedness theorem provided the data satisfies a certain estimate.  In Section \ref{KhinchineSection}, we show that such an improved  estimate occurs with high probability.

Moreover, to use the results from Sections \ref{feynman} and   \ref{sectnumbtheo}, we will restrict discussion  to  the case $T < L^{d-\epsilon_0}$.

\subsection{Strichartz estimate}   Recall  equation  \eqref{impStrich} , which can be written as,
\[
\|e^{it\Delta_{\beta}} P_1 \psi\|_{L^4_{t,x}([0,T]\times \mathbb{T}_L^d)} \le  S_{d,\epsilon} 
 \|\psi\|_{L^2(\mathbb{T}_L^d)}, \qquad  S_{d,\epsilon}\defeq C_{d,\epsilon}
 L^\epsilon \Big \langle  \frac{T}{L^{ \theta_d}} \Big \rangle^{1/4} 
 \]

 Moreover if we denote the characteristic function of the unit cube centered at 
$j \in \mathbb{Z}^d$ by $\mathds{1}_{B_j}$, and define 
\[
\widehat{\psi_{B_j}}(k) = \mathds{1}_{B_j} (k) \widehat{\psi}(k),\qquad \mbox{and therefore }\;\; \psi_{B_0} = P_1\psi\,.
\]
Then, using the Galilean invariance  $\left| e^{-it \Delta_\beta} \psi_{B_j}(x)\right| = \left| [e^{-it \Delta_\beta} (e^{2\pi i j x}\psi)_{B_0}](x-2tj)\right|$, we  have
\begin{equation}\label{eq:stj}
\|e^{it\Delta_{\beta}}  \psi_{B_j}\|_{L^4_{t,x}([0,T]\times \mathbb{T}_L^d)} \le  S_{d,\epsilon}
 \|\psi\|_{L^2(\mathbb{T}_L^d)}.
\end{equation}
Converting this estimate to its dual, and applying the Christ-Kiselev inequality, one gets
\begin{align}\label{eq:st2}
& \left\| \int\limits_0^T e^{-is \Delta} F_{B_j}(s) \,ds \right\|_{L^2(\mathbb{T}^d_L)} \le  S_{d,\epsilon}\| F \|_{L_{t,x}^{4/3} ([0, T]\times \mathbb{T}^d_L)} \\
& \left\| \int\limits_0^t e^{i(t-s) \Delta} F_{B_j}(s) \,ds \right\|_{L^4_{t,x}([0,T]\times \mathbb{T}_L^d)} \le  S_{d,\epsilon}^2 \| F \|_{L_{t,x}^{4/3} ([0, T]\times \mathbb{T}^d_L)}\label{eq:ck}
\end{align} 
for an appropriate choice of $C_{d,\epsilon}$ used in the definition of $ S_{d,\epsilon}$.

\subsection{A priori bound in $\boldsymbol{Z_T^s}$ and energy}  Let $Z_T^s$  denote the function space defined by the norm,
\begin{equation}\label{Zs norm}
\|u\|_{Z^s_T}=\left( \sum\limits_{j \in \mathbb{Z}^d} \langle j \rangle^{2s} \|u_{B_j}\|^2_{L^{4}_{t,x} ([0, T]\times \mathbb{T}^d_L)} \right)^{1/2},
\end{equation}
then the $Z^s_T$ norm of the nonlinearity is bounded.  
\begin{lemma} Fix $s > \frac{d}{2}$.
For every $\epsilon_0 >0$, and an appropriate choice of $C_{d,\epsilon_0}$, we have
\begin{equation}\label{Zsnonlin}
\left\|\int_0^t e^{i(t-s)\Delta_\beta} |u(s)|^2u(s) ds\right\|_{Z_T^s} \le 
 S_*^2 \|u\|_{Z_T^s}^3,  \qquad  S_* \defeq S_{d,\epsilon_0}=  C_{d,\epsilon_0}
 L^{\epsilon_0} \Big \langle  \frac{T}{L^{ \theta_d}} \Big \rangle^{1/4} .
 \end{equation}
\end{lemma} 
\begin{proof} Consider  $v\in L^{\frac43}_{t,x}([0,T]\times \T^d_L)$,  and let   $\widetilde v(s, x)= \int_s^Te^{i(s-s')\Delta_\beta}v(s') ds'$, then
\begin{align*}
\left\|\int_0^t e^{i(t-s)\Delta_\beta}P_{B_j} |u(s)|^2u(s) ds  \right \|_{L^4_{t,x}([0,T]\times \T^d_L)} &\\
&\hskip -20mm  =\sup_{\|v\|_{ L^{4/3}_{t,x}} = 1} \int_{0}^T\int_0^t \int_{\T^d_L}\left[ e^{i(t-s)\Delta_\beta}P_{B_j} |u(s)|^2u(s) \right] \overline{v(t,x)}dx, \,ds\, dt \\
&\hskip -20mm =\sup_{\|v\|_{ L^{4/3}_{t,x}} = 1} \int_{0}^T\int_0^t \int_{\T^d_L}|u(s)|^2u(s)  \overline{e^{i(t-s)\Delta_\beta}v_{B_j}(t,x)} \, dx\,ds\, dt \\
&\hskip -20mm  =\sup_{\|v\|_{ L^{4/3}_{t,x}} = 1}\int_{0}^T \int_{\T^d_L}|u(s)|^2u(s) \overline{\widetilde v_{B_j}(s,x)}\,ds\, dx.
\end{align*}
Using equation \eqref{eq:ck}, we have  for every $\epsilon_0 >0$,
\begin{align*}
&\int_{0}^T\int_{\T^d_L} |u(s)|^2u(s) \overline{\widetilde v_{B_j}(s, x)}\,ds\, dx=\sum_{j_1-j_2+j_3-j=O(1)}\int_{0}^T\int_{\T^d_L} u_{B_{j_1}}\overline{u_{B_{j_2}}}u_{B_{j_3}} \overline{\widetilde v_{B_j}}\,ds\, dx\\
&\qquad \lesssim \sum_{\substack{j_1-j_2+j_3-j=O(1)\\ j_1, j_2, j_3 \in \Z^d}} \prod_{k=1}^3\|u_{B_{j_k}}\|_{L^4_{t,x}} \|\overline{\widetilde v_{B_j}}\|_{L^4_{t,x}}\lesssim 
 L^{\epsilon_0} \Big \langle \frac{T}{L^{ \theta_d}}\Big \rangle^{1/2}
\sum_{\substack{j_1-j_2+j_3-j=O(1)\\ j_1, j_2, j_3 \in \Z^d}} \prod_{k=1}^3\|u_{B_{j_k}}\|_{L^4_{t,x}},
\end{align*}
and therefore
\begin{align*}
\left\|\int_0^t e^{i(t-s)\Delta_\beta}P_{B_j} |u(s)|^2u(s) ds\right\|_{L^4_{t,x}([0,T]\times \T^d)}\lesssim 
 L^{\epsilon_0} \Big \langle \frac{T}{L^{ \theta_d}}\Big \rangle^{1/2}
 \sum_{\substack{j_1-j_2+j_3-j=O(1)\\ j_1, j_2, j_3 \in \Z^d}} \prod_{k=1}^3\|u_{B_{j_k}}\|_{L^4_{t,x}}.
\end{align*}
Consequently, for  $s>d/2$, we have
\begin{multline*}
\left(\sum_{j \in  \Z^d}\langle j\rangle^{2s}\left\|\int_0^t e^{i(t-s)\Delta_\beta}P_{B_j} |u(s)|^2u(s) ds\right\|_{L^4_{t,x}}^2\right)^{1/2}\\
\lesssim   L^{\epsilon_0} \Big \langle \frac{T}{L^{ \theta_d}}\Big \rangle^{1/2}
\left(\sum_{j \in  \Z^d}\langle j_1\rangle^{2s} \|u_{B_{j_1}}\|_{L^4_{t,x}}^2\right)^{1/2} \left(\sum_{\ell \in  \Z^d} \|u_{B_\ell}\|_{L^4_{t,x}}\right)^2
\lesssim L^{\epsilon_0} \Big \langle \frac{T}{L^{ \theta_d}}\Big \rangle^{1/2} \|u\|_{Z_T^s}^3.
\end{multline*}
proving equation \eqref{Zsnonlin}.
\end{proof}

\begin{lemma}[A priori energy estimates]
\begin{equation}\label{eq:engn}
\left\|\int_0^t e^{i(t-s)\Delta_\beta}|u|^2u \, ds \right\|_{L^\infty_t H_x^s} \le S_*  \|u\|_{Z^s_T}^3.
 \end{equation}
\end{lemma}
\begin{proof}
By duality, we have
\begin{align*}
\left\|\int_0^t e^{i(t-s)\Delta_\beta}|u|^2u \, ds \right\|_{L^\infty_t H_x^s}
\leq& \sup_{\substack{\| \psi \|_{L^2_x} = 1\\ 0\le t \le T}} \hskip 3mm \int_0^t \int_{\T^d_L}|u|^2u  \; e^{is\Delta_\beta} \langle \nabla \rangle^s \overline{\psi} \, dx\, ds\\
=& \sup_{\substack{\| \psi \|_{L^2_x} = 1\\ 0\le t \le T}} \hskip 3mm \sum_{j_1-j_3+j_3-j_4=O(1)}\int_0^t \int_{\T^d_L}u_{B_{j_1}}\overline{u_{B_{j_2}}} u_{B_{j_3}}  \; e^{is\Delta_\beta} \langle \nabla \rangle^s \overline{\psi_{B_{j_4}}}\, dx\, ds\, .
\end{align*}

Applying the Strichartz estimate \eqref{eq:stj} yields
\begin{align*}
&\sum_{j_1-j_3+j_3-j_4=O(1)}\left|\int_0^t \int_{\T^d_L}u_{B_{j_1}}\overline{u_{B_{j_2}}} u_{B_{j_3}}  \; e^{is\Delta_\beta} \nabla ^s \psi_{B_{j_4}}\, dx\, ds\right| \\
& \qquad \qquad \lesssim \sum_{j_1-j_2+j_3-j_4=O(1)} \langle j_4 \rangle^{s} \prod_{k=1}^3\|u_{B_{j_k}}\|_{L^4_{t,x}} \|e^{is\Delta_\beta} \psi_{B_{j_4}}\|_{L^4_{t,x}}\\
&\qquad \qquad \lesssim  L^{\epsilon_0} \Big \langle \frac{T}{L^{\theta_d}} \Big \rangle^{1/4} 
\sum_{j_1-j_2+j_3-j_4=O(1)} \langle \max(|j_1|, |j_2|, |j_3|)\rangle^{s}\prod_{k=1}^3\|u_{B_{j_k}}\|_{L^4_{t,x}} \| \psi_{B_{j_4}}\|_{L^2_x}\\
&\qquad \qquad \lesssim  L^{\epsilon_0} \Big \langle \frac{T}{L^{\theta_d}} \Big \rangle^{1/4} 
 \left(\sum_{j} \|\psi_{B_{j}}\|_{L^2}^2\right)^{1/2}\left(\sum_{j} \langle j \rangle^{2s} \|u_{B_{j}}\|_{L^4}^2\right)^{1/2}\left(\sum_{j} \|u_{B_{j}}\|_{L^4}\right)^{2} \\
&\qquad \qquad  \lesssim L^{\epsilon_0} \Big \langle \frac{T}{L^{\theta_d}} \Big \rangle^{1/4} \|u\|_{Z^s_T}^3 \| \psi \|_{L^2_x}\,.
\end{align*}
This establishes the stated bound.
\end{proof}

\subsection{Existence theorem} Local well-posedness for \eqref{NLS} will be established  in the  space $Z^s_T$, with data $f$ of size at most $\mathscr I$,
\begin{equation}\label{eq:id}
\mathscr{I}\defeq L^{\epsilon_0}  (TL^{-d})^{\frac{1}{4}} \ge \| e^{it\Delta_\beta} f \|_{Z^s_T} . 
\end{equation}

 This seemingly strange normalization is actually  well adapted to the problem we are considering.  Indeed, consider for simplicity initial data $f$ supported on Fourier frequencies $\lesssim 1$, whose $L^2$ norm is of size $L^{\epsilon_0} $,  and with random  Fourier coefficients of uncorrelated phases.
Then we  expect  $e^{it\Delta_\beta} f$  to be  evenly spread over $\mathbb{T}_L^d$. By conservation of the $L^2$ norm, this corresponds to  $ \| e^{it\Delta_\beta} f \|_{Z^s_T} \sim \mathscr{I}$.

\begin{theorem} \label{pheasant} 
Let $f\in Z_T^s$  with $\mathscr I$ and $S_*$  defined in equations \eqref{eq:id} and \eqref{Zsnonlin} respectively, then
\[  
\begin{cases}
i\partial_t u -   \Delta_\beta u= - \lambda^{2} |u|^{2} u\\
u(0,x)= f(x)
\end{cases}
\]
is locally well-posed in $Z^s_T$, provided
\begin{equation}
\label{boundT}
R\overset{def}{=} 12(\lambda S_*\mathscr{I})^2 \le \frac 1{2}.
\end{equation}
The solution $u\in Z^s_T$,  satisfies $\| u \|_{Z_T^s} \leq 2\mathscr{I}$.  Moreover 
\begin{equation}\label{eq:engineq}
\|u\|_{L_t^\infty H_x^s ([0,T]\times \mathbb{T}^d_L)} \leq \|f\|_{H_x^s} + C\lambda^{2} S_* \mathscr{I}^3=  \|f\|_{H_x^s} + C \frac{R}{S_*} \mathscr{I} \le  \|f\|_{H_x^s} + C R .
\end{equation}
\end{theorem}

\begin{remark}
The time scale $T$ over which the solution can be constructed would be equal to $\sqrt{\tau}$, up to subpolynomial losses in $L$, if the long-time Strichartz estimate conjectured in~\cite{DGG} for $p=4$ could be established. Since it is currently not known to be true, the result stated above gives a shorter time scale, with a more complicated numerology.
\end{remark}

\begin{proof}
This theorem is proved by using a contraction mapping argument, to find a fixed point  of the map,
$$
\varPhi(u) = e^{it\Delta_{\beta}}f + i\lambda^{2} \int\limits_0^t  e^{i(t-s)\Delta_{\beta}} |u|^{2} u(s)\,ds,
$$
 in $\{ u\in Z^s_T \bigm\lvert  \|u\|_{Z^s_T } \le 2 \mathscr{I}\}$.  Consequently   $u = \lim_{N\to \infty} \varPhi^N(0)$, where $\varPhi^N$ stands for the $N$-th iterate of $\varPhi$:
$$
\varPhi^{0}(0) = e^{it\Delta_{\beta}}f, \qquad \varPhi^{N+1}(0) =  e^{it\Delta_{\beta}}f +  i\lambda^{2}\int\limits_0^t  e^{i(t-s)\Delta_{\beta}} |\varPhi^{N}(0)|^2 \varPhi^{N}(0)\,ds\,. \qquad
$$

To  check that $\varPhi$ is a contraction on $B_{Z^s_T} (0, 2\mathscr{I})$, note that by equation \eqref{Zsnonlin},
\[
\| \varPhi(u) - \varPhi^{0}(0) \|_{Z^s_T}  = \left\|  \lambda^{2} \int\limits_0^t e^{i(t-s) \Delta_{\beta}} |u|^{2} u(s) \, ds \right\|_{Z^s_T} 
 \le \lambda^{2}S_*^2 \| u \|_{Z^s_T}^3 \leq  8\lambda^{2}S_*^2 \mathscr{I}^3  \leq R \mathscr{I}\le   \frac 12 \mathscr{I}.
\]
and thus $\varPhi$ maps $B_{Z^s_T} (0, 2\mathscr{I})$ into itself. Again, by equation \eqref{Zsnonlin},
\[
\| \varPhi(u) - \varPhi(v) \|_{Z^s_T}
 \le3 \lambda^2S_*^2(2\mathscr{I})^2
 \| u - v  \|_{Z^s_T} \leq R\| u - v  \|_{Z^s_T} \le \frac 12 \| u - v  \|_{Z^s_T}.
\]
Therefore $\varPhi$ is a contraction on $\{ u\in Z^s_T \bigm\lvert  \|u\|_{Z^s_T } \le 2 \mathscr{I}\}$, and  the $H^s$ estimate follows from the  a priori energy bound.
\end{proof}

Besides the established bounds on $u$, we need to investigate the rate of convergence of $\varPhi^N(u) \to u $. 

\begin{corollary} \label{pheasant2}
Under the conditions of Theorem~\ref{pheasant}, there holds
$$
\| u - \varPhi^N(0) \|_{L^\infty H^s}\leq C\frac{R^N}{S_*}   \mathscr{I} \le CR^N
$$
\end{corollary}

\begin{proof}
Since  $\varPhi$ is a contraction with modulus $R$, then 
\[
\| \varPhi^j(0) - \varPhi^{j-1}(0) \|_{Z^s_T} \leq R^{j-1}\| \varPhi^0(0) \|_{Z^s_T} .
\]
Moreover  the energy estimate \eqref{eq:engn} gives
\[
\| \varPhi^{j+1}(0) - \varPhi^j(0) \|_{L^\infty_T H^s} \le C\frac{R}{S_*}  \| \varPhi^{j}(0) - \varPhi^{j-1}(0)  \|_{Z^s_T}.
\]
Consequently by writing  $u - \varPhi^N(0) = \sum\limits_{j=N}^\infty \varPhi^{j+1}(0) - \varPhi^j (0)$, we bound 
\begin{align*}
\| u - \varPhi^N(0) \|_{L^\infty_T H^s} & \leq \sum\limits_{j=N}^\infty \left\| \varPhi^{j+1}(0) - \varPhi^j (0) \right\|_{L^\infty_T H^s} \le C\frac{R}{S_*}
\sum\limits_{j=N}^\infty \left\| \varPhi^{j}(0) - \varPhi^{j-1} (0) \right\|_{Z^s} \\
&\le C\frac{R}{S_*}\sum_{j=N}^\infty R^{j-1} \| \varPhi^0(0) \|_{Z^s_T} \leq C\frac{R^N}{S_*} \| \varPhi^0(0) \|_{Z^s_T}.
\end{align*}
\end{proof}
Next we establish an energy bound for the Feynman trees,
$$
U_{n, \boldsymbol{\ell}}(t) = \frac{1}{L^d} \sum\limits_{k \in \mathbb{Z}^d_L} J_{n, \boldsymbol{\ell}}(t,k) e^{2\pi i (k \cdot x + t Q(k))}.
$$
\begin{corollary} \label{pheasant3}
Under the conditions of Theorem~\ref{pheasant},
\[
\| U_{n, \boldsymbol{\ell}} \|_{L^\infty_T H^s} \le C  \frac{R^N}{S_*} \mathscr{I}\le CR^N
\]
\end{corollary}

\begin{proof}
Since  $U_{n, \boldsymbol{\ell}}$ is the linear propagator of $\jm_{n, \boldsymbol{\ell}}$ in physical space, then they can be represented by the following iterative procedure:  Set $v_0^m = e^{2\pi it\Delta_\beta} u_0$ for $0 \leq m \leq 2n+1$ and for any $1 \leq j\leq n$ we define $v^j_m$ for $0\leq m \leq 2(n-j)+1$ as $v_j^m= v_{j-1}^m$ if $m < \ell_j$, and $ v_j^m= v_{j-1}^{m+2} $ if $m > \ell_j$, where we set
\[
v^{\ell_j}= i\lambda^{2} \int_0^t e^{i(t-s)\Delta_\beta}v_{j-1}^{\ell_j} \overline{v_{j-1}^{\ell_j+1}} v_{j-1}^{\ell_j+2}ds\,.
\]
Hence we have $U_{n, \boldsymbol{\ell}} = v_n^1$. 

Using the energy estimate \eqref{eq:engn}, we bound 
$$
\| v_n^1 \|_{L^\infty_T H^s}  \le \lambda^{2}S_* \| v_{n-1}^{1} \|_{Z_T^s}  \| v_{n-1}^{2} \|_{Z_T^s} 
\| v_{n-1}^{3} \|_{Z_T^s}.
$$
We can then  descend down the tree by estimating $v_{n-j}^{\ell_{n-j}}$ using the $Z^s$ estimate \eqref{Zsnonlin}. This leads to the stated bound. 
\end{proof}

\section{Improved integrability through randomization}\label{KhinchineSection}

Recall that
$$
u_0 = \frac{1}{L^d} \sum\limits_{k \in \mathbb{Z}^d_L} \sqrt{\phi(k)} e^{2\pi i k \cdot x} e^{2\pi i \vartheta_k(\omega)},
$$ 
where the $\vartheta_k(\omega)$ are independent random variables, uniformly distributed on $[0,2\pi]$.

For any $t$, $s$, $\omega$, we have
$$
\| e^{it \Delta_\beta} u_0 \|_{H^s} =\left[  \frac{1}{L^d} \sum\limits_{k \in \mathbb{Z}^d_L} \langle k \rangle^{2s} \phi(k) \right]^{1/2}\,.
$$
In other words, the randomization of the angles of the Fourier coefficients does not have any effect on $L^2$ based norms. This is not the case for Lebesgue indices larger than 2.

\begin{theorem}
\label{rndmthm}
Assume that $|\phi(k)| \lesssim \langle k \rangle^{-s}$, with $s>\frac{d}{2}$. Then
\begin{itemize}
\item[(i)] $\mathbb E \left\|e^{it\Delta_{\beta}}u_0 \right\|_{L^4_{t,x}([0,T]\times\mathbb{T}^d_L)}^4 \lesssim \frac T{L^d} \|u_0 \|_{L^2_x}^4$
\item[(ii)] (large deviation estimate) 
$$
\mathbb{P} \left[ \left\| e^{it\Delta_{\beta}}u_0 \right\|_{L^4_{t,x}([0,T]\times\mathbb{T}^d_L)}^4 > \lambda \right] \lesssim \exp \left( - c \left( \frac{\lambda}{T^{1/4} L^{-d/4}} \right)^2 \right)  
$$
\end{itemize}
\end{theorem}

\begin{proof} $(i)$ The proof is more or less standard. See~\cite{BT} for instance.

\noindent $(ii)$ We follow the argument in~\cite{BT}. By Minkowski's inequality (for $p \geq 4$) and Khinchin's inequality,
\begin{align*}
\left\|e^{it\Delta_{\beta}}u_0 \right\|_{L^p_\omega(\varOmega,L^4_{t,x}([0,T]\times\mathbb{T}^d_L))} & \lesssim
\left\| e^{it\Delta_{\beta}}u_0 \right\|_{L^4_{t,x}([0,T]\times\mathbb{T}^d_L,L^p_\omega(\varOmega)))} \\
& \lesssim \frac{\sqrt{p}}{L^d} \left\| \left( \sum \phi(k) \right)^{1/2} \right\|_{L^4_{t,x}([0,T]\times\mathbb{T}^d_L)} \lesssim \sqrt p T^{1/4} L^{-d/4}.
\end{align*}
By Chebyshev's inequality,
$$
\mathbb{P} \left[ \left\| e^{it\Delta_{\beta}}u_0 \right\|_{L^p_{t,x}([0,T]\times\mathbb{T}^d_L)} > \lambda \right] \lesssim \lambda^{-p} (C_0 \sqrt  p T^{1/4} L^{-d/4})^p.
$$
The desired inequality is then obvious if $\lambda < 2eC_0 T^{1/4} L^{-d/4}$; if not, it follows upon choosing \\
$p = \left( \frac{\lambda}{C_0 T^{1/4} L^{-d/4} e} \right)^2$.
\end{proof}

As a consequence, we deduce the following proposition.

\begin{proposition}
\label{oriole}
Let $\epsilon_0>0$, $ \alpha >s + \frac{d}{2}$, and assume that $|\phi(k)| \lesssim \langle k \rangle^{-2\alpha} $. Then, for two constant $C,c>0$,
$$
\mathbb{P} \left[ \left\| e^{it\Delta_{\beta}}u_0 \right\|_{Z^s} < T^{1/4} L^{\epsilon_0-d/4} \right] > 1 - C e^{-c L^{\epsilon_0}}.
$$
\end{proposition}

\begin{proof} Applying Theorem~\ref{rndmthm} to $(u_0)_{B_j}$,
$$
\mathbb{P}  \left[ \left\| e^{it\Delta_{\beta}}u_0 \right\|_{L^4_{t,x}([0,T]\times\mathbb{T}^d_L)} > \langle j \rangle^{-\alpha} T^{1/4} L^{\frac{\epsilon_0}{2}-\frac{d}{4}} \right] \lesssim \exp ( -c \langle j \rangle^{2\alpha} L^{\epsilon_0}).
$$
Therefore, for $L$ sufficiently large,
\begin{align*}
\mathbb{P} \left[ \left\| e^{it\Delta_{\beta}}u_0 \right\|_{Z^s} < T^{1/4} L^{\epsilon_0-d/4} \right] 
& > 1 - \sum\limits_j \mathbb{P} \left[ \left\| e^{it\Delta_{\beta}}(u_0)_{B_j} \right\|_{Z^s} > T^{1/4} L^{\frac{\epsilon_0}{2}-d/4} \langle j \rangle^{-\alpha} \right] \\
& > 1 -  C \sum\limits_j \exp ( -c \langle j \rangle^{2\alpha} L^{\epsilon_0})\\
& > 1 -  C e^{-c L^{\epsilon_0}}.
\end{align*}
\end{proof}

\section{Proof of the main theorem}\label{maintheoremSection}

Fix $\epsilon_0 > 0$ sufficiently small, and recall that $T\le L^d$, with 
\[
 S_* =  C_{4,\epsilon_0}
 L^{\epsilon_0} \Big \langle  \frac{T}{L^{ \theta_d}} \Big \rangle^{1/4}, \qquad 
 \mathscr{I}= L^{\epsilon_0}  (TL^{-d})^{\frac{1}{4}}, \qquad   \mbox{and }\, 
 R\overset{def}{=} 12(\lambda S_*\mathscr{I})^2 .
 \]

\begin{enumerate}[1), leftmargin =*]
\item \emph{Excluding exceptional data}.  Let $E_{\epsilon_0,L}$ be the event $\{ \left\| e^{it\Delta_{\beta}}u_0 \right\|_{Z^s} \le \mathscr{I}\}$, and $F_{\epsilon_0,L}$ its contrary: $\{ \left\| e^{it\Delta_{\beta}}u_0 \right\|_{Z^s} > \mathscr{I}\}$. By Proposition~\ref{oriole},
$$
\mathbb{P} (F_{\epsilon_0,L}) \lesssim e^{-c L^{\epsilon_0}}.
$$
This is the set appearing in the statement of Theorem \ref{cubic theorem}. By conservation of mass 
$$
\mathbb{E} \left( |a_k(t)|^2 \right) = \mathbb{E} \left( |a_k(t)|^2 \; | \; E_{\epsilon_0,L} \right) + O_{\ell^\infty}( e^{-c L^{\epsilon_0}} L^d).
$$
\item \emph{Iterative resolution}.  To ensure that $R\le \frac 12$ we restrict the range of the parameters $\lambda$, $T$ relative to $L$.   There are two regimes depending on  the Strichartz constant $S_*$ and the number theory restriction $t\le L^{d-\epsilon_0}$ (see Remark \ref{rk:degenerate}).
\begin{itemize}[ leftmargin =*]
\item  $L^{ \theta_d}\lesssim T \lesssim L^{d}$.  The condition $R\le \frac 12$ translates into $T\sim\lambda^{-2}L^{\frac{d+ \theta_d}{2}-4\epsilon_0}$. Therefore  we restrict  $\lambda$ to
\[
 L^{\frac{-d+ \theta_d -8\epsilon_0}{4}}\lesssim \lambda \lesssim L^{\frac{d- \theta_d -8\epsilon_0}{4}}\, .
\]
For this range of parameters,  the energy inequality \eqref{eq:engineq} implies $\|u\|_{L_t^\infty H_x^s ([0,T]\times \mathbb{T}^d_L)} \lesssim 1$. 

\item $T\lesssim L^{ \theta_d}$.  In this case the condition on $R$ restricts $T \sim \min(L^{ \theta_d}, \lambda^{-4}L^{d-8\epsilon_0})$, and therefore 
 \[L^\frac{d -\theta_d-8\epsilon_0}{4}\lesssim\lambda.
 \]
Here the energy inequality also implies  $\|u\|_{L_t^\infty H_x^s ([0,T]\times \mathbb{T}^d_L)} \lesssim 1$. 

Note that for these ranges of parameters $T\le L^{-2\delta}\sqrt{\tau}$, where $\delta$ is that of Theorem \ref{t:asymptotic formula}.
\end{itemize}

With these restrictions on the range of the parameters we proceed by writing $u = \varPhi^N(0) + u - \varPhi^N(0)$. Note that since $\varPhi^N(0)$ is a polynomial of degree $3^N$, we write
\[
u = \sum\limits_{n=0}^N U_{n, \boldsymbol{\ell}} + \sum\limits_{(n,\ell) \in S^N} U_{n, \boldsymbol{\ell}} +  u - \varPhi^N(0),
\]
where $S^N$ includes all the terms in $\varPhi^N(0)$ of degree greater than  $N$.

By Corollary~\ref{pheasant2} and Proposition~\ref{pheasant3}, this implies that
\[
u = \sum\limits_{n=1}^N \sum\limits_\ell U_{n, \boldsymbol{\ell}} + O_{L^\infty_tH^s_x} \left(R^N\right)
\]
where the  constant depends on $N$.   In terms of  Fourier variables this can be written as,
\[
|a_k(t)|^2 = \left| \sum\limits_{n=1}^N \sum\limits_\ell J_{n, \boldsymbol{\ell}} \right|^2 + O_{\ell_L^{1,2s}} \left( R^N\right)=\left| \sum\limits_{n=1}^N \sum\limits_\ell J_{n, \boldsymbol{\ell}} \right|^2 + O_{\ell^{\infty}} \left( L^d R^N\right).
\]

\item \emph{Pairing}.  By Proposition~\ref{merganser},
\begin{align*}
 \left| \sum\limits_{n=1}^N \sum\limits_\ell J_{n, \boldsymbol{\ell}} \right|^2 &= \mathbb{E} \left[ |J_1(k)|^2 + J_0(k) \overline{J_2(k)} + \overline{J_0(k)} J_2(k) \right] + O \left( \frac{t}{\tau} \frac{t \log t}{\sqrt{\tau}} \right) \\ 
& = \phi_k + \frac{2 \lambda^4}{L^{4d}} \sum\limits_{k - k_1 + k_2 - k_3=0} \phi_k \phi_{k_1} \phi_{k_2} \phi_{k_3}
\left[ \frac{1}{\phi_k} - \frac{1}{\phi_{k_1}} + \frac{1}{\phi_{k_2}} - \frac{1}{\phi_{k_3}} \right] 
\times \\
& \hskip 65mm \left| \frac{\sin(t \pi \varOmega(k,k_1,k_2,k_3))}{\pi \varOmega(k,k_1,k_2,k_3)} \right|^2 
+ O_{\ell^\infty} \left( \frac{t}{\tau} \frac{t \log t}{\sqrt{\tau}} \right)
\end{align*}

\item \emph{Large box limit $L \to \infty$.} By the  equidistribution theorem \ref{t:asymptotic formula}, we have for $t < L^{d- \epsilon}$
\begin{multline*}
 \frac{2 \lambda^4}{L^{4d}} \sum\limits_{k - k_1 + k_2 - k_3=0} \phi_k \phi_{k_1} \phi_{k_2} \phi_{k_3} \left[ \frac{1}{\phi_k} - \frac{1}{\phi_{k_1}} + \frac{1}{\phi_{k_2}} - \frac{1}{\phi_{k_3}} \right] \left| \frac{\sin(t \pi \varOmega(k,k_1,k_2,k_3))}{\pi \varOmega(k,k_1,k_2,k_3)} \right|^2 \\
 = \frac{2 \lambda^4}{L^{2d}} \int\limits \delta(\varSigma) \phi(k) \phi(k_1) \phi(k_2) \phi(k_3) \left[ \frac{1}{\phi(k)} - \frac{1}{\phi(k_1)} + \frac{1}{\phi(k_2)} - \frac{1}{\phi(k_3)} \right]\times\\
  \left| \frac{\sin(\pi t\varOmega(k,k_1,k_2,k_3))}{\pi \varOmega(k,k_1,k_2,k_3)} \right|^2 \,dk_1 \, dk_2 \,dk_3 +O_{\ell^\infty}(\frac{t}{\tau}L^{-\delta} ).
\end{multline*}

\item \emph{Large time limit $t \sim T \to \infty$}.  Since for  a smooth function $f$, 
\[
\int \left| \frac{\sin(\pi t x)}{x} \right|^2 f(x)\,dx = \pi^2t f(0) + O(1),
\]
then, with $ \tau=\frac{L^{2d}}{2\lambda^4}$, we have 
\begin{multline*}
 \frac{2 \lambda^4}{L^{2d}} \int\limits \delta(\varSigma) \phi(k) \phi(k_1) \phi(k_2) \phi(k_3) 
\left[ \frac{1}{\phi(k)} - \frac{1}{\phi(k_1)} + \frac{1}{\phi(k_2)} - \frac{1}{\phi(k_3)} \right]\times\\
 \left| \frac{\sin(\pi t\varOmega(k,k_1,k_2,k_3))}{\pi \varOmega(k,k_1,k_2,k_3)} \right|^2 \,dk_1 \, dk_2 \,dk_3 
  = \frac{t}{\tau} \mathcal{T}(\phi,\phi,\phi) + O\left( \frac{1}{\tau} \right).
  \end{multline*}

\end{enumerate}

Consequently, for $\epsilon_0$ sufficiently small  and $t\leq T\le L^{d-\epsilon_0}$,
we  choose $L\geq L_1(\epsilon_0)$ to bound  the error term in Step 1 by  $\frac{t}{\tau}L^{-\epsilon_0}$. Also, since   $R\leq \frac 12$ then by picking $N$ large enough we can 
 bound the error in Step 2 by  $O(\frac{t}{\tau}L^{-\epsilon_0})$.  Similarly, since $t\log t \le L^{-\delta }\sqrt{\tau}$, then  the error for Steps 3, 4, and 5, are of  order $O_{\ell^\infty}(\frac{t}{\tau}L^{-\delta} )$,  and this concludes the proof of Theorem \ref{cubic theorem}.

\section{Number theoretic results}
\label{sectnumbtheo}

 Our aim in this section is to prove the asymptotic formula for the following Riemann sum,
 \begin{theorem}\label{t:asymptotic formula}
 Given $\phi\in \mathscr{S}(\mathbb{R}^d)$ and   $\epsilon>0$, there exists a  $\delta >0$ such that  if  $0<t\leq L^{d -\epsilon}$, then 
 \begin{align*}&
 \sum\limits_{\substack{k_i\in\mathbb{Z}_L^d\\ k-k_1+k_2-k_3=0}}\phi_k \phi_{k_1} \phi_{k_2} \phi_{k_3}  \left[ \frac{1}{\phi_k} - \frac{1}{\phi_{k_1}} + \frac{1}{\phi_{k_2}} - \frac{1}{\phi_{k_3}} \right]
\left| \frac{\sin(\pi t\varOmega(k,k_1,k_2,k_3))}{\pi \varOmega(k,k_1,k_2,k_3)} \right|^2 =\\[.3em]
&L^{2d} \int   \delta(\varSigma)  \phi_k \phi_{k_1} \phi_{k_2} \phi_{k_3}   \left[ \frac{1}{\phi_k} - \frac{1}{\phi_{k_1}} + \frac{1}{\phi_{k_2}} - \frac{1}{\phi_{k_3}} \right]    \left| \frac{\sin(\pi t\varOmega(k,k_1,k_2,k_3))}{\pi \varOmega(k,k_1,k_2,k_3)} \right|^2     dk_1 dk_2 \,dk_3\\[.3em]
 & + O\left(tL^{2d-\delta}\right) + O\left(L^d\right) ,
\end{align*}
where we recall $\varSigma(k,k_1,k_2,k_3) = k - k_1 + k_2 - k_3$.
\end{theorem}

The difficulty in proving this theorem is that $\varOmega$ can be very small, while the stated time interval for the validity of the asymptotic formula is very large. In fact if we restrict ourselves to a timescale which is not too long,  then the asymptotic formula is straight forward as will be demonstrated  in Proposition \ref{th:trivial}. However to prove this theorem as stated we need  to generalize a result of 
 Bourgain on pair correlations of generic quadratic forms   \cite{Bourgain}.
 
Bourgain considered a  positive definite diagonal form,
\begin{equation}\label{qpq}
 Q(n)= \sum_{i=1}^d \beta_in_i^2, \quad n= (n_1,\dots, n_d), \qquad Q(p,q) \defeq Q(p)-  Q(q),\\
\end{equation}
for generic $\beta=(\beta_1,\dots,\beta_d)\in [1,2]^d$, and  proved that for $d=3$ the lattice points in the region,
\[
R_{\mathbb{Z}}\overset{def}{=}\{(p,q)\in\mathbb{Z}^{2d}\cap[0,L]^{2d} \bigm\lvert Q(p,q)\in [a,b], p\ne q\},
\]
are equidistributed at a scale of $\frac 1{L^{\rho}}$, for $0< \rho <d-1$. Specifically, he proved,
\[
\sum_{R_{\mathbb{Z}}}1 =  L^{2(d-1)}(b-a)\mathcal H^{2d-1}\left(\{(x,y)\in [-1,1]^{2d} \bigm\lvert Q(x,y)=0\}\right) + O\left(L^{d-2- \delta}(b-a) \right),
\]
provided  $|a|, |b|<O(1)$ and $L^{-\rho}< b-a<1$.  Here $\mathcal H^{2d-1}$ is the $2d-1$~Hausdorff measure.

Our quadratic form $\varOmega$, restricted to $\varSigma$,  can be transformed to $Q(p,q)$, given in \eqref{qpq},  as follows.
Rescale time $\mu\defeq tL^{-2}$, let $K_i = Lk_i\in \mathbb{Z}$, and denote by 
\[
g(x)=\left(\frac{\sin(\pi x)}{\pi x}\right)^2, \quad  W_0\left(\frac{K}{L},\frac{K_1}{L},\frac{K_2}{L},\frac{K_3}{L}\right)  = \phi_k \phi_{k_1} \phi_{k_2} \phi_{k_3}  \left[ \frac{1}{\phi_k} - \frac{1}{\phi_{k_1}} + \frac{1}{\phi_{k_2}} - \frac{1}{\phi_{k_3}} \right].
 \]
Then the sum can be expressed as
 \[
t^2\sum_{\substack{K,K_1,K_2,K_3\in\mathbb Z^d\\ K-K_1+K_2-K_3=0}}W_0\left(\frac{K}{L},\frac{K_1}{L},\frac{K_2}{L},\frac{K_3}{L}\right)  g (\mu \Omega(K,K_1,K_2,K_3))\,.
 \]
By defining 
\[u'=K_1-K\in\mathbb Z^d,\qquad  u''=K_3-K\in\mathbb Z^d,\qquad \mbox{and }u=(u',u'')\in \mathbb Z^{2d}\]
then
\[\Omega(K,K_1,K_2,K_3)=Q_0(u)\]
where 
\begin{equation}
Q_0(u):=-2\beta_1u'_{1}u''_{1}-2\beta_{2}u'_{2}u''_{2}
-\dots-2\beta_{d}u'_{d}u''_{d} \,.
\end{equation} 
Hence the sum can be expressed as
\begin{equation}\label{innerproduct}
t^2\sum_{\substack{ (u'_i, u''_i)\in\mathbb Z^{2}}}  W_0\left(\frac{K}{L},\frac{u'+K}{L},\frac{u'+u''+K}{L},\frac{u''+K}{L}\right) g(\mu Q_0(u)).
\end{equation}
The quadratic form $Q_0$ can be diagonalized  by making the change of coordinates
\[
p_i= u'_i+u''_i , \qquad q_i= u'_i- u''_i 
\]
where  $p_i$ and $q_i$ are either both even or both odd, i.e.
\[
\sum_{u_i \in \mathbb Z^{2}} =\sum_{p_i,q_i \in 2\mathbb Z }+\sum_{p_i,q_i \in (2\mathbb Z+1) }=\sum_{p_i,q_i \in \mathbb Z }-\sum_{p_i\in 2\mathbb Z, q_i \in \mathbb Z }
-\sum_{p_i\in \mathbb Z ,q_i \in 2\mathbb Z } + 2\sum_{p_i,q_i \in 2\mathbb Z }\,.
\]
Consequently,  the sum \eqref{innerproduct}, can be written as four different sums of the form,
\begin{equation}\label{eq:W}
t^2 \sum_{ (p,q)\in\mathbb Z^{2d}}W\left(\frac{p}{L},\frac{q}{L}\right)g(\mu Q(p,q)), 
\end{equation}
 where $Q(p,q)$ is given by\footnote{There are factors of $2$ missing due to sums over even terms.  However, this has no impact since $\beta$ is generic.} \eqref{qpq}, and where we suppressed the dependence of $W$ on $k$ for convenience. 

\begin{remark}\label{rk:degenerate}  Note that we do not exclude the points when $p_i^2=q_i^2$ for all $i\in[1,\dots,n]$, as Bourgain did.  These points contribute $O(L^d)$ to the sum and will be considered as lower order terms.  They also explain the $O(L^d)$ term in Theorem \ref{t:asymptotic formula}.

It is this fact that prevents us from using the full strength of  our equidistribution result which holds for  $\mu=tL^{-2} \leq L^{d-1 -\epsilon}$, and we use the result for $t\le L^{d -\epsilon}$.  This ensures that 
$O(L^d)$ term is an error in the asymptotic formula.
\end{remark}

 To prove the asymptotic formula given in Theorem \ref{t:asymptotic formula}, with  $0< \mu= tL^{-2} \leq L^{d-1 -\epsilon}$, we proceed as follows:  {\it 1)} identify which part of the sum contributes the leading order term and which part contributes error terms;  {\it 2)} prove equidistribution of lattice points on a coarse scale;  {\it 3)} present Bourgain's theorem on equidistribution on a fine scale; and finally  {\it 4)} prove Theorem \ref{t:asymptotic formula}.
 
  \subsection{Identifying main terms vs error terms.} To identify the leading order term in the equidistribution formula, we first obtain upper bounds on lattice sums that are optimal up to sub-polynomial factor.
 
For generic  $\beta=(\beta_1,\dots,\beta_d)\in [1,2]^d$,
a good upper bound for the linear form $\beta \cdot n\in[a,b]$, where $n=\mathbb{Z}^d$ is a consequence of the pigeonhole principle:
\begin{lemma}\label{linear-count}
The linear form  $\beta \cdot n\in[a,b]$ satisfies the following bound
\begin{equation}\label{e:linear_generic_sum}
\#\{n\in\mathbb{Z}^d\cap[-M,M]^d \bigm\lvert a\leq \beta \cdot n \leq b\} =  \sum_{\substack{a\leq \beta \cdot n \leq b\\\abs{n}\leq M}}1\lesssim M^{(d-1)^+}(b-a)+1
\end{equation}
\end{lemma}
\begin{proof}

Since
$\beta = (\beta_1,\dots,\beta_d)$ are generic, then for $0<|n| \le M$ (see for example \cite{Cassels}, Chapter~VII)
 \[
 |\beta\cdot n| \gtrsim  \frac1{M^{(d-1)^+}}.
 \]
For arbitrary   $n^{(1)}\ne n^{(2)}\in \mathbb{Z}^d$ satisfying  $a\leq\beta \cdot n^{(i)} \leq b$  and $0<\abs{n^{(i)}}\leq M$,
\[
\frac{1}{M^{(d-1)^+}}\lesssim \abs{\beta \cdot (n^{(1)}-n^{(2)})}\leq b-a\,.
\]
 By the pigeonhole principle we obtain \eqref{e:linear_generic_sum}.
 \end{proof}

An upper  bound on the cardinality of the set, 
\[
R_{\mathbb{Z}}\overset{def}{=}\{(p,q)\in\mathbb{Z}^{2d}\cap[0,L]^{2d} \bigm\lvert Q(p,q)\in [a,b], p\ne q\},
\]
can be obtained by bounding the number of lattice points in  subsets of the form,
\[
R_{\mathbb{Z}\ell} =\{(p,q)\in\mathbb{Z}^{2d}\cap[0,L]^{2d} \bigm\lvert  Q(p,q)\in [a,b],  p_i\ne q_i, 1\le i\le\ell, \mbox{ and } p_i=  q_i,  \ell +1\le i\le d  \},
\]  
using  Lemma  \ref{linear-count},  and by using   the divisor bound $d(k)\lesssim_\epsilon k^\epsilon$.

\begin{lemma}\label{l:lossy_count-l}
For $\ell=1,\dots d$ the cardinality of $R_{\mathbb{Z}\ell}$ satisfies the bound
\begin{equation}\label{eq:lossy_count}
\# R_{\mathbb{Z}\ell} =\sum_{R_{\mathbb{Z}\ell}} 1 
\lesssim L^{(d+\ell-2)^+}(b-a)+L^{(d-\ell)^+}
\end{equation}
\end{lemma}
\begin{proof}
Define $k_i=(p_i-q_i)(p_i+q_i)$, for $1\le i \le \ell$.  Since $p_i=q_i$, for   $\ell +1\le i \le d$, we conclude 
\[
\# R_{\mathbb{Z}\ell} \lesssim L^{d-\ell}\sum_{\substack{a\leq \sum\limits_{i=1}^{\ell} \beta_i k_i \leq b\\0<\abs{k}\lesssim L^2}}\left(\sum_{(p_i-q_i)(p_i+q_i)=k_i}1\right)
\]
By the divisor bound
\[\sum_{(p_i-q_i)(p_i+q_i)=k_i}1\lesssim L^{0^+},
\]
and  by \eqref{e:linear_generic_sum}, with $M=L^2$,  we obtain
\[
\# R_{\mathbb{Z}\ell}    \lesssim L^{(d-\ell)^+}\left(L^{2(\ell-1)^+}(b-a)+1\right),
\]
 and  \eqref{eq:lossy_count} follows.
\end{proof}

\begin{corollary}\label{l:lossy_count}
The number of elements in $R_{\mathbb{Z}}$, can be bounded by
\begin{equation}\label{e:lossy_lossy_count}
\# R_{\mathbb{Z}} \lesssim L^{2(d-1)^+}(b-a) + L^{(d-1)^+}
\end{equation}
Moreover, if we further assume $\abs{a},\abs{b}\leq 1$, then we have the improved bound
\begin{equation}\label{e:lossy_count}
\# R_{\mathbb{Z}} \lesssim L^{2(d-1)^+}(b-a) + L^{(d-2)^+},
\end{equation}
\end{corollary}
\begin{proof} It suffices to apply the Lemma \ref{l:lossy_count-l}, and to observe that $\ell\in\{ 1,\dots,d \}$ since $p=q$ is excluded. \eqref{e:lossy_count} follows from noting that if $\abs{a},\abs{b}\leq 1$, then $R_{\mathbb{Z}1}$ is empty.
\end{proof}
\begin{remark}
Note, that in terms of the first estimate \eqref{e:lossy_lossy_count}, the second term may be treated as an error as long as  $b-a\geq L^{-(d-1) + \epsilon_0}$ for some $\epsilon_0>0$. Analogously, the second term of \eqref{e:lossy_count} may  be treated as an error assuming  $b-a\geq L^{-d + \epsilon_0}$.
\end{remark}

Following this remark on identifying the leading order term, we can now identify subsets of $R_{\mathbb{Z}}$ that contribute error terms only.  The first such subsets are when $|p_i-q_i|\lesssim L^{1-\delta}$ for some  fixed $\delta>0$ and some $i$ that we may without loss of generality assume to be $1$.
\begin{lemma} \label{l:small-dia}
For $\abs{a},\abs{b}\le 1$, the number of elements in $R_{\mathbb{Z}}$ satisfying $|p_1-q_1|\lesssim L^{1-\delta}$ satisfy the following bound
\[
\# R_{\mathbb{Z}}\cap  \{(p,q)\in \mathbb{Z}^{2d} \bigm\lvert |p_1-q_1| \lesssim L^{1-\delta}\} \lesssim L^{2(d-1)^+ -\delta} (b-a) + L^{(d-1)^+}\,.
\]
\end{lemma}
\begin{proof} If $p_{i}=q_{i}$ for at least one $i$, then by Corollary \ref{l:lossy_count} with $d$ replaced by $d-1$, we have
\[
\# R_{\mathbb{Z}}\cap  \{(p,q) \in\mathbb{Z}^{2d} \bigm\lvert p_{i}=q_{i} \} \lesssim L\left(L^{2(d-2)^+ } (b-a) +L^{(d-3)^+}\right),
\]
which is lower order.  Moreover, if  $p_i \ne q_i$ for all $i$, and $|p_1-q_1| \lesssim L^{1-\delta}$, then the sum over  $2\le i \le d$ can be bounded by 
$L^{2(d-2)^+}(b-a) + L^{0^+}$, using Lemma \ref{l:lossy_count-l}, while the sum over $p_1$ and $q_1$ can be by $L^{2-\delta}$. This gives a bound of $L^{2-\delta} \left(L^{2(d-2)^+}(b-a) + L^{0^+}\right)$, which is lower order if $d\ge 3$.
\end{proof}

Next we show that if one $p_i$ or $q_i$ is less than $L^{1-\delta}$, where we may again assume $i=1$, then the contribution to the number of elements in $R_{\mathbb{Z}}$ is lower order.
\begin{lemma} \label{l:small-range}
For $\abs{a},\abs{b}\leq 1$,  we have the following estimate
\[
\# R_{\mathbb{Z}}\cap \{(p,q)\in\mathbb{Z}^{2d} \bigm\lvert |p_1| \lesssim L^{1-\delta} \} \lesssim  L^{2(d-1)^+ -\delta}(b-a)+ L^{(d-1)^+}
\]
 \end{lemma}
\begin{proof}
If both  $|p_1| \lesssim L^{1-\delta}$ and  $|q_1| \lesssim L^{1-\delta}$ or $p_i =q_i$ for at least one $i$, then by Lemma \ref{l:small-dia} we have the stated bound.  Otherwise,  the sum over    $2\le i \le d$  contributes $L^{2(d-2)^+}(b-a) + L^{0^+}$, while the sum over $p_1$ and $q_1$ contributes $L^{2-\delta}$.
\end{proof}   
From Lemma \ref{l:small-dia} and Lemma \ref{l:small-range}, we have
\begin{corollary}\label{r:reduc-1}  Setting
\[
R_{\mathbb{Z}\delta}= R_{\mathbb{Z}} \setminus  \bigcup_{i=1}^d\left( \{(p,q)\in \mathbb{Z}^{2d} \bigm\lvert \mbox{ \rm where, } |p_i|,  |q_i|, \mbox{ \rm  or } |p_i- q_i| \lesssim L^{1-\delta}, \, \mbox{ \rm for at least one $i$} \} \right).
\]
Then, for $\abs{a},\abs{b}\leq 1$, we have the following cardinality bound on the set difference $R_{\mathbb{Z}}\setminus R_{\mathbb{Z}\delta}$
\[\#R_{\mathbb{Z}}\setminus R_{\mathbb{Z}\delta}\les L^{2(d-1)^+- \delta}(b-a)+L^{(d-1)^+} \]
\end{corollary}

\subsection{Asymptotic formula on a coarse scale} These upper bounds, in particular Corollary \ref{l:lossy_count} allow us to present a simple proof of the  asymptotic formula for $\# R_{\mathbb{Z}}$ on a coarser scale,  e.g.\ $b-a = L^{\frac 43}$.  Note hat this is still better then the trivial Riemann sum scale of $b-a= O(L^2)$. 
\begin{proposition}\label{th:trivial}
Fix  $\delta>0$ sufficiently small, then if $L^{1+4\delta} \leq b-a\leq L^{2-\delta}$, we have the 
asymptotic formula
\begin{align*}
\#\left\{(p,q)\in  \mathbb Z^d\cap [0,L]^{2d }\bigm\lvert Q(p,q) \in[ a, b]\right\} &= L^{2(d-1)}(b-a)\iint\limits_{\mathbb R^{2d}}\mathds{1}_{[0,1]^{2d}} (x,y)\delta(Q(x,y))\,dxdy \\&\quad+O\left(L^{2(d-1)-\delta}(b-a)\right)\,.
\end{align*}
\end{proposition}
\begin{proof}

 First we will smooth the characteristic functions by extending the region to a slightly bigger region with a controlled error term.  This is done as follows. Let $w_L\in  C_c^\infty([-L^{\delta},L+L^{\delta}])$ be a bump function satisfying $w_L(x)=1$ for $x\in[0,L]$ and
\[\norm{w_L}_{C^N}\les L^{ -N\delta}\, .\]
Then by setting $W_L(x, y)= \prod_{i=1}^d w_L(L x_i) w_{L}(y_i)$, we have,
\[
\sum_{p,q\in \mathbb Z^d} W_L\left(\frac{p}{L},\frac{q}{L}\right) -  \mathds{1}_{[0,L]^{2d}} \left(p,q\right) = O\left(L^{2d-1+\delta}\right) \,.
\]
Moreover, if we denote by $h_L\in  C_c^\infty([a-L^{1+2\delta},b+L^{1+2\delta}])$  a bump function $h_L(x)=1$ for $x\in[a, b]$ and 
\[\norm{h_L}_{C^N}\les L^{ -N(1+2\delta)}\,.\]
then by Corollary \ref{l:lossy_count}, we have 
\begin{multline*}
\sum_{p,q\in \mathbb Z^d} W_L\left(\frac{p}{L},\frac{q}{L}\right)h_L\left(Q(p,q)\right) -  \mathds{1}_{[0,L]^{2d}} \left(p,q\right)  \mathds{1}_{[a,b]}( Q(p,q)) = \\
O\big(L^{2d-1+\delta}\big) +  O\big(L^{(2d-1+2\delta)^+}\big) = O\left(L^{2(d-1)-\delta}(b-a)\right)\,.
\end{multline*}
assuming that $b-a\geq L^{1+4\delta}$. Thus, it is sufficient to obtain the asymptotic formula for 
\[ 
\mathcal S :=\sum_{p,q\in \mathbb Z^d} W_L\left(\frac{p}{L},\frac{q}{L}\right)h_L\left(Q(p,q)\right) \,.
\]

Using  Fourier transform, we express $\mathcal S$ as 
\begin{align}
\mathcal S=\int_{-\infty}^{\infty} \widehat h_L (s) \sum_{p,q} W_L\left(\frac{p}{L},\frac{q}{L}\right) e(Q(p,q)s)\,ds := \int_{-\infty}^{\infty}\widehat h_L(s)S(s)\,ds
\end{align}

Applying Poisson summation we may rewrite $S(s)$ as
\begin{align}
S(s)=&\sum_\ell\int  W_L\left(\frac{x}{L},\frac{y}{L}\right) e( Q(x,y)s -m\cdot x -n\cdot y )\, dx\,dy\\
=&L^{2d}\sum_\ell\int W_L\left(z\right) e(L^2Q(z)s-L \ell\cdot z)\,dz \label{e:poisson_St}
\end{align}
where $z = (x,y)$, and $\ell = (m,n)$.

The term $\ell =0$  contributes the asymptotic formula
\[
L^{2d}\int W_L(z) h_L(L^2Q(z)) dz = L^{2(d-1)}(b-a)\int\limits_{\mathbb R^{2d}}\mathds{1}_{[0,1]^{2d}} (z)\delta(Q(z))\,dz +O\left(L^{2(d-1)-\delta}(b-a)\right)
\]
where we used $(b-a)<L^{2-\delta}$ in replacing $h_L(L^2Q)$) by $\delta_{dirac}(Q)$.
So it remains to show that the sum for  $\ell\neq 0$ can be treated as error.  First we estimate the sum for $s \le \frac1{L^{1+\delta}}$. In this case we write 
$ \Phi(z,\ell,s) = L^2Q(z)s-L \ell\cdot z,$ and note that since $|s| \le \frac1{L^{1+\delta}}$ and $|z|\lesssim 1$, then $\abs{\nabla_z \Phi(z,m,s)}\ge \frac {L\abs{\ell}}{2}$, where
\begin{equation}\label{eq:lower-b}
\nabla_z \Phi(z,\ell,s)= L^2\nabla Q(z)s-L\ell\,,
\end{equation}
and thus upon integrating \eqref{e:poisson_St} by parts, we obtain
\begin{align}
S(s)=
\sum_{\ell \neq 0}
L^{2d}\int \nabla_j \left(\frac{ W_L(z)}{2\pi i \nabla_j \Phi(z,\ell,s)} \right)e( \Phi(z,\ell,s) )\,dz .
\end{align}
Since each derivative of  $W_L$ contributes $L^{1-\delta}$, then  each integration by parts contributes a factor of $\frac{1}{L^{\delta}\abs \ell}$. Applying a sufficient number of integrations by parts, and using the fact that $|\widehat h_L(s)|\lesssim b-a$,  we may ensure that the contribution for $\ell \neq 0$ and $|s| \le  \frac1{L^{1+\delta}}$ is arbitrarily small. 

For $|s|\ge  \frac1{L^{1+\delta}}$  we note that 
\[
|\widehat h_L(s)|\lesssim (b-a) \frac {1}{\left(L^{1+2\delta}|s|\right)^N}, 
\]
for all $N$, and thus this term can be treated as an error. This concludes the stated result.
\end{proof}

\subsection{Bourgain's Theorem} Now we present Bourgain's proof of equidistribution.
\begin{theorem}\label{t:Bourgain}
Fix $\epsilon>0$, then for $\delta>0$ sufficiently small the following statement is true: Suppose $I_j,J_j\subset[0,L]$, $j=1,\dots,d$ for $d\geq 3$ are intervals with length satisfying
\begin{equation}
L^{1-\delta} \leq \abs{I_j},\abs{J_j} \leq L
\end{equation}
Then for $a,b$ satisfying $\abs{a},\abs{b}\leq 1$ and ${L^{-d+1+\epsilon}}<b-a<L^{-\epsilon}$ we have
\begin{equation}\label{e:equi_dist_Bourgain}
\sum_{\substack{a\leq Q(p,q) \leq b\\ p_j\in I_j, q_j\in J_j \\  p \neq q}}1=  \int_{I_1\times \dots \times I_d}\int_{J_1\times \dots\times J_d}\mathds{1}_{a\leq Q(x,y)\leq b}\,dxdy+O(L^{2(d-1-d\delta)}(b-a))\,.
\end{equation}
\end{theorem}
In order to prove Theorem \ref{t:Bourgain}, we first make a series of reductions.

\underline{\emph{Step 1: Restrict to dyadic lengths and discrete intervals $(a,b)$.}} 
We first show that it sufficient to assume dyadic lengths $L=2^{N_1}$ for $N_1\in\mathbb N$ and that $(a,b)=(N_2 L^{-d+1+\eps},(N_2+1) L^{-d+1+\eps})$, for  $N_2\in\mathbb Z$ such that $\abs{N_2}\leq2L^{d-1-\eps}$. The restriction to dyadic lengths $L=2^{N_1}$ is valid since it only has potential effect of modifying the implicit constants in the theorem. Now suppose \eqref{e:equi_dist_Bourgain} is satisfied for all such $L$ and $(a,b)$ as described above and suppose we are given another interval $(a',b')$ such that $a',b'$ satisfies $\abs{a'},\abs{b'}\leq 1$ and ${L^{-d+1+2\epsilon}}<b'-a'<L^{-\epsilon}$. Then, by assuming $\delta$ is sufficiently small (depending on $\eps$), and summing over intervals of the form $(N_2 L^{-d+1+\eps},(N_2+1) L^{-d+1+\eps})$ we obtain
\[\sum_{\substack{a'\leq Q(p,q) \leq b'\\ p_j\in I_j, q_j\in J_j \\  p \neq q}}1=  \int_{I_1\times \dots \times I_d}\int_{J_1\times \dots\times J_d}\mathds{1}_{a'\leq Q(x,y)\leq b'}\,dxdy+O(L^{2(d-1)-\eps}(b'-a'))\,.
\]
Thus, by again taking $\delta$ smaller if needed, we obtain Theorem \ref{t:Bourgain} with $\eps$ replaced by $2\eps$, i.e.\ up to a relabeling of $\eps$, we obtain Theorem \ref{t:Bourgain}.

\underline{\emph{Step 2: Ignore intervals that contribute lower order sums.}}  Set $\tilde \delta= 4d\delta$, then by Corollary \ref{r:reduc-1} we have for $\tilde\delta$ sufficiently small,
\begin{equation}\label{e:reduction1}
\sum_{R_{\mathbb{Z}}}1 = \sum_{R_{\mathbb{Z}\tilde\delta}}1 + O\left(L^{2(d-1)^+ - \tilde\delta}(b-a)) \right) + O(L^{(d-2)^+}= \sum_{R_{\mathbb{Z}\tilde \delta}}1 + O\left(L^{2(d-1)^+ - \tilde\delta}(b-a) \right)
\end{equation}
where we have used the restriction of $a-b$ and assumed $\delta $ to be sufficiently small compared to $\epsilon$.

Thus we  restrict our attention to the case where 
\begin{enumerate}[(a), leftmargin =*]
\item \label{cond:1} $\forall p_i\in E_i, \mbox{ and }  \forall q_i\in F_i, \mbox{ we have  } |p_i|> L^{1-\tilde\delta}, |q_i|>L^{1-\tilde \delta}$,
\item \label{cond:2} $\text{distance}(E_i,F_i) > L^{1-\tilde\delta}$.
\end{enumerate}
With this reduction at hand,  we divide each interval into at most $L^{3\tilde\delta}$ intervals, $E_i=\cup_{\alpha}I^\alpha_i$ and $F_i=\cup_{\alpha}J^\alpha_i$   each satisfying
\begin{enumerate}[(a),resume, leftmargin =*]
\item \label{cond:3}
$\frac12 L^{1-3\tilde\delta} \leq \abs{I^\alpha_i},\abs{J^\alpha_i} \leq L^{1-3\tilde\delta}$,
\end{enumerate}
and prove that for intervals $I^{\alpha}_i$ and $J^{\alpha}_i$, satisfying  Conditions \ref{cond:1}, \ref{cond:2}, and \ref{cond:3} we have
\begin{equation}
\sum_{\substack{a\leq Q(p,q) \leq b\\ p_j\in I_j^{\alpha}, q_j\in J^{\alpha}_j \\  p \neq q}}1=  \int_{I^{\alpha}_1\times \dots \times I^{\alpha}_d}\int_{J^{\alpha}_1\times \dots\times J^{\alpha}_d}\mathds{1}_{a\leq Q(x,y)\leq b}\,dxdy+O(L^{2(d-1)-(3d+1)\tilde\delta}(b-a))\,.
\end{equation}
Summing in $\alpha$ and using \eqref{e:reduction1} we have
\begin{align*}
\sum_{\substack{a\leq Q(p,q) \leq b\\ p_j\in I_j, q_j\in J_j \\  p \neq q}}1&= \sum_{\alpha}\left(\int_{I^{\alpha}_1\times \dots \times I^{\alpha}_d}\int_{J^{\alpha}_1\times \dots\times J^{\alpha}_d}\mathds{1}_{a\leq Q(x,y)\leq b}\,dxdy+O(L^{2(d-1-(3d+1)\tilde\delta)}(b-a))\right)\\&\quad+O\left(L^{2(d-1)^+ - 4d\delta}(b-a) \right)\\
&= \sum_{\alpha}\int_{I^{\alpha}_1\times \dots \times I^{\alpha}_d}\int_{J^{\alpha}_1\times \dots\times J^{\alpha}_d}\mathds{1}_{a\leq Q(x,y)\leq b}\,dxdy+O\left(L^{2(d-1)^+ - \tilde\delta}(b-a) \right)\,.
\end{align*}
Using that $\tilde \delta= 4d\delta$ and 
\begin{align*}
&\abs{ \int_{I_1\times \dots \times I_d}\int_{J_1\times \dots\times J_d}\mathds{1}_{a\leq Q(x,y)\leq b}\,dxdy-\sum_{\alpha}\int_{I^{\alpha}_1\times \dots \times I^{\alpha}_d}\int_{J^{\alpha}_1\times \dots\times J^{\alpha}_d}\mathds{1}_{a\leq Q(x,y)\leq b}\,dxdy}\\&\quad\quad\quad\lesssim L^{2(d-1)^+ - \tilde\delta}(b-a) \end{align*}
we conclude \eqref{e:equi_dist_Bourgain}.

Summarizing, if by abuse of notation, we drop the index $\alpha$ and replace $\tilde\delta$ with $\delta$, we have reduced the proof of Theorem \ref{t:Bourgain} to proving the following proposition.
\begin{proposition}\label{p:reduction1}
Fix $\epsilon>0$, then for $\delta>0$ sufficiently small the following statement is true: Suppose $I_j,J_j\subset[-L,L]$, $j=1,\dots,d$ for $d\geq 3$ are intervals satisfying
\begin{enumerate}
\item $\forall p_i\in I_i, \mbox{ and }  \forall q_i\in J_i, \mbox{ we have  } |p_i|> L^{1-\delta}, |q_i|>L^{1- \delta}$.
\item $\mbox{distance}(I_i,J_i) > L^{1-\delta}$.
\item $\frac12 L^{1-3\delta} \leq \abs{I_i},\abs{J_i} \leq L^{1-3\delta}$
\end{enumerate}
Then for $a,b$ satisfying $\abs{a},\abs{b}\leq 1$ and ${L^{-d+1+\epsilon}}<b-a<L^{-\epsilon}$ we have
\begin{equation}\label{e:reduction_asymp}
\sum_{\substack{a\leq Q(p,q) \leq b\\ p_j\in I_j, q_j\in J_j \\  p \neq q}}1=  \int_{I_1\times \dots \times I_d}\int_{J_1\times \dots\times J_d}\mathds{1}_{a\leq Q(x,y)\leq b}\,dxdy+O(L^{2(d-1)-(3d+1)\delta}(b-a))\,.
\end{equation}
\end{proposition}

Let us now suppose $I_j$ and $J_j$ satisfy the hypothesis of Proposition \ref{p:reduction1}.

\underline{\emph{Step 3: Transform the region of summation.}}
The sum can be written as,
\begin{equation}\label{e:main_sum}
\sum_{p_i\in I_j, q_j\in J_j}\mathds{1}_{[a,b]}(Q(p,q))=\sum_{p_i\in I_j, q_j\in J_j}\mathds{1}_{\left[0,\frac{b-a}2\right]}\left(Q(p,q)-\frac{a+b}2\right)
\end{equation} 

By writing $I_d=[u-\Delta u,u+\Delta u]$, and  $J_d=[v-\Delta v,v+\Delta v]$, and
utilizing the fact that $|u- v| > L^{1-{\delta}}$, we express  the region $R_{\mathbb{Z}}$ as,
\begin{align*}
\abs{\frac{\sum_{j=1}^{d-1}\beta_j (p_j^2-q_j^2)-\frac {b+a}{2}}{\beta_d(p_d^2-q_d^2)}+1}&\leq \frac{b-a}{2\beta_d\abs{p_d^2-q_d^2}} \; ,\\
&\le \frac{b-a}{2\beta_d\abs{u^2-v^2}}+O\left(\frac{(b-a) L^{-\delta}}{\abs{u^2-v^2}}\right)
\end{align*}
since
\[\abs{p_d^2-q_d^2-u^2+v^2}\les L(\Delta u+\Delta v)\les L^{2-3\delta}\quad\mbox{and}\quad \abs{u^2-v^2}\geq L^{2-2\delta}\,.\]
Setting $\xi= \frac {b+a}{2}$ and  $\eta =  \frac{b-a}{2}$, then by taking logarithms and Taylor expanding $\ln(x)$ around $x=1$ we obtain
\begin{equation}\label{e:real_sum}
\abs{\ln\left(\sum_{j=1}^{d-1}\beta_j (p_j^2-q_j^2)-\xi\right)-\ln\left(p_d^2-q_d^2\right)-\ln \beta_d}\le \frac{\eta}{\beta_d\abs{u^2-v^2}}  + O\left(\frac{\eta L^{-\delta}}{\abs{u^2-v^2}} \right)\,,
\end{equation}
here we assumed, without loss of generality,  $\sum_{j=1}^{d-1}\beta_j (p_j^2-q_j^2) -\xi> 0$ and $p_d^2-q_d^2 > 0$.  

\underline{\emph{Step 4: Replace the sum with an analogous sum.}}

Instead of considering the sum over the region $R_{\mathbb{Z}}$, we will consider the sum over the region $S_{\mathbb{Z}}$, defined as
\begin{equation}\label{e:reduced}
S_{\mathbb{Z}}=\left\{(p,q)\in \prod_{j=1}^d I_j \times \prod_{k=1}^dJ_k: \abs{\ln\left(\sum_{j=1}^{d-1}\beta_j (p_j^2-q_j^2)-\xi\right)-\ln\left(p_d^2-q_d^2\right)-\ln \beta_d}\leq\frac{\eta}{\beta_d\abs{u^2-v^2}}  \right\}
\end{equation}
In order to make this reduction, we need a bound on cardinality of $(p,q)$ satisfying
\[
\abs{\ln\left(\sum_{j=1}^{d-1}\beta_j (p_j^2-q_j^2)-\xi\right)-\ln\left(p_d^2-q_d^2\right)-\ln \beta_d}=\frac{\eta}{\beta_d\abs{u^2-v^2}}  + O\left(\frac{\eta L^{-\delta}}{\abs{u^2-v^2}} \right)\,,
\]
Such a bound would follow as a consequence of a version of a weaker version of Proposition \ref{p:reduction1} with the asymptotic formula \eqref{e:reduction_asymp} replaced with a sharp upper bound, i.e.,
\begin{proposition}\label{p:auxillary}
Fix $\epsilon>0$, then for $\delta>0$ sufficiently small the following statement is true: Suppose $I_j$ and $J_j$ satisfy the hypothesis of Proposition \ref{p:reduction1}, then for $a,b$ satisfying $\abs{a},\abs{b}\leq 1$ and ${L^{-d+1+\epsilon}}<b-a<L^{-\epsilon}$ we have
\begin{equation}\label{e:reduction_asymp-2}
\sum_{\substack{a\leq Q(p,q) \leq b\\ p_j\in I_j, q_j\in J_j \\  p \neq q}}1=  O(L^{2(d-1)-3d\delta}(b-a))\,.
\end{equation}
\end{proposition}
We note that for Proposition \ref{p:reduction1} compared with Proposition \ref{p:auxillary} we may require a stricter smallness criteria on $\delta$ relative to the choice of $\epsilon$. With this in mind, applying Proposition \ref{p:auxillary}, the difference in summing in $p$ and $q$ satisfying \eqref{e:real_sum} and computing the cardinality of $S_{\mathbb Z}$ is of order $O(L^{2(d-1)-(3d+1)\delta}(b-a))$ and hence can be treated as an error. We remark that such arguments will be used later to bound analogous error terms.

By the arguments above, the sum in Proposition \ref{p:auxillary} may be estimated from above by the cardinality of $S_{\mathbb Z}$ with $\eta$ replaced by $2\eta$ in the set's definition. Hence up to a factor of $2$ in the definition of $S_{\mathbb Z}$, to prove both Proposition \ref{p:auxillary} and Proposition \ref{p:reduction1}, it suffices to obtain an asymptotic formula for $S_{\mathbb Z}$.

If we set
\[
F(p,q) =  \ln\left(\sum_{j=1}^{d-1}\beta_j (p_j^2-q_j^2)-\xi\right)-\ln\left(p_d^2-q_d^2\right)-\ln \beta_d, \quad  A = \frac{\eta}{\abs{u^2-v^2}},
\]
then we can rewrite the cardinality of $S_{\mathbb Z}$ as
\begin{align*}
\sum_{S_{\mathbb{Z}}}1 &= \sum_{(p_j,q_j)\in I_i\times J_j}\mathds{1}_{[-A,A]}(F(p,q)) \,.
\end{align*}
{
For a technical reason (as will be seen in Step 7), we replace $\mathds{1}_{[-A,A]}$ by a smooth approximation. Let $\phi:\mathbb R\rightarrow \mathbb R$ be a smooth, non-negative, symmetric Friedrich mollifier, that is monotonically decreasing on $\mathbb R^+$. Setting $\phi_\eps(x)=\eps^{-1}\phi(\frac{x}{
\eps})$. Then, we have
\begin{align}
\sum_{S_{\mathbb{Z}}}1 &= \sum_{(p_j,q_j)\in I_i\times J_j}\left(\mathds{1}_{[-A,A]}*\phi_{L^{-100d}}\right)(F(p,q)) +
 \sum_{(p_j,q_j)\in I_i\times J_j}\left(\mathds{1}_{[-A,A]}-\mathds{1}_{[-A,A]}*\phi_{L^{-100d}}\right)(F(p,q))\nonumber
 \\
 &=I+II\,.
\end{align}
In an analogous argument to showing that the cardinality of $R_{\mathbb Z}$ can well approximated by the cardinality of $R_{\mathbb Z}$, we may show that sum $II$ can be estimated up to an acceptable error.}

\underline{\emph{Step 5: Expressing the sum using Fourier Transform.}}
The number $\#S_{\mathbb{Z}}$ can be expressed using the Fourier transform as follows.  Let
\[
F(p,q) =  \ln\left(\sum_{j=1}^{d-1}\beta_j (p_j^2-q_j^2)-\xi\right)-\ln\left(p_d^2-q_d^2\right)-\ln \beta_d, \quad  A = \frac{\eta}{\abs{u^2-v^2}},
\]
and write
\begin{align*}
I &= \sum_{(p_j,q_j)\in I_i\times J_j}\left(\mathds{1}_{[-A,A]}*\phi_{L^{-100d}}\right)(F(p,q)) =  \sum_{(p_j,q_j)\in I_j\times J_j} \int e^{\boldsymbol{i}F(p,q)t} \widehat{\left(\mathds{1}_{[-A,A]}*\phi_{L^{-100d}}\right)(t)}dt \\[.5em]
&= \int S_1(t)\overline{S_2(t)}e^{-it\ln \beta_d}\widehat{\mathds{1}_{[-A,A]}}\widehat{\phi_{L^{-100d}}}\,dt,
\end{align*}
where, 
\begin{align}
S_1(t)&=\sum_{\substack{p_i\in I_i, q_i\in J_i\\i=1,\dots,d-1}}\left(\sum_{j=1}^{d-1}\beta_j (p_j^2-q_j^2)+\xi\right)^{it}\\
S_2(t)&=\sum_{p_d\in I_d, q_d\in J_d}(p_d^2-q_d^2)^{it}.
\end{align}

\underline{\emph{Step 6: A scaling argument.}}  As mentioned earlier, if $A$ is large compared to $L^{-1}$, then comparing the sum over $S_{\mathbb{Z}}$ and the area of $S$ is relatively simple.  For this reason we split our sum by scaling with a factor $\frac{A}{A_0}$, where $A_0 = \frac{L^{4/3}}{\abs{u^2-v^2}}$, i.e., split the integral into two terms,
\begin{align*}
I&=\frac{A}{A_0}\int S_1(t)\overline{S_2(t)}e^{-it\ln \beta_d}\widehat{\mathds{1}_{[-A_0,A_0]}}\widehat{\phi_{L^{-100d}}} \,dt\\
&\qquad\qquad
 +\int S_1(t)\overline{S_2(t)}e^{-it\ln \beta_d}\left(\widehat{\mathds{1}_{[-A,A]}}-\frac{A}{A_0}\widehat{\mathds{1}_{[-A_0,A_0]}}\right)\widehat{\phi_{L^{-100d}}}\,dt =III+IV\,.
\end{align*}
{Ignoring the factor $\widehat{\phi_{L^{-100d}}}$}, the first integral  is counting  $p,q$ such that
\begin{equation*}
\abs{\ln\left(\sum_{j=1}^{d-1}\beta_j (p_j^2-q_j^2)-\xi\right)-\ln\left(p_d^2-q_d^2\right)-\ln \beta_d}\leq A_0\beta_d^{-1}\,.
\end{equation*}
As in Step 4, the factor $\widehat{\phi_{L^{-100d}}}$ can be ignored, up to a suitable contributing error. Then, one is reduced to counting
\begin{equation*}
\abs{\sum_{j=1}^{d}\beta_j (p_j^2-q_j^2)-\xi}\leq L^{\frac43}+O(L^{\frac43-\delta})\,.
\end{equation*}
Again, applying a similar upper/lower bounding argument to that used in Step 4 with the use of Proposition \ref{p:auxillary} replaced by the use of  Proposition \ref{th:trivial}, we obtain
\[III= \int_{I_1\times \dots \times I_d}\int_{J_1\times \dots\times J_d}\mathds{1}_{a\leq Q(x,y)\leq b}\,dxdy+O(L^{2(d-1)-(3d+1)\delta}(b-a))\,.\]
For the purpose of proving Proposition \ref{p:auxillary}, one simply observes that the first term is of order $O(L^{2(d-1)-3d\delta}(b-a))$. Thus in order to complete the proof of Proposition \ref{p:reduction1}, Proposition \ref{p:auxillary}, and by implication Theorem \ref{t:Bourgain}, it suffices to estimate $IV$.

\underline{\emph{Step 7: Replace $S_2$ with a sum involving smooth cut-offs}} 

We now replace the sum $S_2$ with a sum involving smooth cut-offs. This is a preparatory step, that will be needed for Step 10, in order to apply an argument involving the Mellin transform and Riemann zeta function estimates.

We rewrite $S_2$ in terms of the coordinates  $m=p_d-q_d$, $n=p_d+q_d$ and the set 
\[K:=\{p_d-q_d) \bigm\lvert (p_d,q_d)\in I_d\times J_d\}\,.\]
Then $S_2$ becomes
\begin{align*}
S_2&
= \sum_{p_d,q_d}\mathds{1}_{I_d}(p_d)\mathds{1}_{J_d}(q_d)(p^2_d-q^2_d)^{it}\\
&=\sum_{m,n}\mathds{1}_{[-1,1]}\left(\frac{m+n-2u}{2\Delta u}\right)\mathds{1}_{[-1,1]}\left(\frac{m-n-2v}{2\Delta v}\right)m^{it}n^{it}\\
&=\sum_{m\in K}m^{it}\sum_n \mathds{1}_{[-1,1]}\left(\frac{m+n-2u}{2\Delta u}\right)\mathds{1}_{[-1,1]}\left(\frac{m-n-2v}{2\Delta v}\right)n^{it}\,.
\end{align*}
Without loss of generality, we may assume $\Delta u\leq \Delta v$. Let us cover $K$ by disjoint intervals $M_j$ of length $ L^{1-100d{\delta}}$ and define $w_j$ to be the center of $M_j$. It is not difficult to show that that may be achieved such that $\# \{M_j\}\les L^{100d{\delta}}$ we have the following bound on the set difference
\[
\# \left(\bigcup_k M_j\right)\setminus K\lesssim L^{1-100d{\delta}}\,.
\]
Thus we have
\[\abs{S_2-\sum_j\sum_{m\in M_j}m^{it}\sum_n \mathds{1}_{[-1,1]}\left(\frac{m+n-2u}{2\Delta u}\right)\mathds{1}_{[-1,1]}\left(\frac{m-n-2v}{2\Delta v}\right)n^{it}}\lesssim L^{2-100d{\delta}}\,.\]
Using that $M_j$ is of length $ L^{1-100d{\delta}}$, we may also replace $m$ with the midpoints $w_j$ in order to obtain the estimate
\begin{multline}
\abs{\sum_n\Big(\mathds{1}_{[-1,1]}\Big(\frac{m+n-2u}{2\Delta u}\Big)\mathds{1}_{[-1,1]}\Big(\frac{m-n-2v}{2\Delta v}\Big)- \mathds{1}_{[-1,1]}\Big(\frac{w_j+n-2u}{2\Delta u}\Big)\mathds{1}_{[-1,1]}\Big(\frac{w_j-n-2v}{2\Delta v}\Big)\Big)}\\
\lesssim L^{1-100d{\delta}}\,,\notag
\end{multline}
and hence
\begin{equation*}
S_2=\sum_j\sum_{m\in M_j}m^{it}\sum_n \mathds{1}_{[-1,1]}\left(\frac{w_j+n-2u}{2\Delta u}\right)\mathds{1}_{[-1,1]}\left(\frac{w_j-n-2v}{2\Delta v}\right)n^{it}+O(L^{2-100d{\delta}})\,.
\end{equation*}
Again, up to an allowable error we may also replace the sharp cut-off cutoff functions with a smooth cut-off $\psi \equiv 1$ on $[-1+L^{-100d\delta},1-L^{-100d\delta}]$ and supported on the interval $[-1,1]$, i.e.\
\begin{equation}\label{e:S_2_decomp}
S_2=\sum_j\sum_{m\in M_j}m^{it}\underbrace{\sum_n \psi \left(\frac{w_j+n-2u}{2\Delta u}\right)\psi\left(\frac{w_j-n-2v}{2\Delta v}\right)n^{it}}_{S_j}+O(L^{2-100d{\delta}})\,.
\end{equation}
Finally, the sum in $m$ can be replaced be a sum involving a smooth cut-off, up to an allowable error
\begin{equation}\label{e:S_2_decomp2}
S_2=\underbrace{\sum_{j}\sum_{m}\psi \left(\frac{w_j-m}{L^{1-100d{\delta}}}\right)m^{it}S_j}_{\widetilde S_2}+O(L^{2-100d{\delta}})\,.
\end{equation}
We now decompose $IV$ as
\begin{align*}IV&=
\int S_1(t)\overline{\widetilde S_2(t)}e^{-it\ln \beta_d}\left(\widehat{\mathds{1}_{[-A,A]}}-\frac{A}{A_0}\widehat{\mathds{1}_{[-A_0,A_0]}}\right)\widehat{\phi_{L^{-100d}}}\,dt\\&
+\int S_1(t)\overline{(S_2(t)}-\overline{\widetilde S_2(t)})e^{-it\ln \beta_d}\left(\widehat{\mathds{1}_{[-A,A]}}-\frac{A}{A_0}\widehat{\mathds{1}_{[-A_0,A_0]}}\right)\widehat{\phi_{L^{-100d}}}\,dt= V+VI
\end{align*}
By \eqref{e:S_2_decomp2} we have
\begin{align*}
\abs{VI}\les L^{2d-100d\delta}\int\abs{\left(\widehat{\mathds{1}_{[-A,A]}}-\frac{A}{A_0}\widehat{\mathds{1}_{[-A_0,A_0]}}\right)\widehat{\phi_{L^{-100d}}}}\,dt\,.
\end{align*}
Observe that
\begin{equation}
\abs{\widehat{\mathds{1}_{[-A , A]}}(t)-\frac{A}{A_0}\widehat{\mathds{1}_{[-A_0, A_0]}}(t)}
=  A\abs{\frac{\sin(At)}{At}-\frac{\sin(A_0t)}{A_0t}}
\lesssim \min\left(AA_0^2\abs{t}^2, \frac{A}{1+A\abs{t}}\right)\, .\label{eq:platypus1}
\end{equation}
and for any $N$ we have
\begin{equation}\label{eq:platypus2}
\abs{\widehat{\phi_{L^{-100d}}}}\les \frac{1}{(1+L^{-200d}t^2)^N}\,.
\end{equation}
Thus using that $A\les (b-a)L^{-2+2\delta}$, we have
\begin{equation*}
\abs{VI}\les (b-a)L^{2(d-1)-50d\delta}\,,
\end{equation*}
which is an acceptable error.

\underline{\emph{Step 8: $V$ is an error.}}
 Now consider $V$, we aim to show that
\begin{equation}\label{e:BC_conclusion}
\abs{V}\les  L^{2(d-1)-3d{\delta}}\eta
\end{equation}
for a set of $(\beta_2, \beta_d)$ of full measure, independent of our choice of length $L=2^{N_1}$ and interval $(a,b)=(N_2 L^{-d+1+\eps},(N_2+1) L^{-d+1+\eps})$.   By Chebyshev's inequality, it suffices to show
\[
\norm{V}_{L^2_{\beta_2, \beta_d}}\lesssim L^{2(d-1)-(3d+1){\delta}}{\eta^{\frac32}}.
\] 
To see this, define
\[\Omega_{L,N_2}=\{\beta\in[1,2]^d\bigm\lvert \abs{V}>  L^{2(d-1)-3d{\delta}}\eta\}\,.\]
By Chebyshev's inequality we have
\[\abs{\Omega_{L,N_2}}\les \frac{1}{L^{4(d-1)-6d{\delta}}\eta^2}\norm{V}_{L^2_{\beta_2, \beta_d}}^2\les L^{-2\delta}\eta\,.\]
Recall that $\eta=L^{-d+1+\eps}$, then, since
\[\bigcap_{N_1\geq M,~\abs{N_2}\leq 2L^{d-1-\eps}}\abs{\Omega_{2^{N_1},N_2}}\geq 1-C\sum_{j=N}^{\infty}2^{-2j\delta}=1-C\frac{4^{\delta(1-N)}}{4^{\delta}-1}\rightarrow 1\quad\mbox{as }M\rightarrow \infty\,.\]
 we obtain \eqref{e:BC_conclusion} for a set of $(\beta_2, \beta_d)$ of full measure, where the implicit  constant depends on $(\beta_2, \beta_d)$.

Applying \eqref{eq:platypus1} and \eqref{eq:platypus2} we have
\[
\abs{\left(\widehat{\mathds{1}_{[-A , A]}}-\frac{A}{A_0}\widehat{\mathds{1}_{[-A_0, A_0]}}\right)\widehat{\phi_{L^{-100d}}}}
\lesssim \min\left(AA_0^2\abs{t}^2, \frac{A}{1+A\abs{t}}\right)\, .
\]

Averaging in $\beta_2$ and $\beta_d$, and using Plancherel's theorem for the integral in $\beta_d$, we have from the bounds $A=\eta L^{-2+2{\delta}}$ and $A_0=  L^{-\frac{2}{3}+2{\delta}}$
\begin{align*}
\norm{V}_{L^2_{\beta_2,\beta_d}}^2&\lesssim A^2\left(A_0^4\int_{\abs{t}\leq L^{{\frac{1}{100}}}} t^4\norm{S_1}_{L^2_{\beta_2}}^2\abs{\tilde S_2}^2\,dt+\int_{\abs{t}\geq L^{{\frac{1}{100}}}}\frac{1}{1+A^2t^2}\norm{S_1}_{L^2_{\beta_2}}^2\abs{\tilde S_2}^2\,dt\right)\\
&\lesssim \eta^2L^{\frac{-20}{3}+12{\delta}}\int_{\abs{t}\leq L^{{\frac{1}{100}}}} t^4\norm{S_1}_{L^2_{\beta_2}}^2\abs{\tilde S_2}^2\,dt
+\eta^2L^{-4+4{\delta}}\int_{\abs{t}\geq L^{{\frac{1}{100}}}}\frac{1}{1+\eta^2 L^{-4}t^2}\norm{S_1}_{L^2_{\beta_2}}^2\abs{\tilde S_2}^2\,dt\\
&\lesssim\underbrace{ \eta^2L^{\frac{-8}{3}+6{\delta}}\int_{\abs{t}\leq L^{{\frac{1}{100}}}} t^4\norm{S_1}_{L^2_{\beta_2}}^2\,dt}_{VII}
+\underbrace{\eta^2L^{-4+4{\delta}}\int_{\abs{t}\geq L^{{\frac{1}{100}}}}\frac{1}{1+\eta^2 L^{-4} t^2}\norm{S_1}_{L^2_{\beta_2}}^2\abs{\tilde S_2}^2\,dt}_{VIII}
\end{align*}
where we have used the trivial bound $\# \tilde S_2\leq \# I_d \#J_d \leq L^{2-6\delta}$.

\underline{\emph{Step 9: Bounding  $VII$.}}  To bound $\norm{S_1}_{L^2_{\beta_2}}$, we rewrite
\begin{align*}
\abs{S_1(t)}^{2} &=\sum_{\substack{p_i\in I_i, q_i\in J_i\\i=1,\dots,d-1}}\sum_{\substack{r_j\in I_j, s_j\in J_j\\j=1,\dots,d-1}}\left(\sum_{j=1}^{d-1}\beta_j (p_j^2-q_j^2)-\xi\right)^{it}\left(\sum_{j=1}^{d-1}\beta_j (r_j^2-s_j^2)-\xi\right)^{-it}\\
&=\sum_{\substack{p_i\in I_i, q_i\in J_i\\i=1,\dots,d-1}}\sum_{\substack{r_j\in I_j, s_j\in J_j\\j=1,\dots,d-1}}(p_1^2-q_1^2+\beta_2(p_2^2-q_2^2)+\psi_1)^{it}(r_1^2-s_1^2+\beta_2(r_2^2-s_2^2)+\psi_2)^{-it}\\
&= \sum_{\substack{p_i\in I_i, q_i\in J_i\\i=1,\dots,d-1}}\sum_{\substack{r_j\in I_j, s_j\in J_j\\j=1,\dots,d-1}} e^{it\left(\left(\ln \left(p_1^2-q_1^2+\beta_2(p_2^2-q_2^2)+\psi_1\right)-\ln \left(r^2_1-s_1^2+\beta_2(r_2^2-s_2^2)+\psi_2\right)\right)\right)}
\end{align*}
where
\[ \psi_1 :=\sum_{j=3}^{d-1}\beta_j (p_j^2-q_j^2)-\xi\quad\mbox{and}\quad \psi_2 =\sum_{j=3}^{d-1}\beta_j (r_j^2-s_j^2)-\xi\]
for $d>3$ or $ \psi_1= \psi_2=\xi$ for the case $d=3$. Setting
\[\phi:=\ln \left(p_1^2-q_1^2+\beta_2(p_2^2-q_2^2)+\psi_1\right)-\ln \left(r^2_1-s_1^2+\beta_2(r_2^2-s_2^2)+\psi_2\right) \]
we have
\begin{align*}
\abs{\partial_{\beta_2}\phi}&=\abs{\frac{p_2^2-q_2^2}{p^2_1-q_1^2+\beta_2(p_2^2-q_2^2)+\psi_1}-\frac{r_2^2-s_2^2}{r^2_1-s_1^2+\beta_2(r_2^2-s_2^2)+\psi_2}}\\ 
&\geq\abs{\frac{(p_2^2-q_2^2)(r^2_1-s_1^2+\psi_2)-(r_2^2-s_2^2)(p^2_1-q_1^2+\psi_1)}{L^4}},
\end{align*}
then for $t\leq L^4$, and by taking the sup over indices  $3\le i\le d-1$, we have
\[
\int \abs{S_1(t)}^{2}\,d\beta_2 \lesssim \sup_{\psi_1,\psi_2} L^{2(d-3)}\sum_{\substack{p_i\in I_i, q_i\in J_i\\r_i\in I_i, s_i\in J_i \\i=1,2}}
\left(1+|t|\inf_{\beta_2}|\partial_{\beta_2}\Psi|\right)^{-1}.
\]
Here we a using the trivial bound for the case $1\geq |t|\inf_{\beta_2}|\partial_{\beta_2}\Psi|$, otherwise we use Van der Corput's Lemma (see for example \cite{Stein} Chapter 8, Proposition 2). For the former case, to apply the proposition, we split the integral into regions for which $\partial_{\beta_2}\Phi$ is monotonic in $\beta_2$.

Set $(p_i-q_i)(p_i+q_i)=w_i$ and $(r_i-s_i)(r_i+s_i)=z_i$, and sum over fixed $w_i$ and $z_i$ using the divisor bound $d(k)\lesssim_\epsilon |k|^\epsilon$, we obtain
\[
\int \abs{S_1(t)}^{2}\,d\beta_2
\lesssim \sup_{\psi_1,\psi_2} L^{2(d-3)^+} \sum_{L^{2-2{\delta}}\leq\abs{w_i},\abs{z_i}\leq L^2}\left(1+\frac{\abs{t}}{L^4}\abs{w_2(z_1+\psi_2)-z_2(w_1+\psi_1)}\right)^{-1}
\]
The above sum can rearranged by summing first over the set,
\[
\mathscr{A}_{\psi}(k,w_2,z_2) =\{L^{2-2\delta}\leq\abs{w_1},\abs{z_1}\leq L^2\bigm\lvert \floor[\big]{\abs{w_2(z_1+\psi_2)-z_2(w_1+\psi_1)}}=k  \},
\]
and then over $(k,w_2,z_2)$ to obtain,
\begin{align*}
\int \abs{S_1(t)}^{2}\,d\beta_2
&\lesssim  \sup_{\psi_1,\psi_2} L^{2(d-3)^+}
\sum_{\substack{0 \le k \lesssim L^2\\ L^{2-2\delta}\leq\abs{w_2},\abs{z_2} \le L^2}}
\#\mathscr{A}_{\psi}(k,w_2,z_2)\frac{L^4}{L^4+\abs{t}k}\\
& \lesssim  \sup_{\psi_1,\psi_2} L^{2(d-3)^+}  
\sum_{L^{2-2\delta}\leq\abs{w_2},\abs{z_2} \le L^2}
\max_k\# \mathscr{A}_{\psi}(k,w_2,z_2)\left(1+\frac{L^{4^+}}{\abs{t}}\right)
\end{align*}
Now we estimate $\#\mathscr{A}_{\psi}(k,w_2,z_2)$ for a fixed $(k,w_2,z_2)$. Assume  $\mathscr{A}_{\psi}(k,w_2,z_2) \neq \emptyset$, then there exists $w_0$ and $z_0$, such that, $L^{2-2{\delta}}\leq \abs{w_0-\psi_2}\leq L^2$ and $L^{2-2\delta}\leq \abs{z_0-\psi_1}\leq L^2$ and
\[\left[\abs{w_2(z_0)-z_2(w_0)}\right]= k\,.\]
Thus 
\begin{align*}
\#\mathscr{A}_{\psi}\lesssim \#\{w_2\widetilde z_1=z_2\widetilde w_1 \bigm\lvert
\abs{\widetilde w_1-w_0},\abs{\widetilde z_1-z_0}\leq L^2\} =\# \{w_1= \frac{w_2\widetilde z_1}{z_2} \bigm\lvert
\abs{\widetilde w_1-w_0},\abs{\widetilde z_1-z_0}\leq L^2\}
\end{align*}
Since $\widetilde w_1\in \mathbb Z$  then $\#\mathscr{A}_{\psi}\lesssim 1+\frac{L^2 \gcd(w_2,z_2)}{z_2}$, and consequently
\begin{align*}
\int \abs{S_1(t)}^{2}\,d\beta_2& \lesssim L^{2(d-3)^+}\left(1+\frac{L^{4^+}}{\abs{t}}\right)\sum_{L^{2-2\delta}\leq\abs{w_2},\abs{z_2}\leq L^2}\left(1+\frac{L^2 \gcd(w_2,z_2)}{z_2}\right)\\
& \lesssim L^{2(d-3)^+}\left(1+\frac{L^{4^+}}{\abs{t}}\right)\left(\sum_{L^{2-2\delta}\leq\abs{w_2},\abs{z_2}\leq L^2}1+\sum_{\substack{L^{2-2{\delta}}\leq\abs{w_2},\abs{z_2}\leq L^2\\\gcd(w_2,z_2)\neq 1}}L^{2^+}\right)\\
& \lesssim L^{2(d-1)+{\delta}}\left(1+\frac{L^{4}}{\abs{t}}\right)\,.
\end{align*}

Hence, applying this bound to $VII$ yields
\begin{align}
VII \lesssim  \eta^2L^{\frac{-8}{3}+6{\delta}}\int_{\abs{t}\leq L^{{\frac{1}{100}}}} \abs{t}^3L^{2(d+1+{\delta})}\,dt \lesssim {\eta^2 L^{4(d-1)} L^{-2d+\frac{10}{3}+7\delta+\frac{1}{25}}\les \eta^3 L^{4(d-1)}}\,.\notag
\end{align}
where we used that $\eta=L^{-d+1+\eps}$, $\delta$ is sufficiently small and $d\geq 3$.

\underline{\emph{Step 10: Bounding $VIII$.}} Now consider $VIII$, we have
\begin{align}
VIII&\lesssim  \eta^2L^{2(d-3)+5{\delta}}\int_{\abs{t}\geq L^{\frac{1}{100}}}\frac{\abs{t}+L^4}{1+\eta^2 L^{-4}t^2}\frac{1}{\abs{t}}\abs{\tilde S_2}^2\,dt\notag\\
&\lesssim  \eta^2L^{2(d-3)+5{\delta}}\sup_{\kappa\geq L^{\frac{1}{100}}}\frac{\kappa^{1+}+L^4\kappa^{0+}}{1+\eta^2 L^{-4}\kappa^2}
\int_{\abs{t}\geq L^{\frac{1}{100}}}\frac{1}{\abs{t}^{1+}}\abs{\tilde S_2}^2\,dt\notag\\
&\lesssim  \eta^2L^{2(d-3)+5{\delta}}\left(\frac{L^{2+}}{\eta}+L^{4+}\right)\left(\sup_{k}L^{-\frac1{100}}2^{-k}\int_{L^{\frac{1}{100}}2^k}^{L^{\frac{1}{100}}2^{k+1}}\abs{\tilde S_2}^2\,dt\right)\,.\label{e:IV_est}
\end{align}

We proceed to estimate $\abs{S_j}$, defined in \eqref{e:S_2_decomp}. Defining
\[\chi(z)=\psi\left(\frac{w_j-2u}{2\Delta u} +z\right)\psi\left(\frac{w_j-2v}{2\Delta v} -\frac{z\Delta u}{\Delta v}\right)\,\]
and letting $\hat\chi$ denote Mellin transform of $\chi$, then
\[S_j^2=\left(\frac{1}{2\pi i}\int_{\Re s=2}\hat \chi(s) (2\Delta u)^s\zeta (s-it) ds\right)^2\,.\]
Shifting the contour to $\Re s=\frac{1}{2}$ we pick up the residue
\[\hat \chi(1+it)(2\Delta u)^{1+it}\,,\]
which for $\abs{t}\geq L^{\frac{1}{3}}$ is order $O(L^{-N})$ for any $N$ due to the decay of $\hat \chi$ and that $\Delta u$ and $\Delta v$ are of comparable size.
Then using $\Delta u\sim L^{1-3{\delta}}$ 
\begin{align*}
|S_j|^2&=\left|\frac{1}{2\pi}\int_{\Re s=\frac{1}{2}}\hat \chi(s) (2\Delta u)^s\zeta (s-it) ds\right|^2+O(1)\\
&\lesssim L^{1-3{\delta}}\left|\int_{\Re s=\frac{1}{2}}\hat \chi(s)\zeta (s-it) ds\right|^2+O(1)=L^{1-3{\delta}}\left|\hat \chi\left(\frac 12+i\cdot \right)* \zeta \left(\frac 12-i\cdot\right)(t)\right|^2+O(1)\,.
\end{align*}
Again, using the rapid decay of $\hat \psi$, we have
\begin{align*}
|S_j|^2
&\lesssim L^{1-3{\delta}}\left|\left(\mathds{1}_{[-L^{100d{\delta}},L^{100d{\delta}}]}(\cdot)\hat \chi\left(\frac 12+\cdot \right)\right)* \zeta \left(\frac 12-i(\cdot)\right)(t)\right|^2+O(1)\,.
\end{align*}
We now utilize following classical $L^4$ bound of the zeta function in the critical strip \cite{HB79}
\[\frac{1}{T}\int_0^T\abs{\zeta\left(\frac12 -it\right)}^4\,dt\les T^{0+}\,.\]
Using the above bound yields
\begin{align*}
&\norm{\left(\mathds{1}_{[-L^{100d{\delta}},L^{100d{\delta}}]}(\cdot)\hat \chi\left(\frac 12+i\cdot \right)\right)* \zeta \left(\frac 12-i\cdot\right)}^4_{L^4([L^{\frac1{100}}2^k,L^{\frac1{100}}2^{k+1}])}\\
&\qquad\lesssim \norm{\hat \chi}_{L^{\infty}}\norm{\zeta \left(\frac 12+i(\cdot)\right)}_{L^4([L^{\frac1{100}}2^k-L^{100d{\delta}},L^{\frac1{100}}2^{k+1}+L^{100d{\delta}}])}^4\\
&\qquad\lesssim L^{\frac1{100}+\delta}2^{k}
\end{align*}
Thus we obtain
\[
\norm{S_j}_{L^4([L^{\frac1{100}}2^k,L^{\frac1{100}}2^{k+1}])}^4\lesssim { L^{2+\frac1{100}-5\delta}2^{k}}\,. 
\]
An analogous argument also yields
\[\norm{\sum_{m}\psi \left(\frac{w_j-m}{L^{1-100d{\delta}}}\right)m^{it}}_{L^4([L^{\frac1{100}}2^k,L^{\frac1{100}}2^{k+1}])}^4\lesssim   { L^{2+\frac1{100}-5\delta}2^{k}}\,. \]
Using the decomposition \eqref{e:S_2_decomp} and the bound $\# \{M_j\}\les L^{100d{\delta}}$, we have
\begin{align*}
\norm{\tilde S_2}^2_{L^2([L^{\frac1{100}}2^k,L^{\frac1{100}}2^{k+1}])}
{\les L^{2+\frac1{100}+100\delta}2^{k}}\,.
\end{align*}
Thus, combining the above estimate on $S_2$ with \eqref{e:IV_est}, we obtain
\begin{align*}
VIII&\lesssim  \eta^2L^{2(d-3)+5{\delta}}\left(\frac{L^{2+}}{\eta}+L^{4+}\right)\left(\sup_{k}L^{-\frac1{100}}2^{-k}L^{2+\frac 1{100}+100d{\delta}}2^{k}\right)\\
&\lesssim  \eta^2L^{2(d-2)+(5+100d){\delta}}\left(\frac{L^{2+}}{\eta}+L^{4+}\right)\\
&\lesssim \eta^2L^{4(d-1)}L^{-2d+2+200d\delta}\eta^{-1}
\end{align*}
where we used that $\eta=L^{-d+1+\eps}>L^{2}$ since $d\geq 3$. Thus assuming $\delta$ to be sufficiently small, then 
\[L^{-2d+2+200d\delta}\leq L^{2d+2+\eps-3d\delta}=\eta^2L^{-3d\delta}\,,\]
 and hence
\[VIII\leq  L^{2(d-1)-3d{\delta}}\eta\,,\]
as desired.

 \subsection{Proof of Theorem \ref{t:asymptotic formula}}. First we note that the sum in Theorem  \ref{t:asymptotic formula} can be simplified  as follows,
\begin{enumerate}[1), leftmargin =*]

 
 \item Ignore all pairs $(p,q)$  such that $\abs{p_j}=\abs{q_j}$ for each $j$. The sum of such pairs such that $\abs{p},\abs{q}\leq L^{1+\delta}$ is of order O($t^2L^{d(1+\delta)})$ and hence contributes to an admissible error, where here we used the restriction $t\leq L^{d -\epsilon}$.
 
 \item We restrict the sum to the positive sector $p,q\in\mathbb Z^d_+\cap [0,L^{1+\delta}]$ for $p\neq q$. Here we are using that the subset of  $(p,q)$ such that $p_j=0$ or $q_j=0$  for some $j$ is an admissible error. This follows as a consequence of Lemma \ref{l:small-range}. To rigorously carry out such an estimate, one must split the contributions when $\abs{Q(p,q)}\leq \mu^{-1}$ and $\abs{Q(p,q)}> \mu^{-1}$. Assuming without of loss of generality that $p_1=0$, then splitting up the later part dyadically in the size of $\abs{Q(p,q)}$ and using $\abs{g(x)}\les \frac{1}{\abs{x}^2}$ one obtains the estimate
  \begin{align*}
 t^2 \sum_{ \substack{(p,q)\in\mathbb Z_+^{2d}\\ p_1=0,~p\neq q}}\abs{W\left(\frac{p}{L},\frac{q}{L}\right)g(\mu Q(p,q))}&\les t^2\left(\frac{L^{2(d-1)^+ -2\delta}}{\mu}+ L^{(d-1)^+}\right)\les tL^{2d-\delta}\, ,
  \end{align*}

 \end{enumerate}
 where $W$ was defined in \eqref{eq:W}.

  With all these reductions in mind, proving Theorem \ref{t:asymptotic formula} will follow as a consequence of the following theorem.
 
\begin{theorem}[Equidistribution]\label{th:eqdist}
Fix $\epsilon>0$ and let $\delta>0$ be sufficiently small. Then for generic $\beta\in [1,2]^d$, we have that  for any function $W \in  \mathscr{S}(\mathbb{R}^d)$,  the following holds,
\begin{align*}
\sum_{\substack{(p,q)\in  \mathbb Z_+^{2d}\\\ p\neq q}}W\left(\frac{p}{L},\frac{q}{L}\right)g(\mu Q(p,q))&=L^{2d}\int\limits_{\mathbb R_+^{2d}}W(x,y)g(L^2\mu Q(x,y))\,dxdy+O\left(\frac{L^{2(d-1)-\delta}}{\mu}\right)
\end{align*}
where $1< \mu\leq L^{d-1 -\epsilon}$. 
\end{theorem}

We remark that the above theorem is actually stronger than required: in view of the restriction on $t$ in the hypothesis of Theorem \ref{t:asymptotic formula}, we need only consider $\mu$ within the range $0<\mu\leq L^{d-2 -\epsilon}$.

Before we prove Theorem  \ref{th:eqdist}, we will need a couple of auxiliary lemmas. The following lemma is helpful in bounding errors to the asymptotic formula.
\begin{lemma}\label{l:large_K}
Let $\epsilon>0$. Given a generic quadratic from $Q(p,q)$
as defined in \eqref{qpq}, we have the following estimate
\begin{equation} \label{e:large_Q_ineq}
\sum_{\substack{(p,q)\in \mathbb Z^{2d}\cap[0,L]^{2d}\\ p\neq q, \abs{Q(p,q)}\geq a}}\frac{1}{Q(p,q)^2}\lesssim 
\frac{L^{(2d-2)+}}{a}\,.
\end{equation}
for $a\geq L^{-d+\epsilon}$
\end{lemma}
\begin{proof}
We begin by dyadically subdividing the interval $[a,CL^2]$, for some large $C$, we define
\begin{gather*}
R_{\mathbb{Z}^{(m)}}\overset{def}{=}\{(p,q)\in\mathbb{Z}^{2d}\cap[0,L]^{2d} \bigm\lvert \abs{Q(p,q)}\in [2^m,2^{m+1}], p\ne q\}, \\
 m_{\min}=\lfloor \log_2 a\rfloor,\quad\mbox{and}\quad m_{\max}=\lceil \log_2 CL^2\rceil\,.
\end{gather*}
Applying Lemma \ref{l:lossy_count-l} yields
\begin{align*}
\sum_{\substack{(p,q)\in \mathbb Z^{2d}\cap[0,L]^{2d}\\ p\neq q, \abs{Q(p,q)}\geq a}}\frac{1}{Q(p,q)^2}&\lesssim
\sum_{m=m_{\min}}^{m_{\max}}2^{-2m}\# R_{\mathbb{Z}^{(m)}}\\
&\lesssim
\sum_{m=m_{\min}}^{m_{\max}}2^{-2m}L^{(2d-2)+}2^{m}\\
&\lesssim
\sum_{m=m_{\min}}^{m_{\max}}\frac{L^{(2d-2)+}}{a}\,.
\end{align*}
\end{proof}

The following lemma will be useful localizing the sum in Theorem \ref{th:eqdist}.
\begin{lemma}\label{e:count_with_g}
Fix $\epsilon>0$, then for $\delta>0$ sufficiently small  the following statement is true: Suppose $I_j,J_j\subset \mathbb [0,L]$ for $j=1,\dots,n$ are intervals with length satisfying
\begin{equation}
L^{1-\delta} \leq \abs{I_j},\abs{J_j}\,
\end{equation}
and define
\[S_{(I,J)}\overset{def}{=}\{(p,q)\in\mathbb{Z}^{2d} \bigm\lvert p_j\in I_j, ~q_j\in J_j,~p\ne q\}\,.\]
Then for $\mu$ satisfying $L^{\epsilon}\leq \mu\leq L^{d-\epsilon}$ we have
\begin{equation}
\sum_{ (p,q)\in S_{(I,J)}}g(\mu Q(p,q))=\int_{I_1\times \dots \times I_d}\int_{J_1\times \dots\times J_d}g(\mu Q(x,y))\,dxdy+O\left(\frac{L^{2(d-1-d\delta)}}{\mu}\right)\,.
\end{equation}
\end{lemma}
\begin{proof}
First note that by Lemma \ref{l:large_K}
\begin{align*}
\abs{\sum_{(p,q)\in S_{(I,J)}}g(\mu Q(p,q))-\sum_{ \substack{(p,q)\in S_{(I,J)}\\ \abs{Q(p,q)}\leq \mu^{-1}L^{4d\delta}}}g(\mu Q(p,q))}&\lesssim \frac{L^{2(d-1)+(1-4d)\delta}}{\mu}\\
&\lesssim\frac{L^{2(d-1-d\delta)}}{\mu}\,.
\end{align*}

Define the sum $A(y)$ and the integral $\tilde A$ as follows
\[A(y)=\sum_{\substack{(p,q)\in S_{(I,J)}\\\abs{ Q(p,q)} \leq y}} 1\quad\mbox{and}\quad \tilde A(y)=\int_{I_1\times \dots\times I_d}\int_{J_1\times \dots \times J_d} \mathds{1}_{[-y,y]} (Q(u,v))\,dudv\]

Then in the sense of distributions
\begin{align*}
\sum_{ \substack{(p,q)\in S_{(I,J)}\\ \abs{Q(p,q)}\leq \mu^{-1}L^{4d\delta}}} g(\mu Q(p,q))&=\int_{0}^{\mu^{-1}L^{4d\delta}}g(\mu y)A'(y)\,dy\\
&=-\mu \int_{0}^{\mu^{-1}L^{4d\delta}}g'(\mu y)A(y)\,dy+g(L^{4d\delta})A(\mu^{-1}L^{4d\delta})\\
&=-\mu \int_{0}^{\mu^{-1}L^{4d\delta}}g'(\mu y)A(y)\,dy+O\left(\frac{L^{2(d-2)-4d\delta}}{\mu}\right)
\end{align*}
where in the last inequality we applied Lemma \ref{l:lossy_count-l} and the bound $L^{\epsilon}\leq \mu\leq L^{d-\epsilon}$. Writing $A=\tilde A+(A-\tilde A)$, we have
\[
\sum_{ \substack{(p,q)\in S_{(I,J)}\\ \abs{Q(p,q)}\leq \mu^{-1}L^{4d\delta}}} 
\hskip -4mm  g(\mu Q(p,q)) =-\mu\hskip -2mm  \int\limits_{0}^{\mu^{-1}L^{4d\delta}}\hskip -3mm g'(\mu y)\tilde A(y)\,dy+\mu \hskip -2mm \int\limits_{0}^{\mu^{-1}L^{4d\delta}} \hskip -3mm g'(\mu y)(A(y)-\tilde A(y))\,dy+O\left(\frac{L^{2(d-2)-4d\delta}}{\mu}\right)
\]
By Theorem \ref{t:Bourgain} (by choosing $\delta$ smaller than the $\delta$ used in the theorem) it follows that assuming $y\geq L^{-d+\epsilon}$ then
\[\abs{A(y)-\tilde A(y)}\lesssim L^{2(d-1)-10d\delta}y\,.\]
For  $y\leq L^{-d+\epsilon}$ by the trivial bound
\[\abs{A(y)-\tilde A(y)}\lesssim A(y)+\tilde A(y)\lesssim L^{d-2+\epsilon+\delta}\,.\]
Using the trivial bound $g'(z)\lesssim 1$ we have
\begin{align*}
&\mu \int_{0}^{\mu^{-1}L^{4d\delta}}\abs{g'(\mu y)(A(y)-\tilde A(y))}\,dy\\&\qquad
\lesssim\mu \int_{0}^{{L^{-d+\epsilon}}}\abs{g'(\mu y)(A(y)-\tilde A(y))}dy+ \int_{{L^{-d+\epsilon}}}^{\mu^{-1}L^{4d\delta}}\abs{g'(\mu y)(A(y)-\tilde A(y))}dy\\
&\qquad
\lesssim\mu L^{-2+2\epsilon+\delta} + \mu L^{2(d-1)-10d\delta}\int_{{L^{-d+\epsilon}}}^{\mu^{-1}L^{4d\delta}}y dy\\
&\qquad\lesssim \mu^{-1}L^{2(d-1)-\epsilon+\delta} +   \mu^{-1} L^{2(d-2)-2d\delta}
\end{align*}
where in the last inequality we used $\mu\leq L^{d-\epsilon}$. Choosing $\delta$ sufficiently small in relation to $\epsilon$, this constitutes an allowable error.
The proof concludes by noting that by integration by parts
\begin{align*}
-\mu \int\limits_{0}^{\mu^{-1}L^{4d\delta}}g'(\mu y)\tilde A(y)\,dy&=\iint\limits_{\substack{I_1\times \dots \times I_d\\J_1\times \dots \times J_d} }\mathds{1}_{[-\mu^{-1}L^{4d\delta},\mu^{-1}L^{4d\delta}]}(Q(x,y))\left(g(\mu Q(x,y))+g(L^{4d\delta})\right)\,dxdy\\
 \end{align*}
 \end{proof}

\begin{proof}[Proof of Theorem \ref{th:eqdist}]
We first note that by symmetry, it is sufficient to restrict ourselves to the positive sector $p,q\in\mathbb Z^d_+$. Note that Lemma \ref{l:small-range} implies the subset of $(p,q)$ such that $p_j=0$ or $q_j=0$ may be treated as an admissible error. Thus, it suffices to show
\begin{align*}
\sum_{\substack{(p,q)\in  \mathbb Z_+^{2d}\\ p\neq q}}W\left(\frac{p}{L},\frac{q}{L}\right)g(\mu Q(p,q))&=L^{2d}\int\limits_{\mathbb R_+^{2d}}W(x,y)g(L^2\mu Q(x,y))\,dxdy+O\left(\frac{L^{2(d-1)-\delta}}{\mu}\right)\\
&= \frac{L^{2(d-1)}}{\mu}\int\limits_{\mathbb R_+^{2d}}W(x,y)\delta(Q(x,y))\,dxdy +O\left(\frac{L^{2(d-1)-\delta}}{\mu}\right)
\end{align*}

Divide $[0,L^{\delta}]^d\times [0,L^{\delta}]^d$ into products of cubes $M_j,N_k\subset \mathbb R_+^d$ of length $L^{-10d\delta}$. Define $W_{j,k}$ to be the average of $W$ over $M_j\times N_k$:
\[W_{j,k}:=\fint_{M_j}\fint_{N_k}W(x,y)\,dxdy\]
Note that if $(x,y)\in M_j\times N_k$ then from the smoothness of $W$
\[\abs{W(x,y)-W_{j,k}}\lesssim L^{-10d\delta}\,.\]
Hence using Lemma \ref{l:large_K} 
\begin{align*}
\abs{\sum_{\substack{(p,q)\in  \mathbb Z_+^{2d}\\ p\neq q}}W\left(\frac{p}{L},\frac{q}{L}\right)g(\mu Q(p,q))-\sum_{j,k}\sum_{\substack{p\in LM_j, q\in LN_k\\ p\neq q}}W_{j,k}g(\mu Q(p,q))}&\lesssim \frac{L^{2(d-1)(1+\delta)+\delta-10d\delta}}{\mu}\\&\lesssim \frac{L^{2(d-1)-\delta}}{\mu}
\end{align*}
Applying Lemma \ref{e:count_with_g} (taking $\delta$ to be sufficiently small) we obtain
\begin{align*}
\sum_{j,k}\sum_{\substack{p\in LM_j, q\in LN_k\\ p\neq q}}W_{j,k}g(\mu Q(p,q))&=\sum_{j,k}\int_{LM_j}\int_{LN_k}W_{j,k}g(\mu Q(x,y))\,dxdy+O\left( \frac{L^{2(d-1)-\delta}}{\mu}\right)\\
&=L^{2d}\int\limits_{\mathbb R_+^{2d}}W(p,q)g(L^2 \mu Q(x,y))\,dxdy+O\left( \frac{L^{2(d-1)-\delta}}{\mu}\right)\,.
\end{align*}
\end{proof}

\noindent {\bf Acknowledgments.} TB was supported by the National Science Foundation grant DMS-1820764.
PG was supported by the National Science Foundation  grant  DMS-1301380.
JS was  supported by the National Science Foundation  grant DMS-1363013.
ZH was supported by National Science Foundation grant DMS-1852749, NSF CAREER award DMS-1654692, and a Sloan Fellowship.


\end{document}